\documentclass[11pt]{article}
\usepackage[T1]{fontenc}
\usepackage{mathrsfs}
\usepackage[centertags]{amsmath}
\usepackage{amsfonts}
\usepackage{amssymb}
\usepackage{amsthm}
\usepackage{cases}
\usepackage{epsfig}
\usepackage{graphicx}
\usepackage{color}
\usepackage{multirow}
\usepackage{tabularx}
\usepackage{booktabs}
\usepackage{threeparttable}
\usepackage{graphicx}
\usepackage{placeins}
\usepackage{array}
\usepackage{caption,float}
\usepackage{makecell}
\usepackage{microtype}         
\usepackage{bm}                
\usepackage{enumerate}         
\usepackage{tikz}
\usepackage{pgfplots}
\usepackage{subcaption}
\allowdisplaybreaks[4]
\usepackage{MnSymbol}
\usepackage{bbm}
\usepackage[round,authoryear]{natbib}
\usepackage{csquotes}
\usepackage{algorithm}
\usepackage{algpseudocode}
\usepackage{authblk}
\usepackage[a4paper]{geometry} 

\makeatletter

\usepackage[unicode,hidelinks]{hyperref}

\usepackage[capitalize]{cleveref}

\usepackage{caption}
\patchcmd{\NAT@citex}
  {\@citea\NAT@hyper@{%
     \NAT@nmfmt{\NAT@nm}%
     \hyper@natlinkbreak{\NAT@aysep\NAT@spacechar}{\@citeb\@extra@b@citeb}%
     \NAT@date}}
  {\@citea\NAT@nmfmt{\NAT@nm}%
   \NAT@aysep\NAT@spacechar\NAT@hyper@{\NAT@date}}{}{}

\usepackage{thmtools}
\declaretheorem[name=Theorem,numberwithin=section]{theorem}
\declaretheorem[name=Lemma,sibling=theorem]{lemma}

\declaretheorem[name=Proposition,sibling=theorem]{proposition}

\declaretheorem[name=Assumption,sibling=theorem,style=definition]{assumption}

\newtheoremstyle{remboldstyle}
  {}{}{}{}{\itshape}{.}{.5em}{{\thmname{#1 }}{\thmnumber{#2}}{\thmnote{ (#3)}}}

\crefname{theorem}{Theorem}{Theorems}
\Crefname{theorem}{Theorem}{Theorems}
\crefname{lemma}{Lemma}{Lemmas}
\Crefname{lemma}{Lemma}{Lemmas}
\crefname{proposition}{Proposition}{Propositions}
\Crefname{proposition}{Proposition}{Propositions}
\crefname{corollary}{Corollary}{Corollaries}
\Crefname{corollary}{Corollary}{Corollaries}
\crefname{definition}{Definition}{Definitions}
\Crefname{definition}{Definition}{Definitions}
\crefname{condition}{Condition}{Conditions}
\Crefname{condition}{Condition}{Conditions}
\crefname{assumption}{Assumption}{Assumptions}
\Crefname{assumption}{Assumption}{Assumptions}
\crefname{conclusion}{Conclusion}{Conclusions}
\Crefname{conclusion}{Conclusion}{Conclusions}
\crefname{remark}{Remark}{Remarks}
\Crefname{remark}{Remark}{Remarks}
\crefname{example}{Example}{Examples}
\Crefname{example}{Example}{Examples}
\crefname{figure}{Figure}{Figures}
\Crefname{figure}{Figure}{Figures}

\numberwithin{equation}{section}

\newcommand{\pr}{\operatorname{\mathbb{P}}}

\renewcommand{\d}{\operatorname{d}}
\newcommand{\E}{\operatorname{\mathbb{E}}}
\newcommand{\Var}{\operatorname{var}}
\newcommand{\cov}{\operatorname{Cov}}

\newcommand{\cdt}{\mid}

\newcommand{\I}{\operatorname{\mathbbm{1}}}

\newcommand{\e}{\mathrm{e}}
\newcommand{\EMP}{\mathrm{EMP}}

\newcommand{\bY}{\boldsymbol{Y}}
\newcommand{\bZ}{\boldsymbol{Z}}

\newcommand{\F}{\operatorname{\mathcal{F}}}
\newcommand{\cbr}[1]{\left\{ #1 \right\}}
\newcommand{\rbr}[1]{\left( #1 \right)}

\newcommand{\HR}{H\"usler--Reiss}

\begin{document}
\title{Estimating the H\"{u}sler--Reiss variogram matrix by clipped moments}

\author[1]{Shuang Hu\thanks{Corresponding author: hushuang@cqupt.edu.cn}}
\author[2,3]{Johan Segers\thanks{E-mail: jjjsegers@kuleuven.be}}

\affil[1]{School of Mathematics and Statistics/Key Laboratory of Intelligent Analysis and Decision on Complex Systems, Chongqing University of Posts and Telecommunications, 400715 Chongqing, China}
\affil[2]{Department of Mathematics, KU Leuven, B3001 Heverlee, Belgium}
\affil[3]{LIDAM/ISBA, UCLouvain, B1348 Louvain-la-Neuve, Belgium}

\maketitle

\begin{abstract}
In multivariate extreme value analysis, the tail dependence between some of the risk variables at hand may be weak, even when other variables do tend to become large simultaneously. Weak tail dependence may induce a substantial bias in estimation procedures based on the limiting multivariate (generalized) Pareto distribution of excesses over high thresholds. We consider a \HR{} multivariate generalized Pareto model and, motivated by this issue, propose first- and second-order moment estimators of its variogram matrix constructed from a lower-tail-clipped version of the underlying random vector. The asymptotic normality of the proposed estimators is established. We demonstrate by simulation studies that they have lower bias than the empirical variogram estimator in certain cases, particularly when the dependence between components is weak. The estimators are applied to flood discharge data from the Danube river basin and the US flight delay data, showing that the tail dependence structure implied by the fitted model based on the first-order clipped moment estimator aligns more closely with the empirical tail dependence of the data than that based on the empirical variogram estimator.
 \smallskip

\smallskip

\end{abstract}

\section{Introduction}
In many areas such as finance, environmental science, and engineering, extreme events can have disproportionately large impacts. Accurate modeling of such rare occurrences is essential for effective risk assessment and decision making. While univariate extreme value theory has been well-developed, practical problems often involve multiple interconnected variables, necessitating a multivariate framework that can capture both marginal extremes and their joint dependence structure.

Multivariate extreme value theory addresses this need by characterizing the limiting distribution of componentwise maxima through multivariate max-stable distributions or multivariate generalized Pareto (MGP) distributions (MGPD).
To enable flexible and interpretable modeling of extremal dependence, various parametric families have been introduced in the literature (see \citealp[Chapter 9]{BGST2004}; \citealp{DPR2012}; \citealp{HDC2017}; and references therein). Among these models, the \HR{} distribution parameterized by the so-called variogram matrix stands out for its analytical tractability and practical relevance. It arises as the limiting distribution of componentwise maxima in Gaussian triangular arrays~\citep{HR1989}, and also appears as the finite-dimensional marginal distribution of the Brown--Resnick process~\citep{EMKS2015}.

Owing to its modeling flexibility, the \HR{} distribution has garnered significant attention in the literature. Its theoretical aspects, such as the tail dependence properties and convergence results where it appears as a limit, have been extensively studied; see, for example, \citet{FR2010}, \citet{LP2014}, and \citet{HPW2016}. Beyond these theoretical developments, the widespread use of this model in practice highlights the importance of reliable statistical inference. Various estimation methods developed for multivariate max-stable distributions or multivariate (generalized) Pareto distributions can be adapted to the estimation of the variogram matrix, such as the composite likelihood estimation method in~\citet{PRS2010}, the M-estimator in~\citet{EKS2012} and the weighted least squares estimator in~\citet{EKS2018}. Specifically for the \HR{} family, \citet{HD2013} investigated the composite likelihood estimation approach in the context of the Brown--Resnick process, while \citet{EMKS2015} proposed several new estimators for the \HR{} distribution based on the peaks-over-threshold approach. Of particular relevance to our work, \citet{EV20} proposed an empirical estimator of the variogram matrix for a general multivariate Pareto distribution, which can be directly applied to the \HR{} model as a special case. Due to its simplicity and broad applicability, this estimator has been widely used in applications.

In practice, extremes often involve only a subset of variables becoming large while others remain moderate. Such configurations, referred to as extreme clusters \citep{CCFS2020} or extreme directions \citep{MKS2024}, motivate the need for flexible models that can capture both strong and weak tail dependence across different subsets of components. The \HR{} model is a natural choice in this context, as its variogram parameter matrix provides a parsimonious description of the entire dependence structure, including partial extremal behavior. However, estimation procedures based on empirical variograms, the widely adopted method in practice, can be severely biased when the dependence between certain components is weak. This bias arises because weak dependence increases the likelihood that some components take relatively small values, even when others are extreme. As a result, observations exceeding the threshold in only a subset of components often include non-extreme values in the remaining components, which weakens the estimation of the dependence structure implied by the limiting MGPD. Consequently, empirical variogram estimators, which rely on the assumption that the asymptotic model provides a good approximation in the tail region, can yield inaccurate estimates of the elements of the variogram matrix.

This observation motivates us to employ a clipping method in the sense that the standardized MGPD observations $\bm{X}$ below a threshold $c$ are modified through a deterministic transformation $\bm{Y}=\bm{X}\vee c=(X_i \vee c , i \in V)$, to mitigate the bias of the empirical variogram estimator introduced by non-extreme components. Based on this clipping, we propose two moment-type estimators for the elements of the variogram matrix and establish the asymptotic normality of these estimators. Simulation studies and empirical analyses indicate that the proposed estimators significantly mitigate bias in the \HR{} model when the dependence among certain components is weak. Furthermore, we apply the \HR{} model to extreme flood discharges from the Danube River basin and the US flight delay data, and estimate the parameters using the proposed estimators. The resulting fitted model based on the first-order clipped moment estimator better captures the empirical tail dependence structure than that based on the empirical variogram estimator.

In the remainder of this paper, we first review the necessary background in \cref{sec:back}, and then describe the construction method of the moment estimators in \cref{sec:ConstrOfEsts}. The main theoretical results are presented in \cref{sec:Results}.
A simulation study is conducted in \cref{sec:simstudy}, followed by a case study on daily discharges from the Danube and the US flight delay data in \cref{sec:application}. \Cref{sec:conclusion} concludes. All proofs are deferred to \cref{sec:Appendix}. Additional simulation results are provided in \cref{sec:gamma2}.

\paragraph{Notation.}
Throughout the paper, bold symbols refer to multivariate quantities. Comparison and arithmetic operators between two vectors or between a vector and a real number are understood componentwise. For instance, for any $d$-dimensional vectors $\bm{x}$ and $\bm{y}$, we write ``$\bm{x}\le \bm{y}$'' if ``$x_i\le y_i$'' for all $1\le i\le d$. For a real number $T$, ``$\bm{x}\le T$'' means ``$x_i \le T$'' for all $i \in \{1,\ldots,d\}$. Furthermore, ``$\bm{x}\vee \bm{y}$'' denotes the vector whose $i$-th component is $\max(x_i,y_i)$. The relation $X \overset{d}=Y$ indicates that the random variables $X$ and $Y$ are identically distributed. For sets $U$ and $V$, the Cartesian product is denoted by $U\times V$. Vectors are treated as column vectors by default.

\section{Background}
\label{sec:back}

\subsection{Multivariate generalized Pareto distributions}
The MGPD models the extremal dependence of a random vector through threshold exceedances. To make abstraction of the univariate marginal distributions and concentrate on the extremal dependence, we introduce the MGPD on a canonical exponential scale. The MGPD with other margins can be obtained via monotone marginal transformations.

Let $d>1$ and $V=\{1,\ldots,d\}$. Consider a random vector $\tilde{\bm{X}}=(\tilde{X}_{i},\,i\in V)$ with standard exponential margins, i.e.,
\[
\pr(\tilde{X}_i\le x) = 1-\exp(-x), \quad x\ge 0.
\]
Define $\mathcal{L}=\{\bm{y}\in [-\infty,\infty)^d: \bm{y} \nleqslant  \bm{0}\}$. A random vector $\bm{Y}=(Y_v,v\in V)$ is said to follow a \emph{standard multivariate generalized Pareto distribution} if it can arise as the limit
\begin{align}
\label{eq:MPD}
\pr(\bm{Y}\le y)=\lim_{u\to\infty} \pr(\tilde{\bm{X}}-u \le \bm{y} \cdt \tilde{\bm{X}} \nleqslant u), \quad \bm{y} \in \mathcal{L}.
\end{align}
In this case, the random vector $\tilde{\bm{X}}$ is said to belong to the domain of attraction of $\bm{Y}$. General MGPDs with generalized Pareto margins can be obtained through componentwise transformations
\[
\sigma_i \frac{\exp(\xi_i Y_i)-1}{\xi_i}, \quad i \in V,
\]
where $\sigma_i \in (0,\infty)$ and $\xi_i \in \mathbb{R}$. We refer to~\citet{R2008}, \citet{BGST2004} and \citet{naveau2024multivariate} for more details of the MGPD and the multivariate extreme value theory.
Since we are mainly interested in the dependence structure, we will focus on the standard MGPD in the following.

For a standard MGPD on the exponential scale, it admits a useful stochastic representation. For $m\in V$, let $\bm{Y}^{(m)}$ denote the conditional distribution of the exceedance vector given that the $m$-th component exceeds the threshold zero, i.e., $\bm{Y}^{(m)}=\bm{Y} \cdt Y_m>0$, with support $\mathcal{L}^m=\{\bm y\in [-\infty,\infty)^d \, | \, y_m >0\}$.
Then,
\begin{equation}
\label{eq:StochRepr}
    \bY^{(m)} \stackrel{d}{=} E + \bZ^{(m)} = \rbr{E + Z^{(m)}_1, \ldots, E + Z^{(m)}_d},
\end{equation}
where $\bm{Z}^{(m)}=(Z_v^{(m)},v\in V)$ is a $d$-dimensional random vector with $Z^{(m)}_m=0$ almost surely and where $E$ is a standard exponential random variable, with distribution function $\pr(E\le x)=1-\exp(-x)$, $x\ge 0$, which is independent of $\bZ^{(m)}$ (see \citet[Eq.~(6)]{EV20}; \citet[Definition 2.1]{naveau2024multivariate}). The representation separates the extremal magnitude and dependence structure. The common exponential component $E$ describes the size of the extreme event, while $\bm{Z}^{(m)}$ determines the dependence among components in the tail region.

\subsection{\HR{} MGPD}

A prominent and tractable subclass of MGPD is the \HR{} MGPD. In this model, the dependence structure is fully characterized  by a matrix $\bm{\Gamma}=(\gamma_{ij})_{i,i \in V}$ called \emph{variogram matrix}, which is a $d\times d$ dimensional symmetric, conditionally negative definite matrix with elements $\gamma_{ij}\in [0,\infty)$ satisfying $\gamma_{ii}=0$ for $i \in V$. That is, $\bm{\Gamma}\in \mathcal{M}$ with
\begin{equation}
\label{eq:NegDefMatrix}
\mathcal{M}=\left\{\bm{\Gamma}\in [0,\infty)^{d\times d}: \bm{\Gamma}=\bm{\Gamma}^{\top}, \, \mathrm{diag}(\bm{\Gamma})=0, \ \text{and} \ \bm{u}^{\top} \bm{\Gamma} \bm{u} <0, \, \forall \bm{u} \neq \bm{0}\in \mathbb{R}^{d}, \, \sum_{i\in V} u_i =\bm{0} \right\}.
\end{equation}
A \HR{} MGPD can be defined through the stochastic representation of the standard MGPD. Specifically, a random vector $\bm{Y}$ is said to follow a \emph{\HR{} MGPD} with variogram matrix $\bm{\Gamma}=(\gamma_{ij}) \in \mathcal{M}$ if, for any $m\in V$, the conditional excess vector admits the representation \eqref{eq:StochRepr}, i.e.,
\[
\bm{Y}^{(m)}
\stackrel{d}{=}
E+\bm{Z}^{(m)},
\]
where $E$ is a standard exponential random variable independent of $\bm{Z}^{(m)}$, and $\bm{Z}^{(m)}$ is a (possibly degenerate) Gaussian random vector satisfying
\[
Z_m^{(m)}=0,\quad \mathrm{a.s.}
\]
with mean vector $\bm{\mu}^{(m)}$ and covariance matrix $\bm{\Sigma}^{(m)}$ given by
\[
\mu_i^{(m)}=-\frac12\gamma_{im}, \  \  \
\Sigma_{ij}^{(m)}
=
\frac12
(\gamma_{im}+\gamma_{jm}-\gamma_{ij}).
\]
Since $\bm{\Gamma}$ is symmetric, i.e.,
\[
\gamma_{ij}=\gamma_{ji}, \quad i,j\in V,
\]
the marginal distributions of $Z_i^{(j)}$ and $Z_j^{(i)}$ are identical, with
\[
Z_j^{(i)} \overset{d}= Z_i^{(j)}
\sim
N(-\gamma_{ij}/2,\gamma_{ij}).
\]

\subsection{Empirical variogram estimator}

The variogram matrix is not only a fundamental parameter for the \HR{} MGPD, indeed, it has been extended to general MGPDs by \citet{EV20}.
To facilitate statistical inference, they proposed the empirical variogram estimator based on data in the domain of attraction of the MGPD. Although it was originally defined using the multivariate Pareto representation, we reformulate it on the exponential scale here to maintain consistency with the notation and framework adopted in this paper.

Consider a random vector $\bY$ following a standard MGPD on the exponential scale. Let $\bm X =(X_i,i\in V)$ denote a $d-$variate random vector with continuous marginal functions $F_i$, $i\in V$. Denote $F(\bm{X})=(F_i(X_i),\, i\in V)$. Assume that, after marginal transformation
\begin{align}
\label{eq:marginal_Transform}
\bm{X} \mapsto \tilde{\bm{X}} = -\log(1-F(\bm{X})),
\end{align}
the standardized random vector $\tilde{\bm{X}}$, with standard exponential margins, belongs to the domain of attraction of $\bY$. Suppose $\bm{X}_{t}=(X_{t1},\ldots,X_{td})$, $t=1,\ldots,n,$ are independent copies of $\bm{X}$.

Let $k=k(n)$ be an integer sequence satisfying $k\to\infty$ and $k/n\to q\in [0,1]$ as $n\to\infty$. Denote the logarithmic transformation to the empirical Pareto exceedance by
\begin{equation}
\label{eq:emp_Y}
\hat{Y}_{ti} = \log \left\{\frac{k}{n(1 - \hat{F}_i(X_{ti}))}\right\},\  i\in V, \ t=1,\ldots,n,
\end{equation}
where
\begin{equation*}
    \hat{F}_{i}(x) = \frac{1}{n+1}\sum_{t=1}^{n} \I\!\left\{X_{ti}\le x\right\}
\end{equation*}
is the (adjusted) empirical distribution function of the $i$-th variable.
The \emph{empirical variogram estimator rooted at $m$}, denoted by $\hat{\bm{\Gamma}}^{(m)}=\left(\hat{\gamma}_{n,ij}^{(m)}\right)$,
proposed by~\citet{EV20} is given as
\begin{align}
\label{eq:EmpVario}
\hat{\gamma}_{n,ij}^{(m)}=\widehat{\Var}\left\{(\hat{Y}_{ti}-\hat{Y}_{tj}): \hat{F}_{m}(X_{m})\ge 1-k/n\right\}, \, i,j,m\in V.
\end{align}
where $\widehat{\Var}$ denotes the sample variance of $\hat{Y}_{ti}-\hat{Y}_{tj}$ for those $t \in \cbr{1,\ldots,n}$ such that $\hat{F}_m(X_{tm}) \ge 1 - k/n$.

For a \HR{} MGPD with variogram matrix $\bm{\Gamma}$,
it turns out that $\bm{\Gamma}$ can be estimated by any  $\hat{\bm{\Gamma}}^{(m)}$, $m\in V$,
so does the averaged empirical estimator introduced also by \citet{EV20} (hereafter referred to as the empirical variogram estimator)
\begin{equation}
\label{eq:MeanEmpVario}
\hat{\gamma}_{n,ij}^{\mathrm{EMP}} = \frac{1}{d} \sum_{m \in V} \hat{\gamma}_{n,ij}^{(m)},
\end{equation}
with $\hat{\gamma}_{n,ij}^{(m)}$ given by~\eqref{eq:EmpVario}.

Owing to its computational simplicity and ease of implementation, the empirical estimator has been widely adopted in practice. However, in heterogeneous dependence structures, it may suffer from bias due to the inclusion of non-extreme observations. Specifically, for fixed $k \in \{1,\ldots,n\}$ and $(i, j) \in V \times V$ with $i \ne j$, the empirical variogram is the average of $\hat{\gamma}_{n,ij}^{(m)}$ over $m \in V$, where $\hat{\gamma}_{n,ij}^{(m)}$ is based on the points $(\hat{Y}_{ti}, \hat{Y}_{tj})$ for all $t \in \{1,\ldots,n\}$ such that $\hat{Y}_{tm} \ge 0$. For $m \notin \{i,j\}$, if variable $m$ is only weakly dependent with variables $i$ and/or $j$, then the empirical variogram also uses data points for which the observations of variables $i$ and $j$ are not large at all. Although the bias issue arises for general MGPDs, in this work we concentrate on the \HR{} MGPD case. To mitigate this bias, we introduce a moment method that employs a clipping strategy on the lower tail of the \HR{} MGPD vector, as described in the next section.

\section{Moment-based variogram estimators}
\label{sec:ConstrOfEsts}

\subsection{Moments and sample moments of clipped MGP random vectors}

To eliminate the bias caused by the non-extreme values of the sample in the estimation of the variogram matrix, we propose an approach based on the clipped MGP random vector.

For a $d$-variate random vector $\bY$ with an arbitrary MGPD, we fix a constant $c\ge 0$ and consider the random vector clipped at level $-c$ given by
\begin{align*}
\left(Y_i^{(j)} \vee (-c), \ Y_j^{(j)}\right)
\end{align*}
for each pair $(i,j)\in V\times V$ and $i\neq j$.
Under the stochastic representation in~\eqref{eq:StochRepr}, the pair above is distributed as
\begin{equation*}
\label{eq:EqualinDf}
\left((E+Z_{i}^{(j)})\vee (-c), \ E\right).
\end{equation*}
Assume that $\bY$ follows a parametric model with parameter vector $\bm{\theta} \in \Theta \subset \mathbb{R}^p$. Write $x_+ = x \vee 0$, and note that $(x \vee (-c)) + c = (x + c)_+$. We focus on the first- and second-order moments of the clipped vector shifted by $c$ defined as
\begin{equation}
    \label{eq:eijl}
  e_{ij}^{(\ell)}(\bm{\theta},c) := \E \left\{
        \left( Y_i^{(j)} +c \right)_{+}^{\ell} \right\}, \qquad i,j \in V, \, \ell=1,2.
\end{equation}
These clipped moments always exist for any $i,j\in V$ (see \cref{lemma:ExistTruVario} in~\cref{sec:A.2}).

Recall that $\bm{X}_{t}=(X_{t1},\ldots,X_{td})$, $t=1,\ldots,n,$ are independent copies of the random vector $\bm{X}$, where the marginally transformed vector $\tilde{\bm X}$ of $\bm X$ in \eqref{eq:marginal_Transform} belongs to the domain of attraction of $\bm Y$ in the sense of \cref{eq:MPD}. Let $k=k(n) \in \cbr{1,\ldots,n}$ be an intermediate sequence satisfying
\[
k\to\infty, \, \, k/n\to 0, \, \,  \mbox{as} \ n\to \infty.
\]
Recall that $\hat{Y}_{ti}$ is the logarithmic transformation to the empirical Pareto exceedance defined in \eqref{eq:emp_Y}.
Motivated by the convergence in~\eqref{eq:MPD}, for $\ell=1,2$ and any $i,j\in V$, the sample versions of $e_{ij}^{(\ell)}(\bm{\theta},c)$ in~\eqref{eq:eijl} can be constructed as
\begin{align}
\notag
\hat{e}_{n,ij}^{(\ell)}\left(k,c\right) &=\frac{1}{k} \sum_{t=1}^{n} \left( \log \left\{\frac{k}{n(1 - \hat{F}_i(X_{ti}))}\right\}+c\right)_{+}^{\ell} \I\!\left\{\hat{F}_{j}(X_{tj}) \ge 1-\frac{k}{n} \right\}
\\
\label{eq:Empetilde}
&=\frac{1}{k} \sum_{t=1}^{n} \left( \hat{Y}_{ti}+c\right)_{+}^{\ell} \I\!\left\{\hat{F}_{j}(X_{tj}) \ge 1-\frac{k}{n} \right\}.
\end{align}
Hence, a straightforward idea is that we can estimate the parameter $\bm{\theta}$ by letting
\begin{align}
\label{eq:MomentMethod}
    \hat{e}_{n,ij}^{(\ell)}\left(k,c\right)=e_{ij}^{(\ell)}(\bm{\theta},c), \qquad
    i,j \in V, \; \ell=1,2,
\end{align}
provided that the corresponding moment equations have a unique solution within the valid parameter space.

If the variogram matrix of a general MGPD is continuous in $\bm{\theta}$ for $m \in V$, then $\bm{\theta}$ can first be estimated by $\hat{\bm{\theta}}$ via the (generalized) method of moments, and the corresponding estimate of the variogram matrix can subsequently be obtained from $\hat{\bm{\theta}}$. This estimation idea applies, theoretically, to a general MGPD. However, in the following, we focus on its implementation for the \HR{} MGPD.

\subsection{Moment variogram estimators for a \HR{} MGPD}
\label{sec:MomentHR}

In the following, we assume that $\bm{Y}=(Y_v,v\in V)$ is a $d$-variate random vector following a \HR{} MGPD with parameter $\bm{\theta=\bm{\Gamma}}$, where $\bm{\Gamma}\in \mathcal{M}$ is a variogram matrix (cf. \eqref{eq:NegDefMatrix}).
By a straightforward calculation, the explicit expressions of the moment functions defined in \eqref{eq:eijl} for the \HR{} MGPD can be derived, as stated in the lemma below.

\begin{lemma}[{\bf clipped moment functions for a \HR{} MGPD}]
\label{lemma:monotonicity}
Assume the random vector $\bm{Y}$ follows a \HR{} MGPD with variogram matrix $\bm{\Gamma}=(\gamma_{ij})_{i,j\in V}$. Then, for fixed $c\in [0,\infty)$ and each pair of $i,j\in V$, we have
\begin{equation*}
   \E \left\{
        \left( Y_i^{(j)} + c \right)_{+}^{\ell} \right\}= e^{(\ell)}(\gamma_{ij},c),
\end{equation*}
with
 \begin{align}
 \label{eq:etildeFirst}
e^{(1)}(\gamma, c)&=
    \exp(c) \Phi\left(-\frac{c+\gamma/2}{\sqrt{\gamma}}\right)
     + \sqrt{\gamma} \, \phi\left(\frac{c-\gamma/2}{\sqrt{\gamma}}\right) + \left(1-\frac{\gamma}{2}+c\right) \Phi\left(\frac{ c-\gamma/2}{\sqrt{\gamma}}\right),
\end{align}
and
\begin{align}
\label{eq:etildeSecond}
\nonumber
    e^{(2)}(\gamma, c)&=
    2 \exp(c) \Phi\left(-\frac{c+\gamma/2}{\sqrt{\gamma}}\right)  - \sqrt{\gamma} \left(\frac{\gamma}{2} - c - 2\right) \phi\left(\frac{c -\gamma/2}{\sqrt{\gamma}}\right)  \\
    &\qquad \null + \left\{\frac{1}{4} \gamma^2 - \gamma c + (1 + c)^2 +1 \right\} \Phi\left(\frac{ c-\gamma/2}{\sqrt{\gamma}}\right),
\end{align}
where $e^{(\ell)}(0; c)$ is defined as the limit $e^{(\ell)}(0, c)=\lim_{\gamma\to 0} e^{(\ell)}(\gamma, c)$ with
\[
e^{(1)}(0,c)=\E \left(E+c\right) = 1+c, \qquad
e^{(2)}(0,c)=\E \left\{(E+c\right)^2\} = (1+c)^2+1.
\]
Moreover, both $e^{(1)}(\gamma, c)$ and $e^{(2)}(\gamma, c)$ are strictly decreasing in $\gamma$ on $(0,\infty)$.
\end{lemma}

Note that for fixed $c$, the moment functions allow for the construction of moment estimators. In consideration of the symmetry of the variogram matrix, i.e., the property that $\bm{\Gamma}=\bm{\Gamma}^{\top}$, we propose to estimate the entry $\gamma_{ij}$ for any $i,j\in V$ and $i\neq j$ based on~\eqref{eq:MomentMethod} by the moment estimators $\hat{\gamma}_{n,ij}^{\mathrm{M},(\ell)}$ defined by
\begin{equation}
\label{eq:M1}
    e^{(\ell)}\left(
        \hat{\gamma}_{n,ij}^{\mathrm{M},(\ell)}(k, c), c
    \right) =  \frac{1}{2} \left\{
        \hat{e}_{n,ij}^{(\ell)}(k,c) + \hat{e}_{n,ji}^{(\ell)}(k,c)
    \right\}, \quad \ell=1,2.
\end{equation}
Here, the estimators are defined only when $i \ne j$, since the diagonal elements of $\bm{\Gamma}$ are equal to zero by definition. In practice, the choice for $c$ is restricted to $[0, -\log(k/n)]$, since if $ c \ge -\log(k/n)$, the clipping has no effect at all on the value of $\hat{e}_{n,ij}^{(\ell)}(k, c)$.

As the choice of the clipping parameter $c$ is not the focus of this work, all subsequent discussions proceed under the assumption that $c$ is a given number.
To simplify the notation, for fixed $c$, we write $ e^{(\ell)}(\gamma) := e^{(\ell)}(\gamma, c)$, and we suppress \((k, c)\) in \(\hat{e}_{n,ij}^{(\ell)}(k,c)\) and \(\hat{\gamma}_{n,ij}^{\mathrm{M},(\ell)}(k,c)\) occasionally whenever no ambiguity arises.

\section{Consistency and asymptotic normality}
\label{sec:Results}

In this section, we present the theoretical properties of the proposed estimators. We start by establishing the weak consistency of the moment estimators defined in~\eqref{eq:M1}. The following proposition shows that the empirical moments $\hat{e}_{n,ij}^{(\ell)}(k,c)$ converge to their true counterparts for an arbitrary MGPD, providing a crucial step toward proving the weak consistency result stated in the subsequent theorem. Before stating the main result, we collect the conditions on $k(n)$ in the following assumption.

\begin{assumption}
\label{Asp:IntermediateSeq}
Assume $k=k(n)$ is an intermediate sequence such that $k\to \infty$ and $k/n\to0$ as $n\to\infty$.
\end{assumption}

\begin{proposition}
\label{pro:weakconsistency}
Let $\bY$ be an arbitrary MGP distributed random vector, and let $\bm{X}_{t}=(X_{t1},\ldots,X_{td})$, $t=1,\ldots,n$, denote independent copies of a random vector $\bm{X}$, which has continuous margins and the marginally transformed vector of $\bm{X}$ in \eqref{eq:marginal_Transform} lies in the domain of attraction of $\bY$. For $c \ge 0$ and for $i,j\in V$ with $i\neq j$, let $\hat{e}_{n,ij}^{(\ell)}(k,c)$ with $\ell=1,2$ be defined as in~\eqref{eq:Empetilde}. Then, under~\cref{Asp:IntermediateSeq}, we have
\[
\hat{e}_{n,ij}^{(\ell)}(k,c) \overset{\pr} \to
\E\left\{\left(Y_i^{(j)} + c\right)_{+}^{\ell}\right\}
, \  \text{as} \  n\to \infty.
\]
\end{proposition}

For a \HR{} MGPD, we show in~\cref{thm:WeakConstMoment} that the proposed moment estimators for the variogram are well-defined and converge weakly to the true values. The standing assumption is stated as follows.

\begin{assumption}
\label{Asp:HRGPD}
Assume $\bY$ is a \HR{} MGP distributed random vector parameterized by a $d\times d$ dimensional variogram matrix $\bm{\Gamma}\in \mathcal{M}$ stated in~\eqref{eq:NegDefMatrix}.
Suppose $\bm{X}_{t}=(X_{t1},\ldots,X_{td})$, $t=1,\ldots,n$, are independent copies of a random vector $\bm{X}$, which has continuous margins and the marginally transformed vector of $\bm{X}$ in \eqref{eq:marginal_Transform} lies in the domain of attraction of $\bY$ in the sense of~\eqref{eq:MPD}.
\end{assumption}

\begin{theorem}
\label{thm:WeakConstMoment}
Suppose that \cref{Asp:IntermediateSeq,Asp:HRGPD} hold. For $i,j\in V$ and $i\neq j$, the moment estimators in~\eqref{eq:M1} are well-defined with probability tending to one, and $\hat{\gamma}^{\mathrm{M},(\ell)}_{n,ij} \overset{\pr} \to \gamma_{ij}$ as $n\to \infty$ for $\ell=1,2$.
\end{theorem}

Recall that $F_i$ is the marginal distribution of $X_i$, $i\in V$. Let
\[
U_i=1-F_i(X_i),\quad i=1,\ldots,d,
\]
and denote the joint distribution function of $\bm U=(U_1,\ldots,U_d)$ by $C(\bm{x})$.
The convergence of multivariate threshold exceedances in \eqref{eq:MPD} is equivalent to the existence of the limit
\begin{equation}
\label{eq:UTDF}
\lim_{q\to 0} q^{-1} \pr\left(1-F_{i}(X_i)\le q x_{i}, i\in V\right)=
\lim_{q\to 0} q^{-1} C(q \bm{x})=R(\bm{x})
\end{equation}
for $\bm{x}\in[0,\infty]^d \setminus \{(\infty,\ldots,\infty)\}$, where $R(\bm{x})$ is called the \emph{tail copula} of $\bm{X}$, see, e.g., \citet{S06} and \citet{chiapino2019identifying}.

For any non-empty set $I\subset V$ and vector $\bm{x}_I=(x_i,i\in I)\in[0,\infty]^{|I|} \setminus \{(\infty,\ldots,\infty)\}$, define $R_{I}(\bm{x}_I)$ as the value of the function $R(\bm{x})$ evaluated at the point $\bm{x}$ whose components are $x_i$ for $i\in I$ and $\infty$ for $i\in V\setminus I$.
Let
\[
\dot{R}_{I}(\bm{x}_{I})=(\dot{R}_{I}^{i}(\bm{x}_{I}),i\in I)=(\partial R_I(\bm{x}_I)/\partial x_i, i\in I)
\]
be the vector of its first-order partial derivatives.

Assume $W$ is a mean-zero Gaussian process on $[0,\infty]^{d}\setminus \{(\infty,\ldots,\infty)\}$ with continuous trajectories and covariance function
\begin{equation}
\label{eq:covW}
\E\rbr{W(\bm{x})W(\bm{y})}=R(\bm{x}\wedge\bm{y}), \qquad \bm{x}, \bm{y}\in [0,\infty]^{d} \setminus \{(\infty,\ldots,\infty)\},
\end{equation}
with $\bm{x} \wedge \bm{y}=(x_i \wedge y_i, i\in V)$.
For any nonempty set $I\subseteq V$ and $\bm{x}_I\in [0,\infty]^{|I|}\setminus \{(\infty,\ldots,\infty)\}$,
define
\[
W_{I}(\bm{x}_{I}) = W(\bm{y}),
\]
where \(\bm{y} \in [0, \infty]^{d} \setminus \{(\infty, \ldots, \infty)\}\) is a vector such that \(y_i = x_i\) for \(i \in I\) and \(y_i = \infty\) for \(i \notin I\). In particular, $W_{i}(x_i)=W(\infty,\ldots, \infty, x_i, \infty, \ldots, \infty)$, where $x_i$ appears in the $i$-th component.
Define
\begin{align}
\label{eq:limitprocess}
B_{I}(\bm{x}_{I})=W_{I}(\bm{x}_{I})-\sum_{i\in I}\dot{R}_{I}^{i}(\bm{x}_{I})W_{i}(x_i),
\end{align}
which is a zero-mean stochastic process on $\bm{x}_I\in [0,\infty]^{|I|}\setminus \{(\infty,\ldots,\infty)\}$.

Next, we show the asymptotic normality of the empirical moments in~\cref{pro:WeakCovMoment}. Based on this result and using the delta method, the asymptotic normality of the moment estimators can be established, as shown in \cref{thm:AsyNormalityMoment}. A second-order condition, as stated in~\cref{Asp:SecondOrderCdt}, is required to control the convergence rate in~\eqref{eq:MPD}. Since the subsequent results rely on the continuity of the partial derivatives of the tail copula, we restrict our attention to the case $\gamma_{ij} > 0$ for all $i, j \in V$ with $i \ne j$, thereby excluding the degenerate case $\gamma_{ij}=0$, in accordance with the assumption in~\cref{Asp:PositiveVariogram}.

\begin{assumption}
\label{Asp:SecondOrderCdt}
There exist constants $\xi, K \in (0,\infty)$ such that for any $I\subseteq V$ with $|I|=2$
and $q\in (0,1)$, we have
\[
\sup_{x_{I}\in[1,\infty]^{|I|}} \big|\pr\left(\bm{F}_{I}(\bm{X}_{I})\le 1-q/\bm{x}_{I} \mid \bm{F}_{I}(\bm{X}_{I})\nleqslant 1-q \right)-\pr(\bm{Y}_{(I)}\le \bm{x})\big|\le K q^{\xi},
\]
where $\bm{Y}_{(I)}$ is the random vector obtained from~\eqref{eq:MPD} with $\bm{X}$ replaced by $\bm{X}_{I}$.
\end{assumption}

\begin{assumption}
\label{Asp:PositiveVariogram}
The off-diagonal elements of the variogram matrix $\bm{\Gamma}$ associated with the \HR{} MGPD in~\cref{Asp:HRGPD} are strictly positive, that is, $\gamma_{ij}>0$ for all $i,j\in V$ and $i\neq j$.
\end{assumption}

\begin{proposition}
\label{pro:WeakCovMoment}
If \cref{Asp:IntermediateSeq,Asp:HRGPD,Asp:SecondOrderCdt,Asp:PositiveVariogram} hold with $k=o\left(n^{\xi/\left(\xi+\frac{1}{2}\right)}\right)$, then
\begin{multline*}
\left(\sqrt{k}\left\{\hat{e}_{n,ij}^{(\ell)}(k,c)  -e^{(\ell)}(\gamma_{ij},c)\right\},  \, i,j \in V, \, i\neq j, \, \ell=1,2\right) \\
\overset{d} \to \left(\int_{0}^{\exp(c)} \frac{B_{ij}(x,1) \cdot \ell \cdot (-\log x + c)^{\ell-1}}{x} \d x, \, i,j \in V, \, i\neq j, \, \ell=1,2 \right)
\end{multline*}
as $n\to \infty$,
where for each pair $(i,j)\in V \times V$ such that $i\neq j$, $B_{ij}(x,1)$ is the stochastic process defined in~\eqref{eq:limitprocess} and given by
\begin{equation}
\label{eq:ProcessBij}
B_{ij}(x,1)=W_{ij}(x,1)-\dot{R}_{ij}^{i}(x,1)W_{i}(x)-\dot{R}_{ij}^{j}(x,1)W_{j}(1).
\end{equation}
\end{proposition}

\begin{theorem}
\label{thm:AsyNormalityMoment}
Under the assumptions of \cref{pro:WeakCovMoment}, we have
\begin{multline*}
    \left(\sqrt{k}(\hat{\gamma}_{n,ij}^{\mathrm{M},(\ell)}-\gamma_{ij}); \; i,j \in V, \, i \neq j\right) \\
    \overset{d}\to
    \left(\frac{1}{2  e^{(\ell),'}(
    \gamma_{ij}) }
        \int_{0}^{\exp(c)} \frac{\left(B_{ij}(x,1) + B_{ji}(x,1) \right) \cdot \ell \cdot (-\log x + c)^{\ell-1} }{x} \d x ; \, i,j \in V, \, i \neq j\right)
\end{multline*}
as $n\to \infty$, where $B_{ij}(x,1)$ is given in~\eqref{eq:ProcessBij}, and where $e^{(\ell),'}(\gamma)$ is the derivative of $e^{(\ell)}(\gamma)$ in \eqref{eq:etildeFirst} and \eqref{eq:etildeSecond} with respect to $\gamma$, given by
\begin{align*}
    e^{(1),'}(\gamma)
    &=
    -\frac{1}{2}\,\Phi\left(\frac{c-\tfrac{\gamma}{2}}{\sqrt{\gamma}}\right), &
    e^{(2),'}(\gamma)
    &=\left(\frac{\gamma}{2} - c \right)  \Phi\left(\frac{ c -\frac{\gamma}{2}}{\sqrt{\gamma}}\right)-\sqrt{\gamma} \, \phi\left(\frac{ c - \frac{\gamma}{2}}{\sqrt{\gamma}}\right).
\end{align*}
\end{theorem}

\section{Simulation study}
\label{sec:simstudy}

In this section, we study the finite-sample behavior of the proposed estimators on simulated data. To investigate the performance of the moment estimator in~\eqref{eq:M1}, we compare its finite-sample behavior with that of the empirical variogram estimator in~\eqref{eq:MeanEmpVario}. For convenience, we set $c = -\log a$ with $a \in (0,1]$. All plots in the simulation study and empirical analysis are presented in terms of $a$, since the range $(0,1]$ is easier for selecting values.

We first generate $n=1000$ independent samples
$\bm{X}_{1}^{\star},\ldots,\bm{X}_n^{\star}$ from two \HR{} max-stable models,
a 10-dimensional model with randomly generated variogram matrix $\bm{\Gamma}_1$ (see \cref{sec:gamma2}),
and a 5-dimensional model with variogram matrix
\begin{align*}
\bm{\Gamma}_2=
\begin{pmatrix}
0 & 3 & 4 & 7 & 6 \\
3 & 0 & 3 & 8 & 5  \\
4 & 3 & 0 & 7 & 8  \\
7 & 8 & 7 & 0 & 11  \\
6 & 5 & 8 & 11 & 0
\end{pmatrix}.
\end{align*}
Both models are in the domains of attraction of the corresponding \HR{} MGPDs with the same variogram matrices.
The matrix $\bm{\Gamma}_1$ is obtained by first generating a positive definite matrix $\bm \Sigma$ and then computing the variogram matrix using the function \texttt{Sigma2Gamma} from the \texttt{graphicalExtremes} package in software \texttt{R}. The second distribution is the example considered in \citet{ELV2021} with the name of \emph{non-faithful \HR{} distribution}, since $Y_2$ and $Y_4$ are conditionally independent (in the sense of extremal conditional independence defined in~\citet{EH2020}) given $Y_1$ and $Y_3$.
In order to perturb the samples, we add standard normally distributed noise. To be precise, we set
\begin{equation*}
\bm{X}_t=\bm{X}^{\star}_t+|\bm{\epsilon}_t|, \qquad t=1,\ldots,n,
\end{equation*}
where $\bm{\epsilon}_t=(\epsilon_{ti}, \, i\in V)$, for $t=1,\ldots, n$, are independent standard normal random variables, independent of $\bm{X}^{\star}_t$, $t=1,\ldots, n$.

Under these simulation settings, we investigate the finite-sample performance of the proposed two moment estimators and compare it with the empirical variogram estimator. The performance of the estimators is evaluated by computing the average relative squared error between the estimates $\hat{\gamma}_{n,ij}$ and their true counterparts $\gamma_{ij}$ of the variogram matrix $\bm{\Gamma}$, defined as
\begin{equation*}
L(\hat{\bm{\Gamma}},\bm{\Gamma})=\left\{\frac{2}{d(d-1)}\sum_{i,j\in V, \, i < j}\left(\hat{\gamma}_{n,ij}/\gamma_{ij}-1\right)^2\right\}^{1/2},
\end{equation*}
where $\hat{\bm{\Gamma}}=(\hat{\gamma}_{n,ij})_{i,j\in V}$ denotes the estimates of $\bm{\Gamma}$ based on one of the aforementioned estimators. In addition, we also measure the average distance between the model-based and empirical (data-based) tail dependence coefficients, namely,
\begin{equation}
\label{eq:D}
    D(\hat{\mathcal{\bm{\chi}}}, \hat{\mathcal{\bm{\chi}}}^{\rm{EST}})=\frac{2}{d(d-1)}\sum_{i,j\in V, \, i < j}|\hat{\chi}_{ij}-\hat{\chi}_{ij}^{\rm{EST}}|
\end{equation}
where $\hat{\mathcal{\bm{\chi}}}=(\hat{\chi}_{ij})_{i,j\in V}$ and $\hat{\mathcal{\bm{\chi}}}^{\rm{EST}}=(\hat{\chi}_{ij}^{\rm{EST}})_{i,j\in V}$ denote the empirical and model-based tail dependence coefficient matrix, respectively. More specifically, for each pair $(i,j)\in V^2$ with $i \ne j$, the empirical tail dependence coefficient is calculated by
\[
    \hat{\chi}_{ij}
    = \frac{1}{k} \sum_{t=1}^n \I\!\left\{1-\hat{F}_i(X_{ti})\le \frac{k}{n}, 1-\hat{F}_j(X_{tj})\le \frac{k}{n}\right\},
\]
with the same $k$ used in the estimation of the variogram matrix, and the model-based tail dependence coefficients are obtained via
\[
    \hat{\chi}_{ij}^{\rm{EST}}
    = 2 \left\{1-\Phi\left(\sqrt{\hat{\gamma}_{n,ij}}/2\right)\right\}.
\]

The simulation results for samples from the $10-$dimensional \HR{} model are shown in \crefrange{fig:MeanerrorM1}{fig:combined_Gamma1}. The impact of the clipping level $c$ (or $a$ equivalently) on the performance of the moment estimators $\hat{\gamma}_{n,ij}^{\mathrm{M},(\ell)}(k,c)$, $\ell=1,2$, is examined first. \Cref{fig:MeanerrorM1,fig:MeanerrorM2} present the mean $L$ and $D$ corresponding to the first-and second-order moment estimator with
$a=\exp(-c)=0.25,0.5$ and $k=50,75,\ldots,250$, respectively.
For estimator $\hat{\gamma}_{n,ij}^{\mathrm{M,(1)}}(k,-\log a)$, we see that smaller values of $a$ yield better performance for small $k$, whereas larger values of $a$ become preferable as $k$ increases. This indicates a trade-off between $a$ and $k$, and suggests the existence of an optimal choice of $k$ in terms of $L$. In contrast, the metric $D$ exhibits a different trend, where smaller $a$ always results in smaller deviations. For estimator $\hat{\gamma}_{n,ij}^{\mathrm{M,(2)}}(k,-\log a)$ in \cref{fig:MeanerrorM2}, both $L$ and $D$ consistently indicate that smaller $a$ values yield lower estimation bias.

To compare the performance of the estimators, we set $a=0.25$ for $\hat{\gamma}_{n,ij}^{\mathrm{M},(\ell)}(k,-\log a)$, $\ell=1,2$, and show the boxplots of $L$ and $D$ values based on $300$ replications for each of the three estimators in \cref{fig:Cmp10D}. It suggests that the first-order moment estimator $\hat{\gamma}_{n,ij}^{\mathrm{M},(1)}(k,-\log a)$ outperforms the empirical variogram estimator, especially when $k$ is small, whereas the second-order moment estimator only shows superior performance for large $k$, in terms of $L$. In \cref{fig:combined_Gamma1}, we further analyze the averaged bias, variance and MSE over $300$ replications of the three estimators $\hat{\gamma}_{n,ij}^{\mathrm{EMP}}(k)$, $\hat{\gamma}_{n,ij}^{\mathrm{M},(1)}(k,-\log a)$ and $\hat{\gamma}_{n,ij}^{\mathrm{M},(2)}(k,-\log a)$ for the element $\gamma_{12}=0.89$ of $\bm{\Gamma}_1$, where $a=0.25$ and $k=50,75,\ldots,250$. The results indicate that, the empirical estimator tends to have a larger bias but smaller variance. The proposed first-order moment estimator is particularly advantageous for small $k$, and in some cases it uniformly dominates the empirical estimator in terms of MSE. Additional results for other elements of $\bm{\Gamma}_1$ can be found in \crefrange{fig:Cmp10DBias}{fig:Cmp10DMSE} in \cref{sec:gamma2}.

The corresponding results for random samples from the $5$-dimensional \HR{} distribution with variogram matrix $\bm{\Gamma}_2$ are illustrated in \crefrange{fig:MeanerrorM15D}{fig:combined_Gamma2}. The estimation error $L$ varies with $a$ in a manner similar to those of $\hat{\gamma}_{n,ij}^{\mathrm{M},(1)}(k,-\log a)$ observed for the $10$-dimensional \HR{} distribution.
Moreover, the value of $D$ exhibits greater stability with respect to changes in $k$.
It could be noted in \cref{fig:Cmp5D} that, under such circumstances, namely, when the extremal dependence between certain components is relatively weak or asymptotically conditional independent, the moment estimators yield improved performance. We see that both moment estimators $\hat{\gamma}_{n,ij}^{\mathrm{M},(\ell)}$, $\ell=1,2$, have much lower estimation error than the empirical variogram estimator. In this case, \cref{fig:combined_Gamma2} indicates that $\hat{\gamma}_{n,ij}^{\mathrm{M},(2)}$ achieves the smallest bias and MSE for large values of $k$, while the empirical variogram estimator still performs best in terms of variance. For more details on other elements, see \crefrange{fig:Cmp5DBias}{fig:Cmp5DMSE} in \cref{sec:gamma2}.

Indeed, what distinguishes the two matrices $\bm{\Gamma}_1$ and $\bm{\Gamma}_2$ is that the dependence in $\bm{\Gamma}_2$ is weaker.
The simulations show that, for the \HR{} models with stronger dependence, the moment estimator tends to achieve better performance with relatively small values of $a$ (or $c$ equivalently). However, such a pattern is not clear for weakly dependent models.
Moreover, in both cases, the clipped moment estimators could provide improvements over the empirical variogram estimator, especially when the dependence is weak. Overall, the first-moment estimator with $a = 0.25$ comes out as the best one: in the two $L$ and $D$-plots, it has the lowest error in almost all cases. In addition, compared with the empirical variogram estimator, the two moment estimators appear to be less sensitive to the choice of $k$ in view of $D$-plots, when the dependence is weak.

\begin{figure}
\centering
\includegraphics[width=0.45\textwidth]{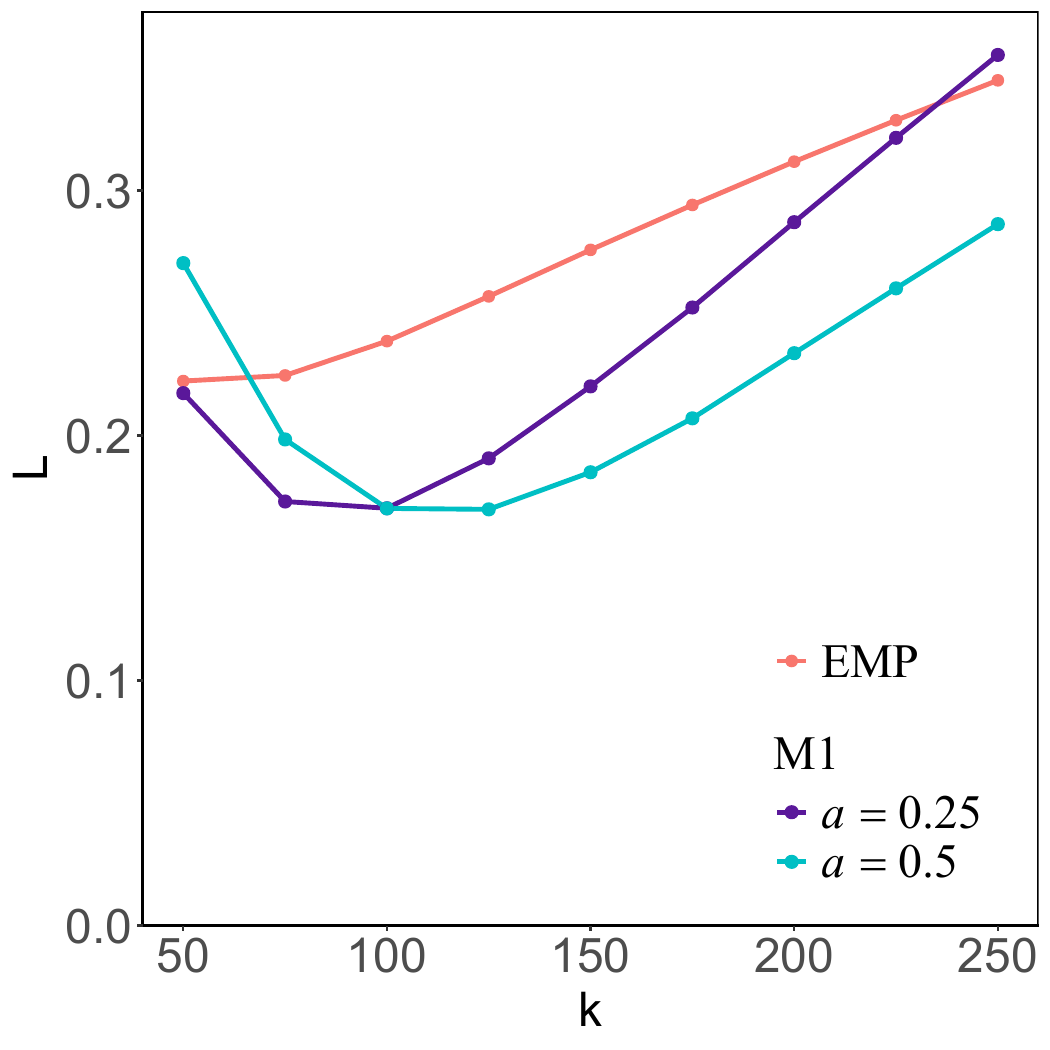}
\includegraphics[width=0.45\textwidth]{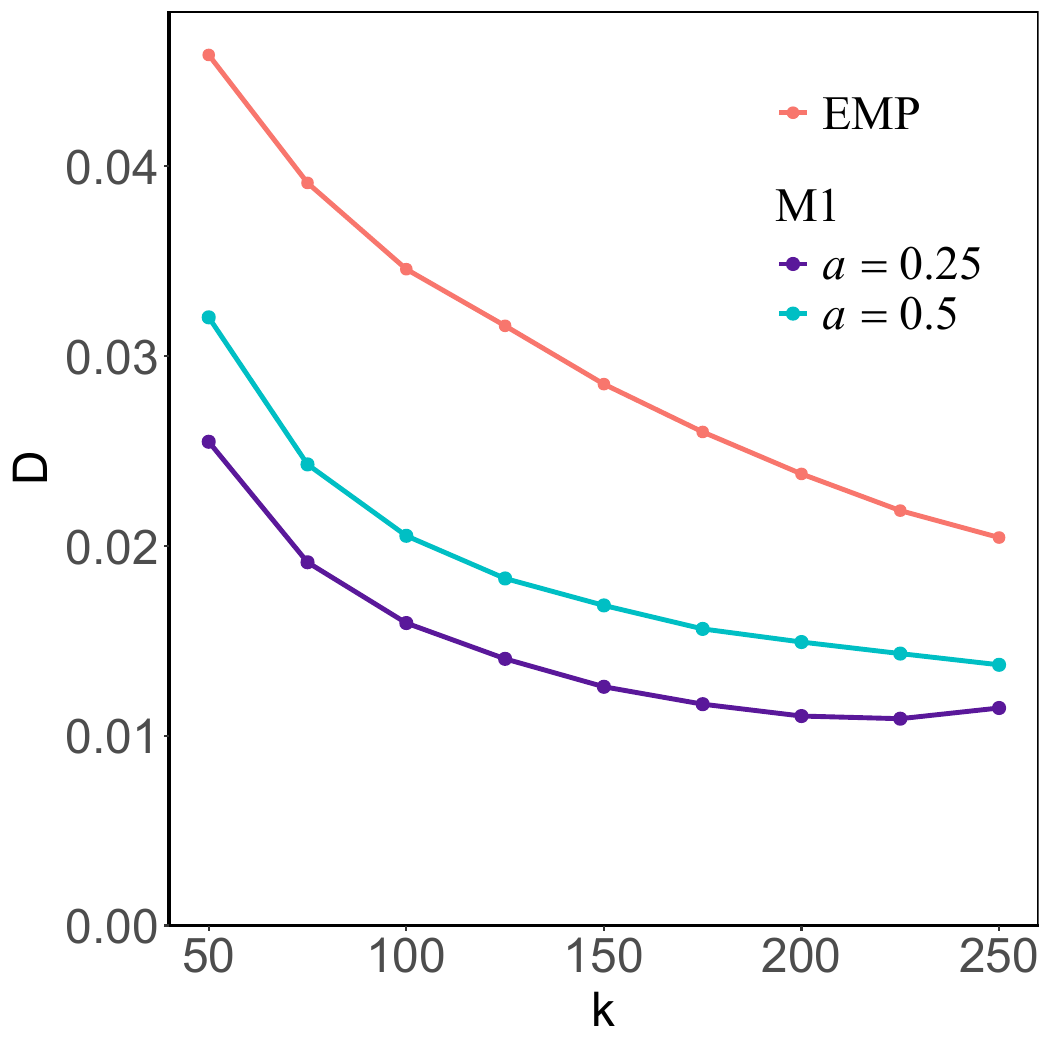}
\caption{The mean $L$ (left) and $D$ value (right) based on first-order moment estimator $\hat{\gamma}_{n,ij}^{\mathrm{M},(1)}(k,-\log a)$ (M1) with $a=0.25,0.5$ and the empirical estimator $\hat{\gamma}_{n,ij}^{\EMP}$ (EMP) over 300 replications, with random samples of size $n = 1000$ drawn from the $10$-dimensional \HR{} distribution with variogram matrix $\bm{\Gamma}_1$.
}
\label{fig:MeanerrorM1}
\end{figure}

\begin{figure}
\centering
\includegraphics[width=0.45\textwidth]{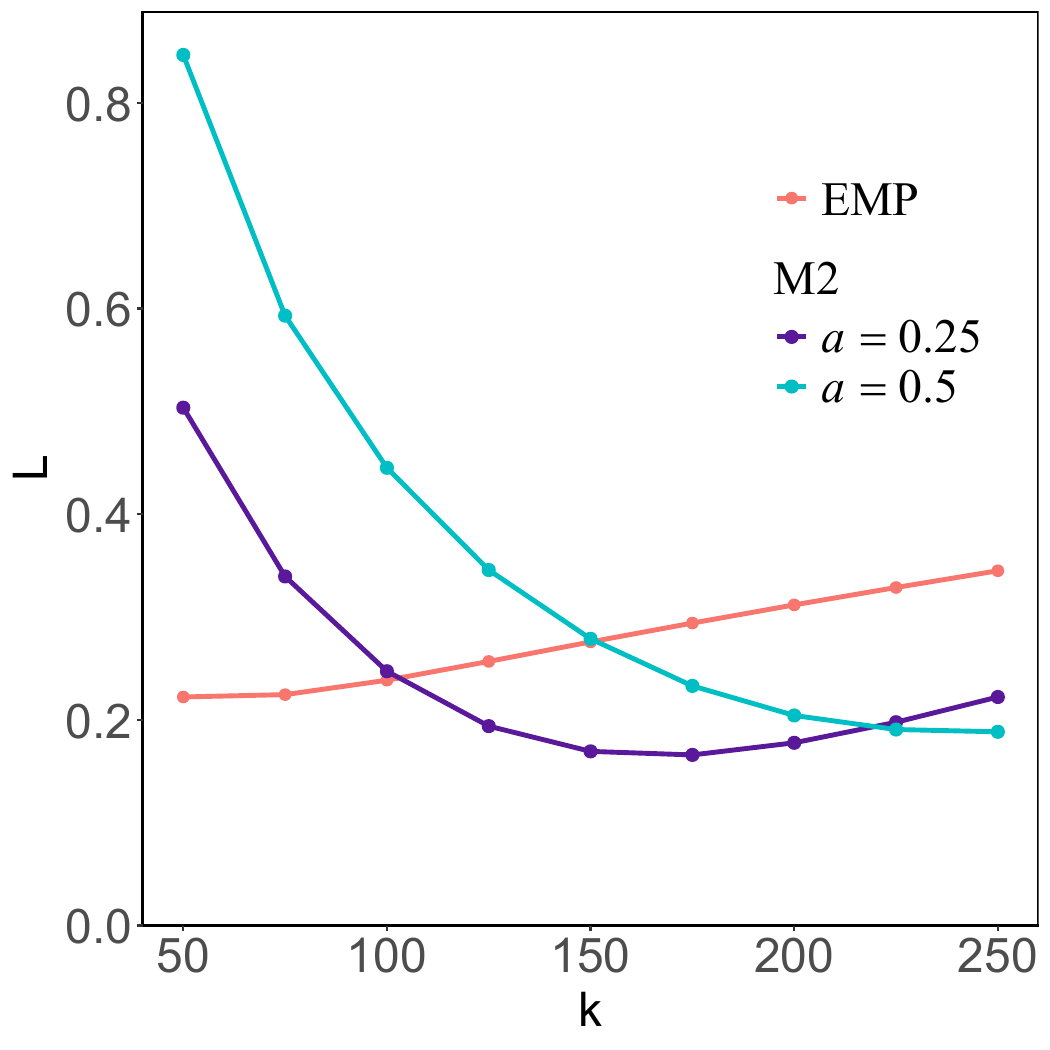}
\includegraphics[width=0.45\textwidth]{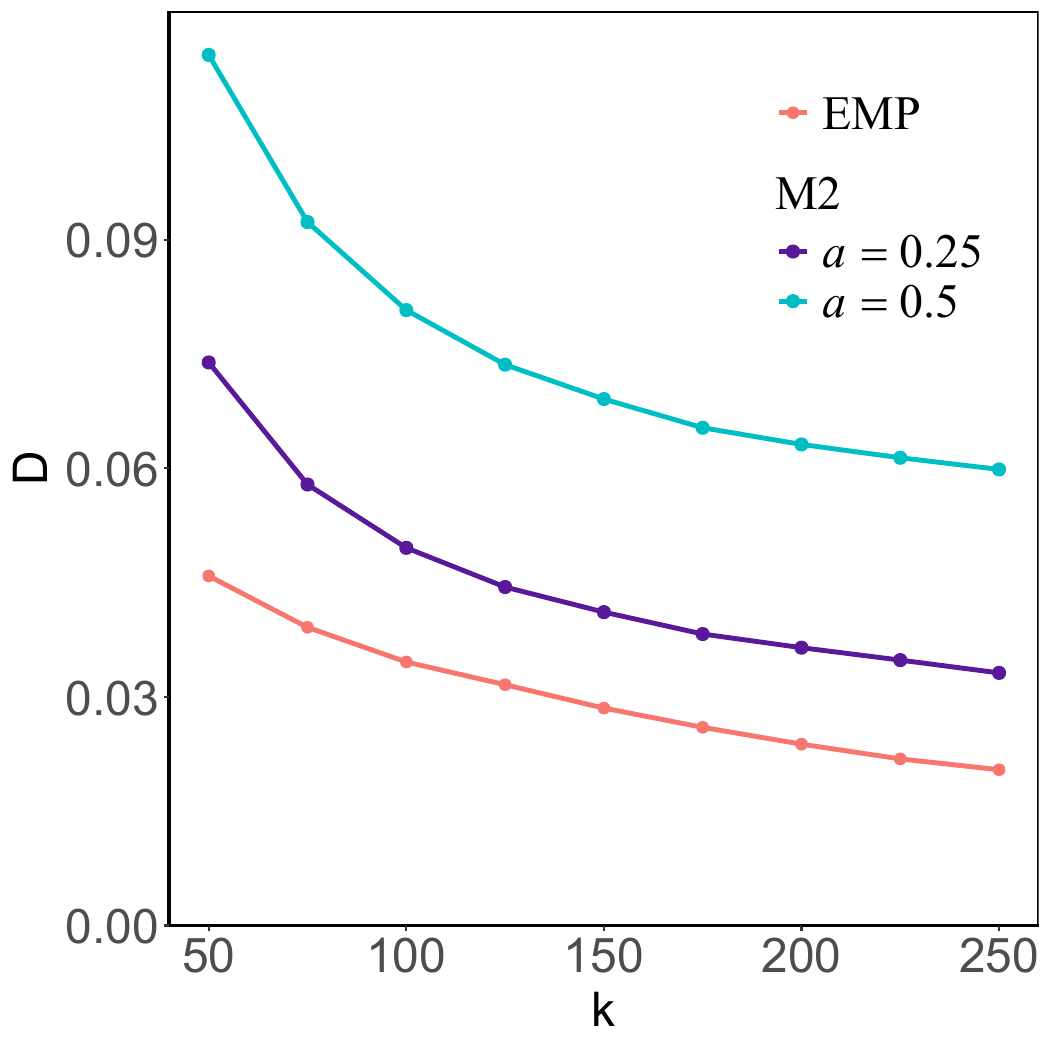}
\caption{The mean $L$ (left) and $D$ value (right) based on second-order moment estimator $\hat{\gamma}_{n,ij}^{\mathrm{M},(2)}(k,-\log a)$ (M2) with $a=0.25,0.5$ and the empirical variogram estimator $\hat{\gamma}_{n,ij}^{\EMP}$ (EMP) over 300 replications, with random samples of size $n = 1000$ drawn from the $10$-dimensional \HR{} distribution with variogram matrix $\bm{\Gamma}_1$.}
\label{fig:MeanerrorM2}
\end{figure}

\begin{figure}
\centering
\includegraphics[width=0.45\textwidth]{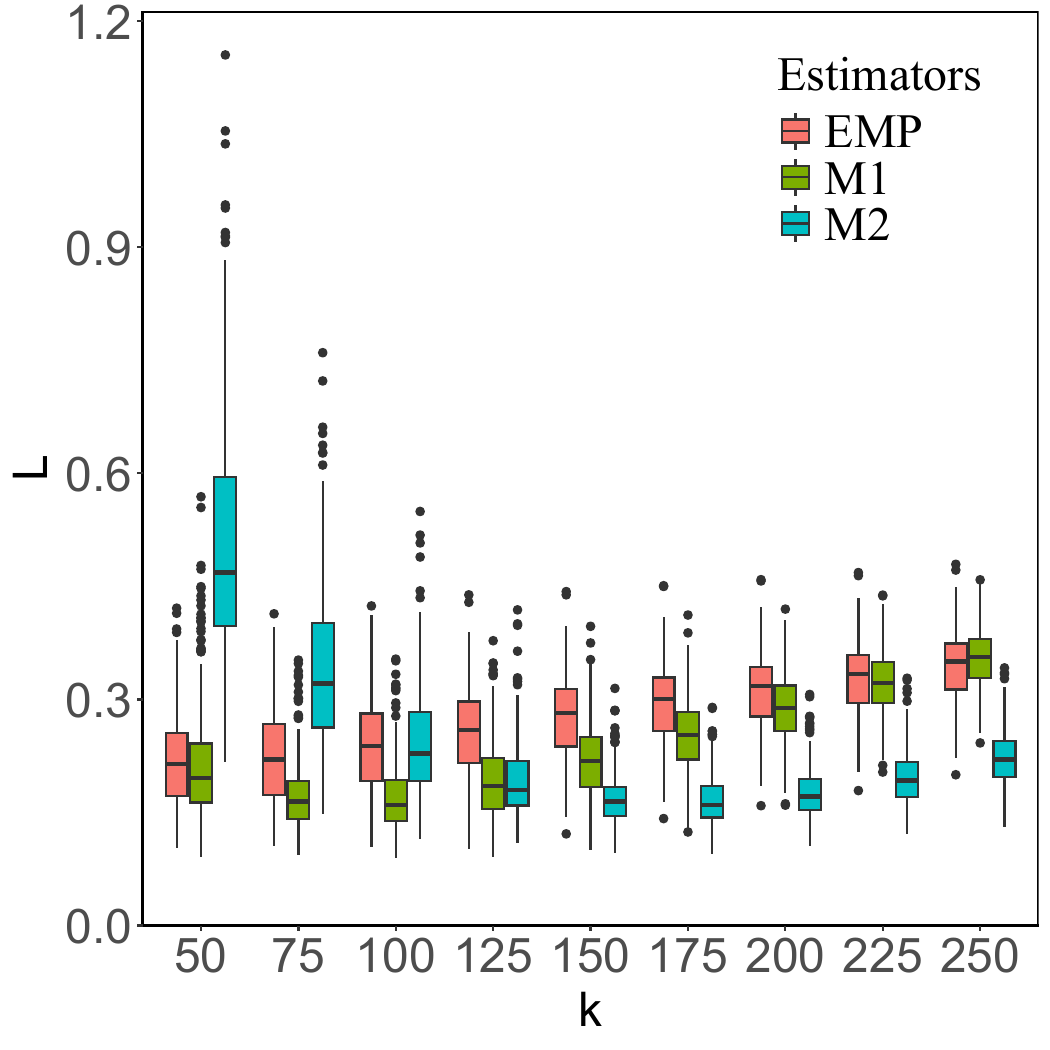}
\includegraphics[width=0.45\textwidth]{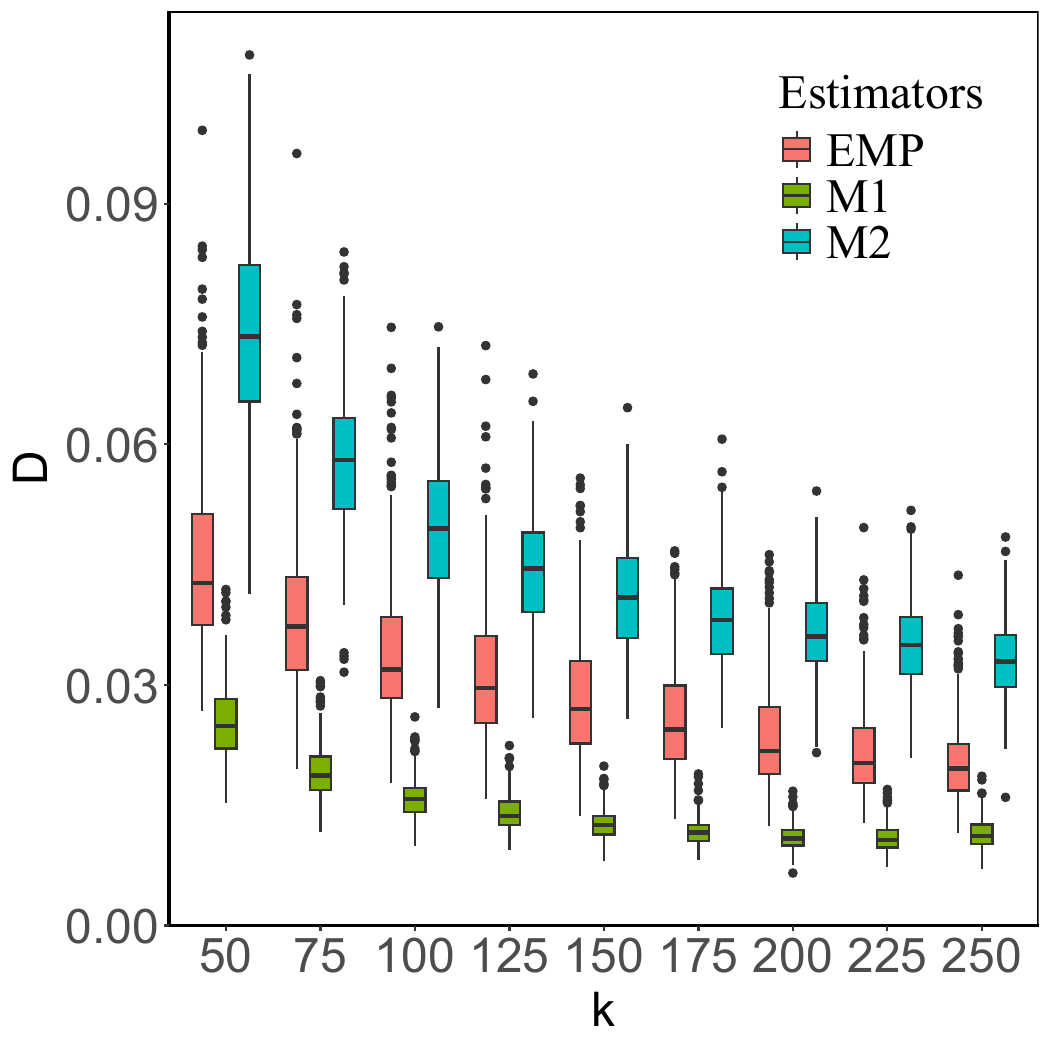}
\caption{The $L$ value (left) and distance $D$ of all bivariate tail dependence coefficients (right) based on the empirical variogram estimator $\hat{\gamma}_{n,ij}^{\EMP}$ (EMP), first-order moment estimator $\hat{\gamma}_{n,ij}^{\mathrm{M},(1)}(k,-\log a)$ (M1) and second-order moment estimator $\hat{\gamma}_{n,ij}^{\mathrm{M},(2)}(k,-\log a)$ (M2) with $a=0.25$ in 300 replications. The random samples with size $n=1000$ are drawn from the $10-$dimensional \HR{} distribution with variogram matrix $\bm{\Gamma}_1$.}
\label{fig:Cmp10D}
\end{figure}

\begin{figure}
\centering
\includegraphics[width=0.95\textwidth]{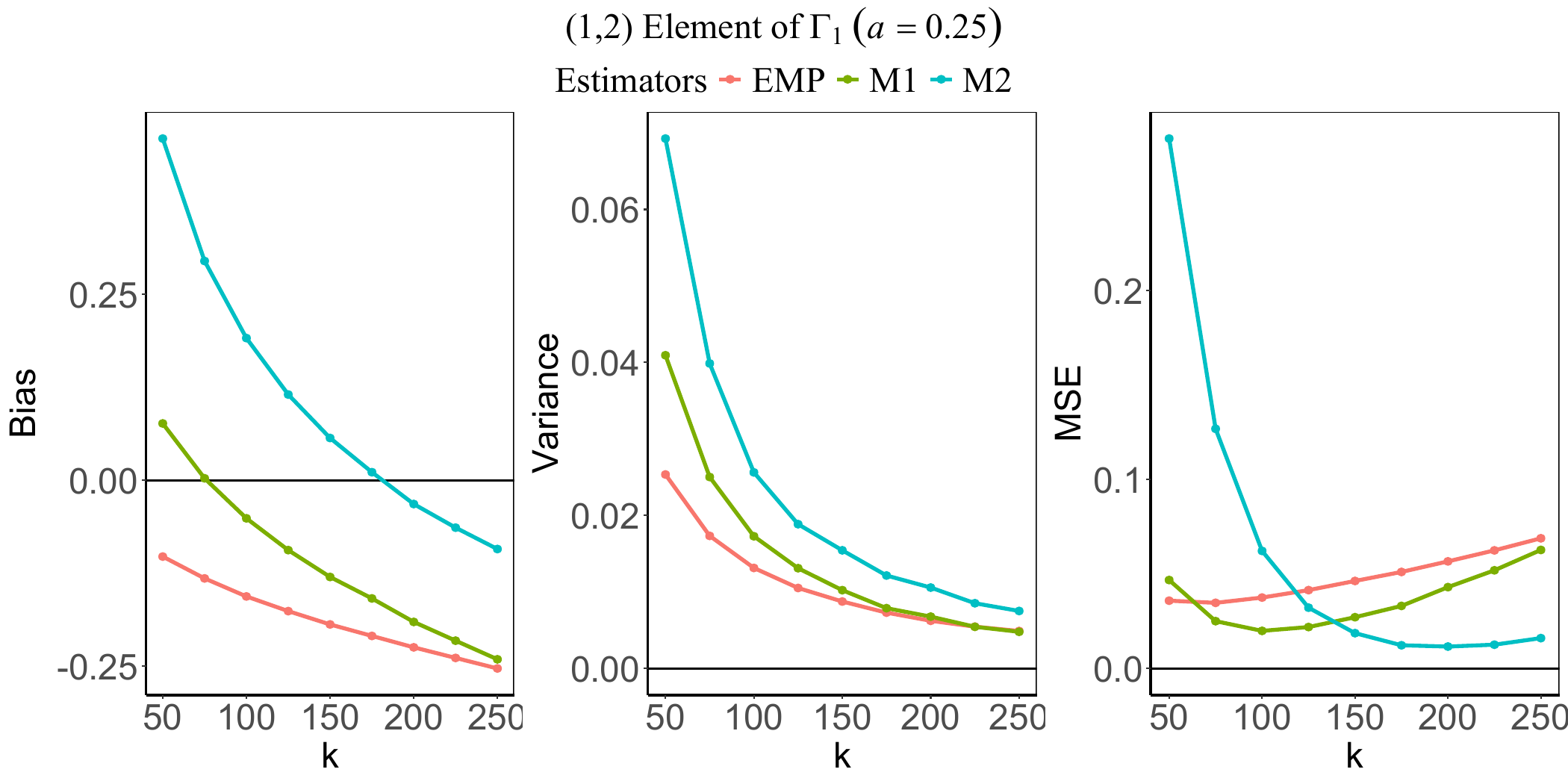}
\caption{The mean estimation bias, variance and MSE for $\gamma_{12}=0.89$ of $\bm \Gamma_1$, based on the empirical variogram estimator $\hat{\gamma}_{n,ij}^{\EMP}$ (EMP), first-order moment estimator $\hat{\gamma}_{n,ij}^{\mathrm{M},(1)}(k,-\log a)$ (M1) and second-order moment estimator $\hat{\gamma}_{n,ij}^{\mathrm{M},(2)}(k,-\log a)$ (M2) with $a=0.25$ in 300 replications. The random samples with size $n=1000$ are drawn from the $10-$dimensional \HR{} distribution with variogram matrix $\bm{\Gamma}_1$.}
\label{fig:combined_Gamma1}
\end{figure}

\begin{figure}
\centering
\includegraphics[width=0.45\textwidth]{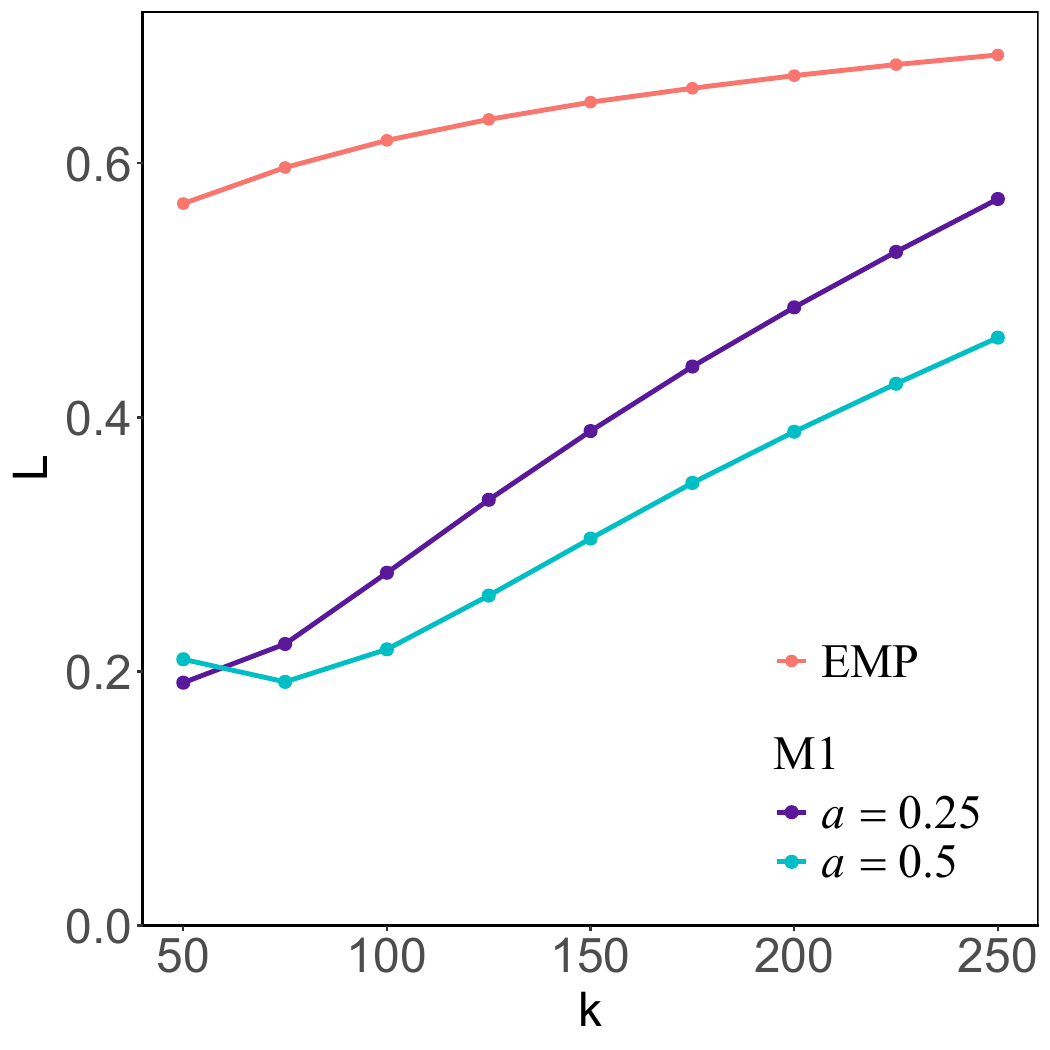}
\includegraphics[width=0.45\textwidth]{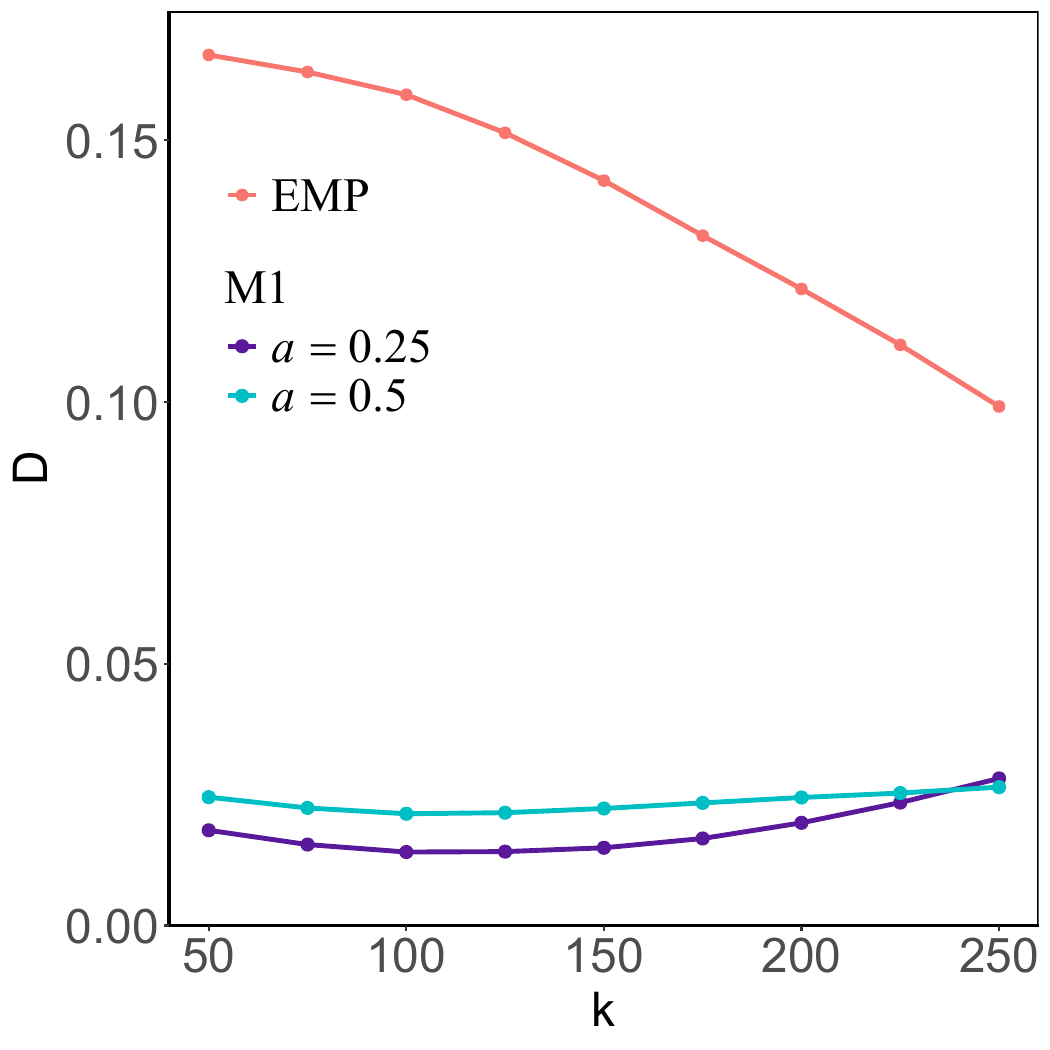}
\caption{The mean $L$ (left) and  $D$ value (right) based on the first-order moment estimator $\hat{\gamma}_{n,ij}^{\mathrm{M},(1)}(k,-\log a)$ (M1) and the empirical variogram estimator $\hat{\gamma}_{n,ij}^{\EMP}$ (EMP) with $a=0.25,0.5$ over 300 replications, with random samples of size $n = 1000$ drawn from the $5$-dimensional \HR{} distribution with variogram matrix $\bm{\Gamma}_2$.
}
\label{fig:MeanerrorM15D}
\end{figure}

\begin{figure}
\centering
\includegraphics[width=0.45\textwidth]{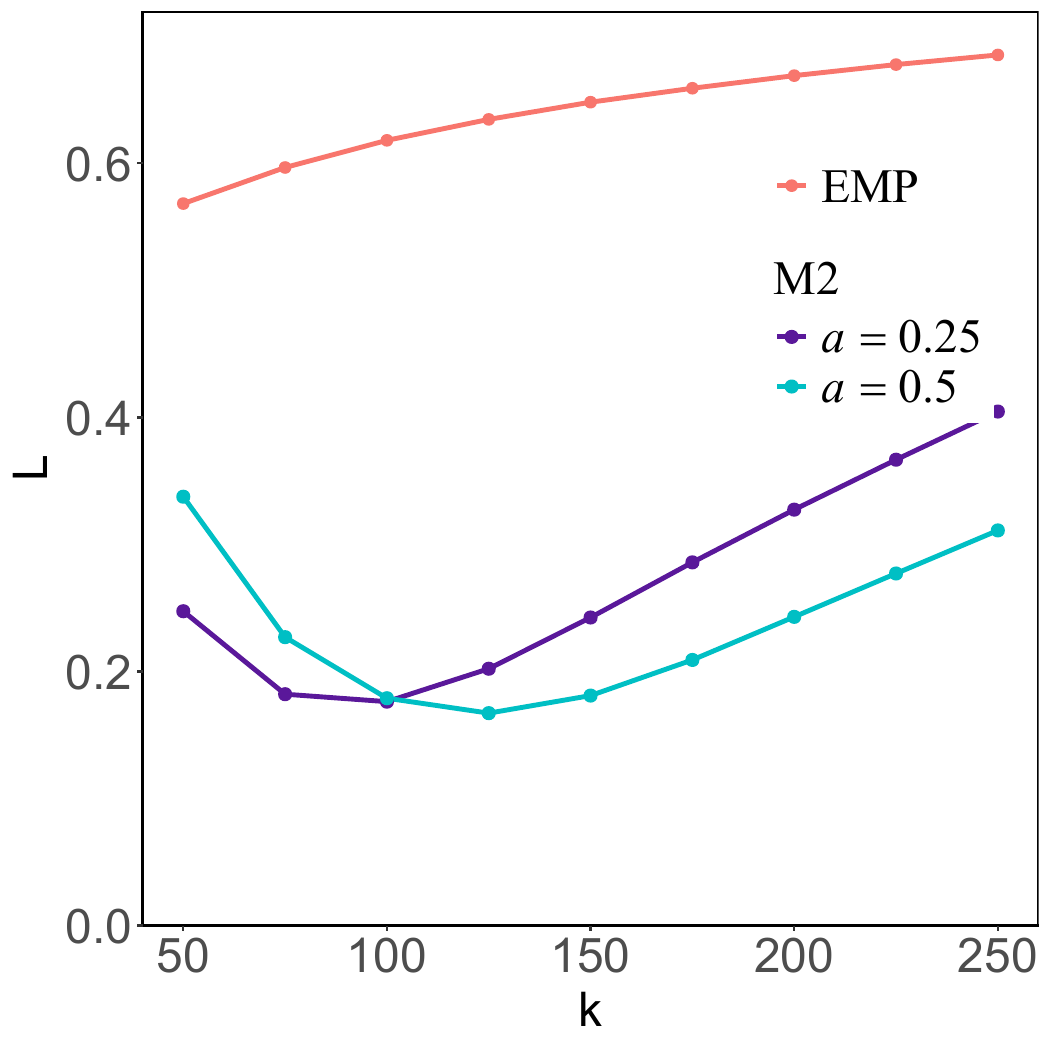}
\includegraphics[width=0.45\textwidth]{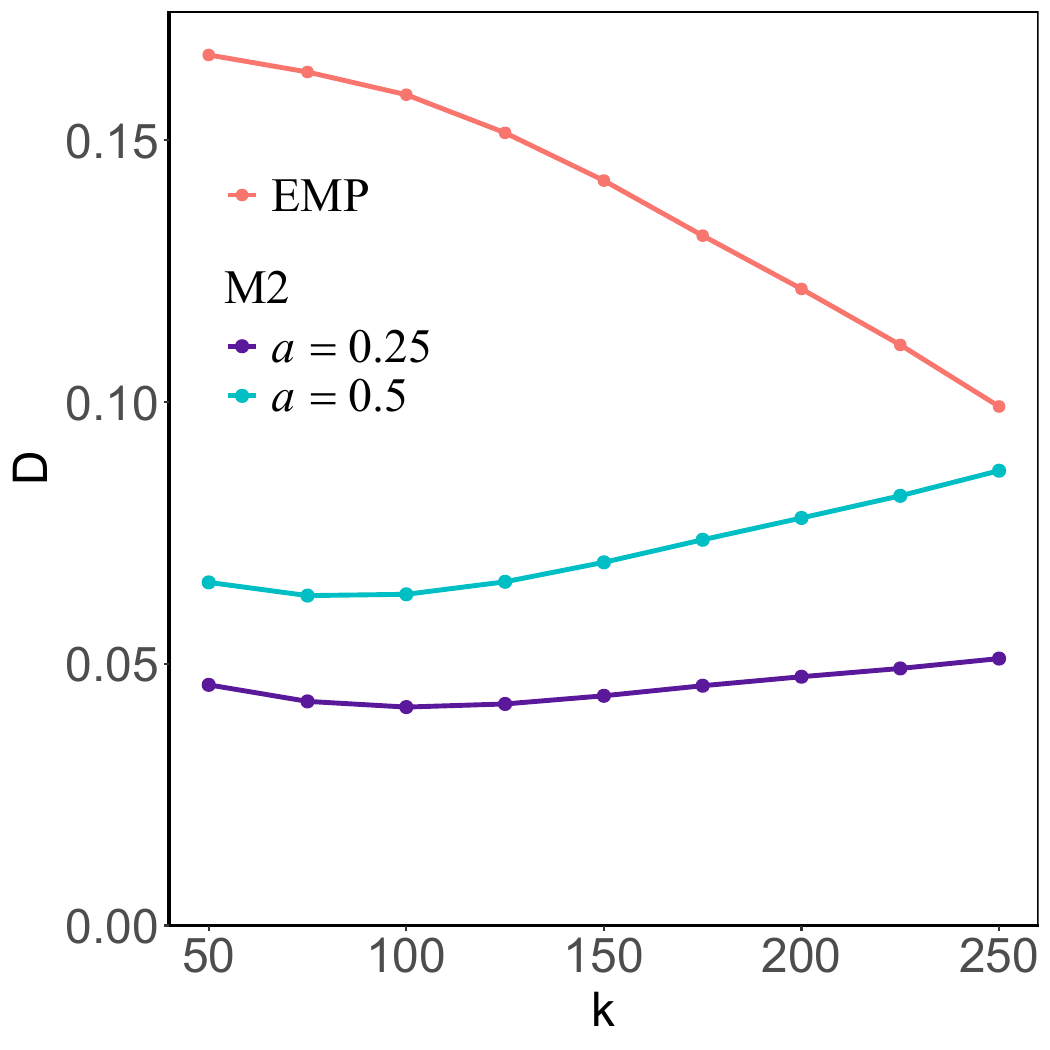}
\caption{The mean $L$ (left) and $D$ value (right) based on the second-order moment estimator $\hat{\gamma}_{n,ij}^{\mathrm{M},(2)}(k,-\log a)$ (M2) with $a=0.25, 0.5$ and the empirical estimator $\hat{\gamma}_{n,ij}^{\EMP}$ (EMP) over 300 replications, with random samples of size $n = 1000$ drawn from the $5$-dimensional \HR{} distribution with variogram matrix $\bm{\Gamma}_2$.}
\label{fig:MeanerrorM25D}
\end{figure}

\begin{figure}
\centering
\includegraphics[width=0.45\textwidth]{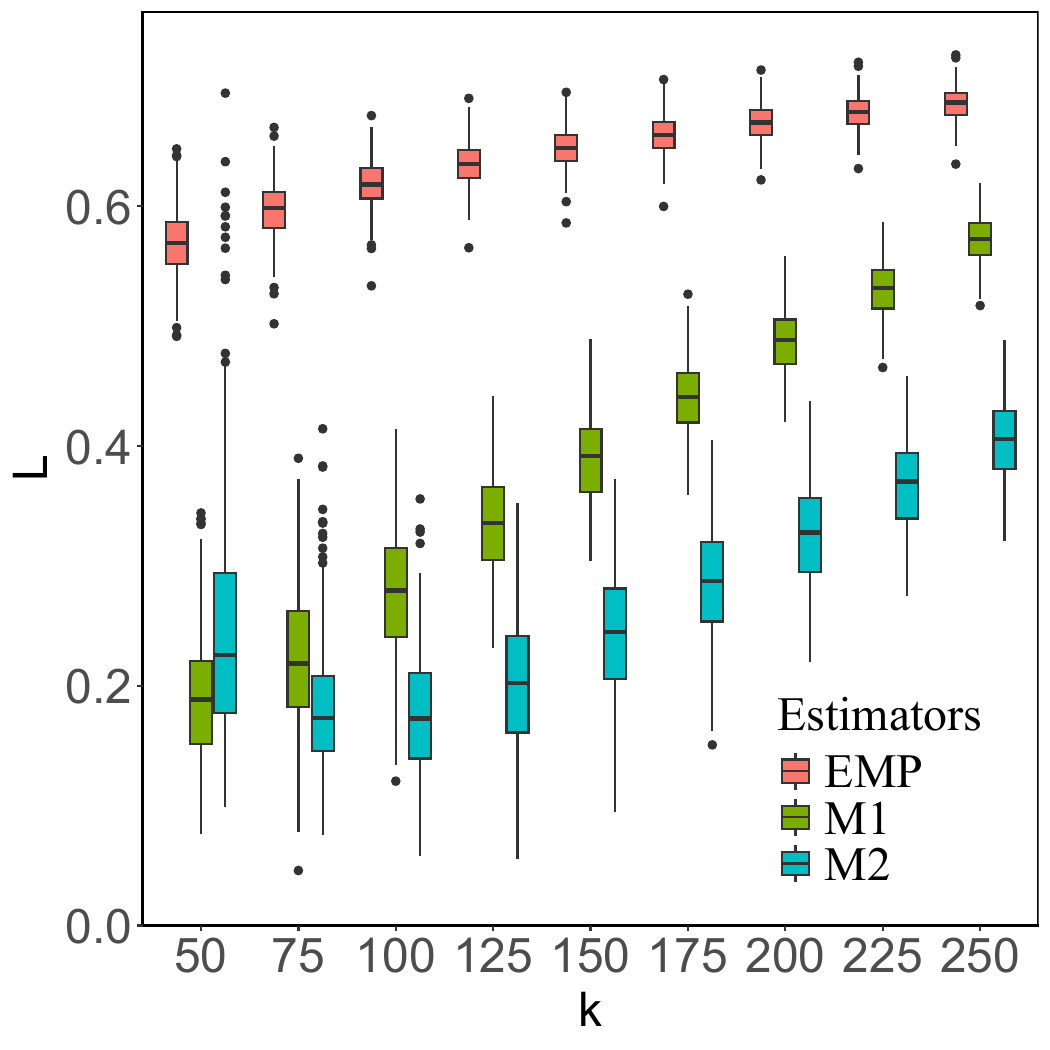}
\includegraphics[width=0.45\textwidth]{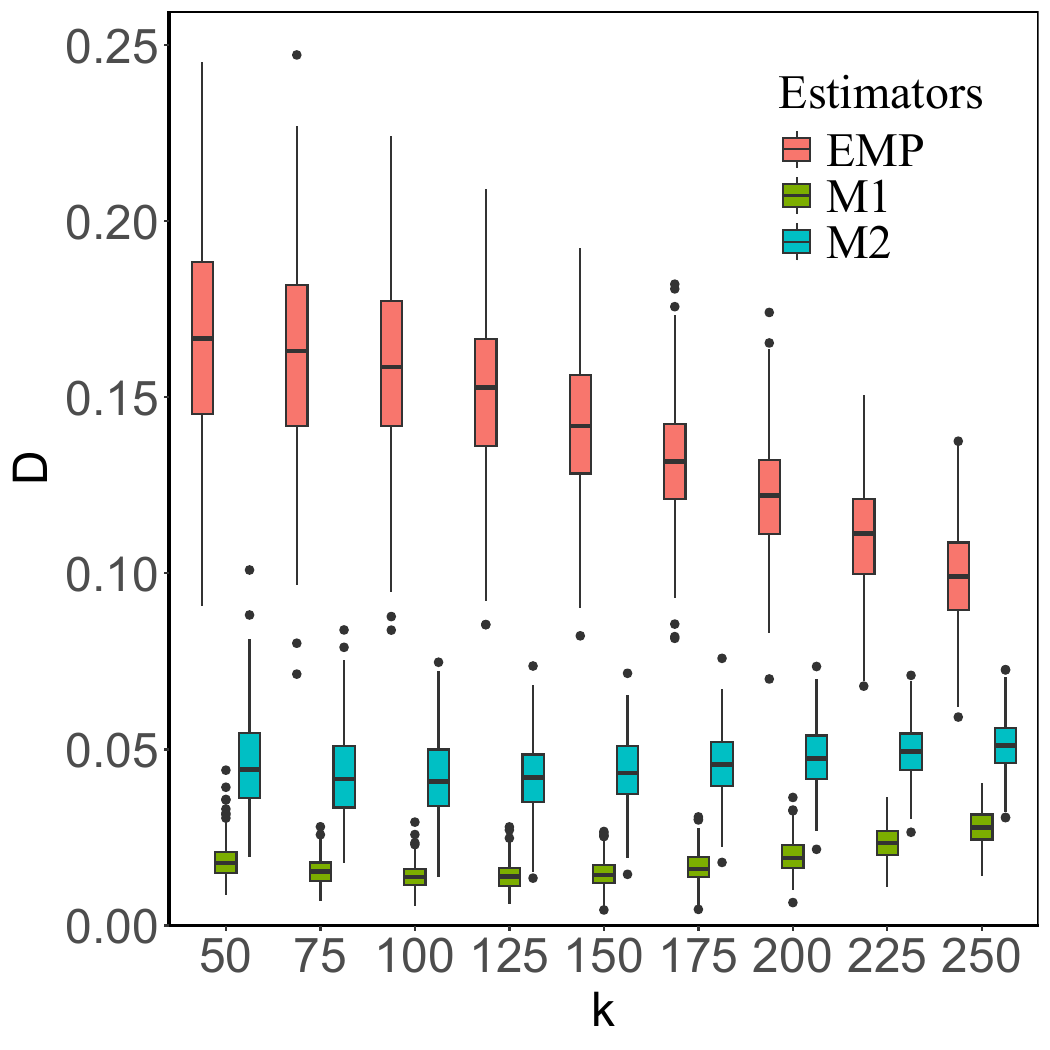}
\caption{The $L$ value (left) and distance $D$ of all bivariate tail dependence coefficients (right) based on the empirical variogram estimator $\hat{\gamma}_{n,ij}^{\EMP}$ (EMP), first-order moment estimator $\hat{\gamma}_{n,ij}^{\mathrm{M},(1)}(k,-\log a)$ (M1) and second-order moment estimator $\hat{\gamma}_{n,ij}^{\mathrm{M},(2)}(k,-\log a)$ (M2) with $a=0.25$ in 300 replications. The random samples with size $n = 1000$ are drawn from the $5$-dimensional \HR{} distribution with variogram matrix $\bm{\Gamma}_2$.}
\label{fig:Cmp5D}
\end{figure}

\begin{figure}
\centering
\includegraphics[width=0.95\textwidth]{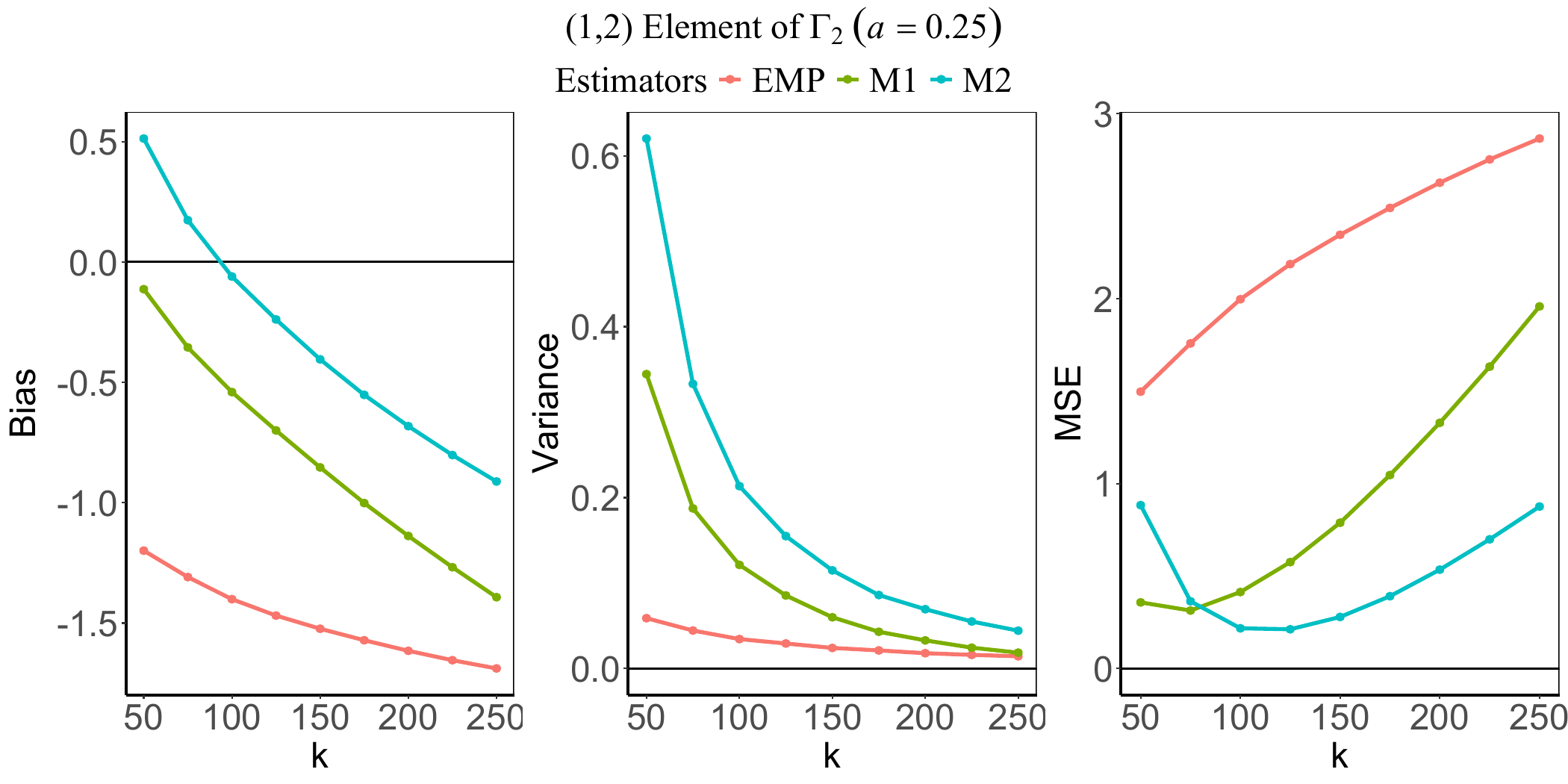}
\caption{The mean bias, variance and MSE for $\gamma_{12}=3$, based on the empirical variogram estimator $\hat{\gamma}_{n,ij}^{\EMP}$ (EMP), first-order moment estimator $\hat{\gamma}_{n,ij}^{\mathrm{M},(1)}(k,-\log a)$ (M1) and second-order moment estimator $\hat{\gamma}_{n,ij}^{\mathrm{M},(2)}(k,-\log a)$ (M2) with $a=0.25$ in 300 replications. The random samples with size $n=1000$ are drawn from the $5-$dimensional \HR{} distribution with variogram matrix $\bm{\Gamma}_2$.}
\label{fig:combined_Gamma2}
\end{figure}

\section{Application}
\label{sec:application}

\subsection{Danube discharges data}

To illustrate the practical performance of the proposed method, we fit a \HR{} model to the Danube discharge dataset. The variogram matrix $\bm{\Gamma}$ is estimated using the proposed moment estimators, and the resulting estimates are compared with those obtained from the empirical variogram estimator in~\citet{EV20}.
This dataset includes the average daily discharges recorded at 31 gauging stations in the upper Danube basin covering parts of Germany, Austria and Switzerland, available in the supplementary material of~\citet{ADE2015}. The series at individual stations has lengths from 54 to 113 years, with 51 years of data
for all stations from 1960 to 2010. We follow~\citet{ADE2015} and use the pre-processed data containing $n=428$ observations after a declustering of the time series to perform the estimation, considering the declustered samples as independent observations from a $31$-variate random vector $\bm{X}$ in the domain of attraction of a \HR{} distribution with variogram matrix $\bm{\Gamma}$.

To investigate how the proportion of samples used for estimation affects the results, we set $k/n\in \{0.05,0.10,0.15,0.20,0.25\}$. The absolute distance $D$ between the model-based and empirical tail dependence coefficients in~\eqref{eq:D} as a function of $k/n$ is plotted in the left panel of \cref{fig:Errvsp}, where $a=\exp(-c)$ is fixed to be $0.25$.
It can be seen from the plot that, compared to the empirical variogram estimator, the absolute distance $D$ for the first moment estimators is less sensitive to the choice of $k$ and has consistently smaller $D$ values than that of the empirical variogram estimator $\hat{\gamma}_{n,ij}^{\EMP}$.

To further assess the performance of the two estimation methods, we take $k/n=0.10$, and present the scatter plot of $(\hat{\lambda}_{ij},\hat{\lambda}_{ij}^{\rm EST})$ for $i,j\in V$ and $i\neq j$, where $\lambda_{ij}^{\rm EST}$ is the model-based tail dependence coefficient with the parameter matrix estimated by the moment estimators $\hat{\gamma}_{n,ij}^{\mathrm{M},(1)}(k,-\log a)$ and $\hat{\gamma}_{n,ij}^{\mathrm{M},(2)}(k,-\log a)$ with $a=0.25$, as well as the empirical variogram estimator $\hat{\gamma}_{n,ij}^{\EMP}$. \Cref{fig:Errvsp} shows that
the models based on the proposed estimators $\hat{\gamma}_{n,ij}^{\mathrm{M},(1)}$ provide a better fit to the empirical tail dependence coefficients than those of the other two estimators.

\begin{figure}
\centering
\includegraphics[width=0.45\textwidth]{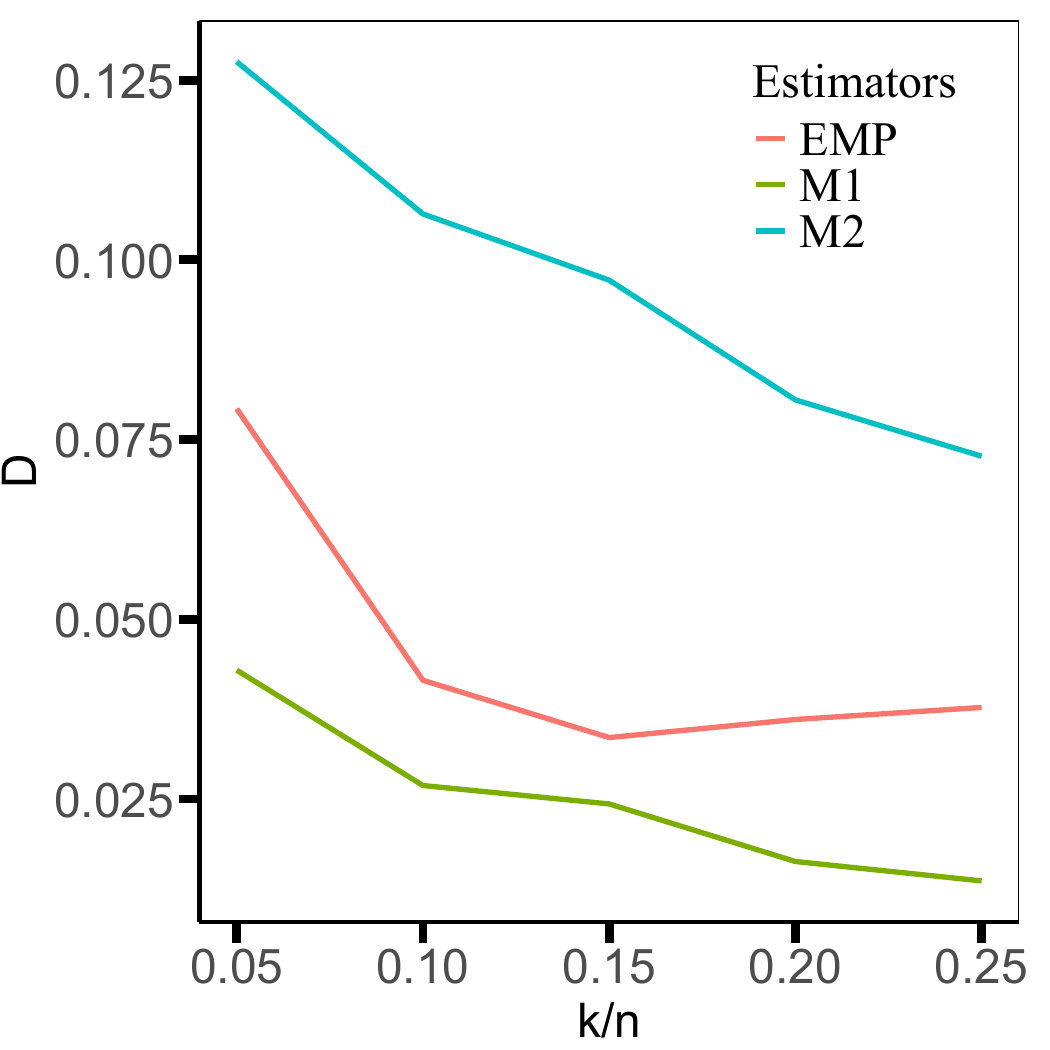}
\includegraphics[width=0.45\textwidth]{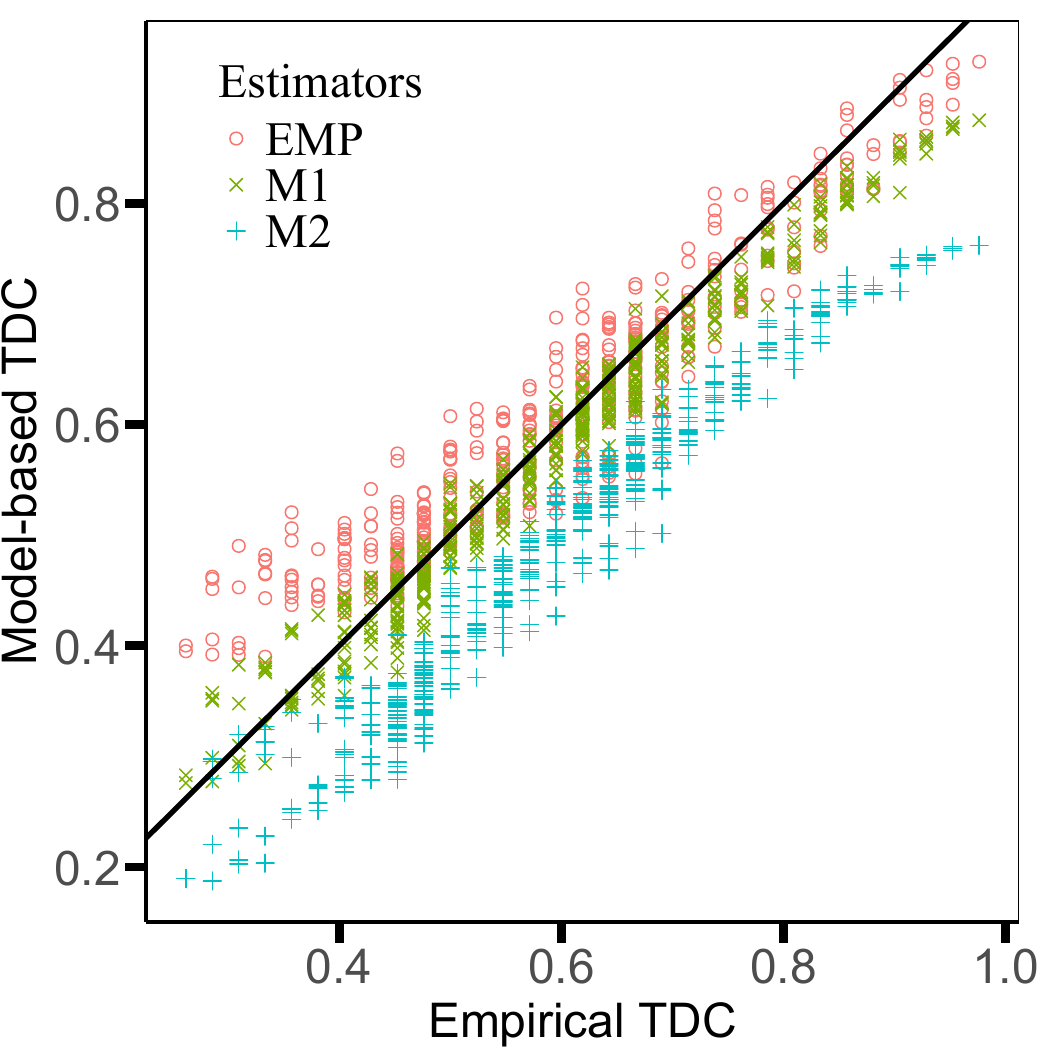}
\caption{
The $D$ values (left) and scatter plot (right) of empirical and model-based tail dependence coefficients (TDC) for the Danube discharges data, computed using the empirical variogram estimator $\hat{\gamma}_{n,ij}^{\mathrm{EMP}}$ (EMP),  the moment estimator $\hat{\gamma}_{n,ij}^{\mathrm{M},(1)}$ (M1) and $\hat{\gamma}_{n,ij}^{\mathrm{M},(2)}$ (M2), with clipping level $a=0.25$, ratio $k/n\in\{0.05,0.10,0.15,0.20,0.25\}$ for the line plot and $k/n=0.10$ for the scatter plot.}
\label{fig:Errvsp}
\end{figure}

\subsection{US flight delay data}

In this subsection, we apply the proposed estimators to the U.S. flight delay dataset, which is a widely used benchmark in the analysis of multivariate extremes (see, e.g., \citealp{HES2024,KLS2025}). The dataset contains records of domestic flights in the United States operated by major carriers (airlines with at least 1\% market share) and involving airports that account for at least 1\% of domestic enplanements. \citet{HES2024} preprocessed the dataset by restricting attention to airports located in the contiguous United States with at least 1000 flights per year. For each airport, daily accumulated positive flight delays (in minutes) are computed by summing arrival and departure delays. This results in a dataset consisting of $5601$ observations of daily accumulated flight delays over a $16-$year period across $170$ airports. Based on this dataset, they further apply a $K$-medoids clustering approach using a tail dependence coefficient distance matrix at probability level $p=0.85$, yielding six clusters. The original data are publicly available from the U.S. Bureau of Transportation Statistics. Preprocessed versions of the dataset can be obtained from the GitHub repository of Manuel Hentschel or through the \texttt{R} package \texttt{graphicalExtremes} \citep{HES2024}.

Here, we focus on the Texas cluster for model comparison purposes. This cluster consists of $d=29$ airports and contains $n=3603$ observation days over the period 2005--2020. We model the resulting data using the \HR{} model to capture tail dependence among extreme flight delays. The variogram matrix is estimated by the proposed moment estimators and the empirical variogram estimator.

The $D$ values based on the clipped moments estimator and the empirical variogram estimator, as a function of $k/n\in\{0.05,0.10,0.15,0.20,0.25\}$, are presented in the left panel of \cref{fig:Errvsp_flight}. The right panel shows scatter plots comparing the model-based tail dependence coefficients obtained from the three estimators with the empirical tail dependence coefficients, where $k/n=0.1$ and $a=0.25$. We observe that, in this setting, the clipped moment estimators exhibit greater stability with respect to the choice of $k$. In particular, the improvement of the first-order clipped moment estimator $\hat{\gamma}_{n,ij}^{\mathrm{M},(1)}(k,-\log a)$ is more pronounced. Compared with the Danube dataset, the dependence structure in this dataset appears to be weaker. This leads to a more evident improvement of the first-order clipped moment estimator.

\begin{figure}
\centering
\includegraphics[width=0.45\textwidth]{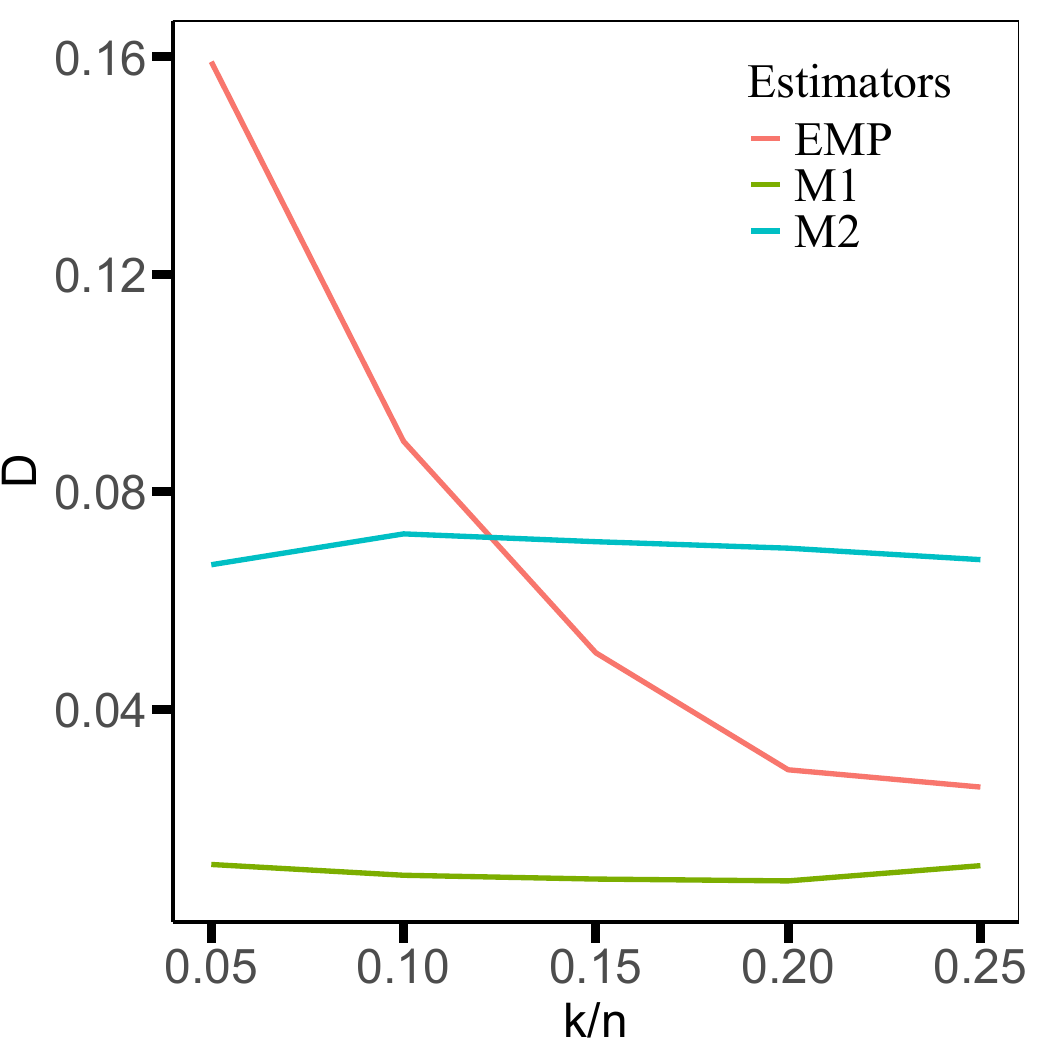}
\includegraphics[width=0.45\textwidth]{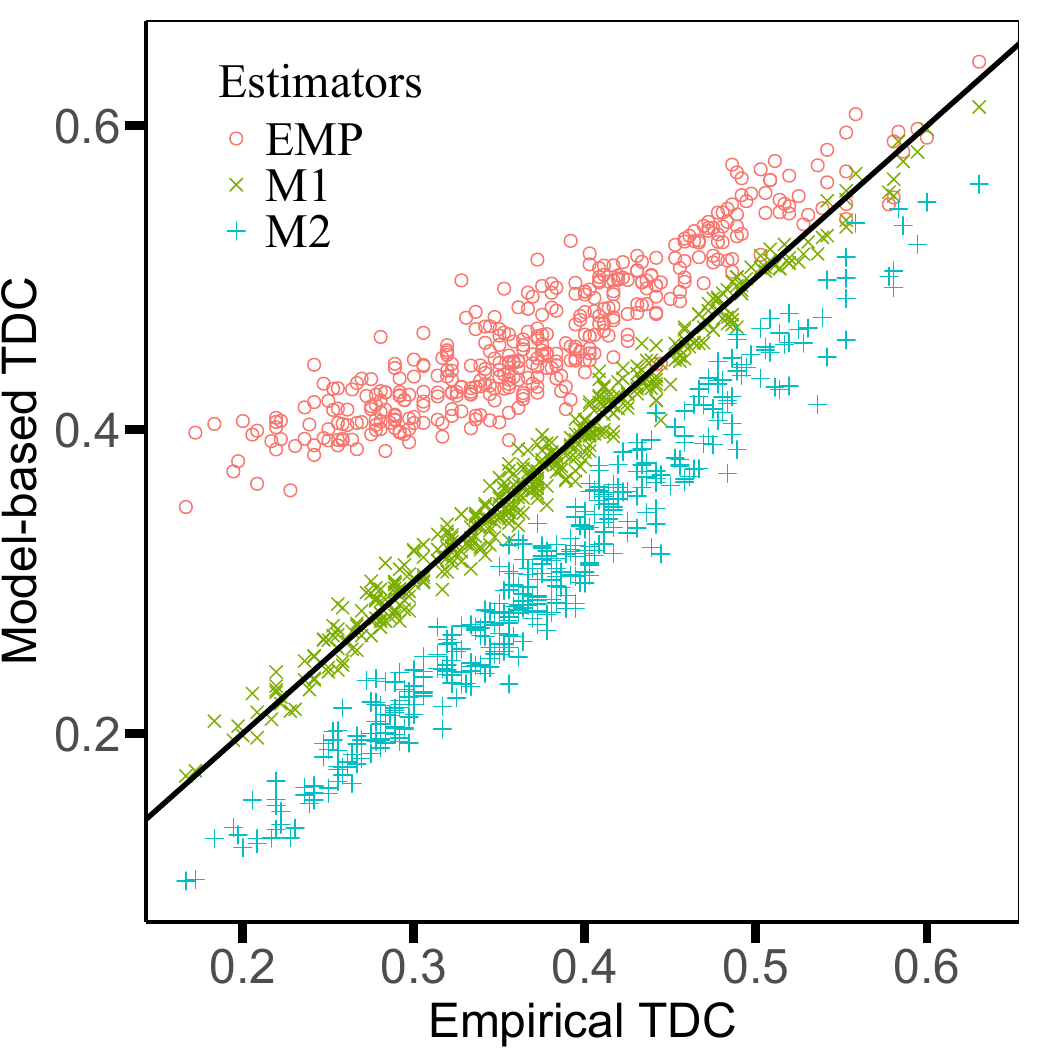}
\caption{
The $D$ values (left) and scatter plot (right) of empirical and model-based tail dependence coefficients (TDC) for the flight delay data, computed using the empirical variogram estimator $\hat{\gamma}_{n,ij}^{\mathrm{EMP}}$ (EMP),  the moment estimator $\hat{\gamma}_{n,ij}^{\mathrm{M},(1)}$ (M1) and $\hat{\gamma}_{n,ij}^{\mathrm{M},(2)}$ (M2), with clipping level $a=0.25$, ratio $k/n\in\{0.05,0.10,0.15,0.20,0.25\}$ for the line plot and $k/n=0.10$ for the scatter plot.}
\label{fig:Errvsp_flight}
\end{figure}

\section{Concluding remarks}
\label{sec:conclusion}

The empirical variogram estimator of~\citet{EV20} is the natural estimator of the variogram matrix of a \HR{} MGPD and is widely used in practice. It relies, however, on the limiting MGPD being an accurate approximation to the joint tail of the data. When the tail dependence between some of the components is weak, this approximation deteriorates: an observation that is extreme in some of the components is then likely to be non-extreme in the weakly dependent ones, and the resulting bias can be substantial. Motivated by this, we have proposed first- and second-order moment estimators of the elements of the variogram matrix, constructed from a lower-tail-clipped version of the standardized variables, and we have established their consistency and asymptotic normality.

The simulation study and the two case studies point in the same direction. The clipped moment estimators trade variance for bias: the empirical variogram estimator attains the smaller variance throughout, but the larger bias, and the moment estimators are the more accurate ones precisely in the regime that motivated them, that is, when the dependence is weak. Among the estimators considered, the first-order moment estimator with clipping level $a=0.25$ comes out best overall, having the lowest error in almost all cases in terms of both $L$ and $D$. Its advantage is most pronounced for small $k$, and both moment estimators are less sensitive to the choice of $k$ than the empirical variogram estimator. The applications confirm the picture: for the US flight delay data, whose dependence structure is weaker than that of the Danube discharge data, the improvement brought by the first-order clipped moment estimator is the more evident of the two.

The choice of the moment order $\ell$ and of the clipping level $c$ is left open. It is tempting to select the two by minimizing the asymptotic variance of \cref{thm:AsyNormalityMoment}, for which \cref{sec:AsyVar} provides an explicit expression. This route, however, is not viable: the asymptotic variance is minimal for $\ell=1$ and in the limit $c \to \infty$, that is, $a \to 0$, which amounts to no clipping at all. The purpose of clipping is to reduce the bias, so any meaningful choice of $\ell$ and $c$ must weigh bias against variance. We have no theoretical expression for the bias of the moment estimators, nor a way of estimating it, and the trade-off is therefore out of reach. This is why we have fixed $a=0.25$ throughout, guided by the simulation study rather than by theory. A tractable handle on the bias appears to us to be the key to any principled, data-driven choice of $\ell$ and $c$.

Finally, one may ask whether removing the clipping recovers the empirical variogram estimator. It does not. The moment functions $e^{(\ell)}(\gamma,c)$ in \eqref{eq:etildeFirst} and \eqref{eq:etildeSecond} diverge as $c \to \infty$, on account of the shift by $c$, so that the limit cannot be taken directly in \eqref{eq:M1}. The clipped variable itself does converge, $Y_i^{(j)} \vee (-c) \to Y_i^{(j)}$, and its first two moments are
\[
    \E \left( Y_i^{(j)} \right) = 1 - \frac{\gamma_{ij}}{2},
    \qquad
    \E \left\{ \left( Y_i^{(j)} \right)^{2} \right\} = \frac{1}{4} \gamma_{ij}^{2} + 2 .
\]
Equating these to their empirical counterparts yields moment estimators of $\gamma_{ij}$ based on the unclipped variables. These are not the empirical variogram estimator: the latter is built on the variance of the differences $Y_i^{(m)} - Y_j^{(m)}$, whereas the former rest on the moments of $Y_i^{(j)}$ alone. The clipping level thus does not interpolate between the moment estimators proposed here and the empirical variogram estimator; the two constructions remain distinct.

\section*{Statements and Declarations}

\paragraph{Competing interests.}
The authors declare no competing interests.

\paragraph{Author contributions.}
Both authors contributed equally to the ideas and to the theory. The simulation study and the case study are largely the work of Shuang Hu.

\paragraph{Data availability.}
Both datasets analyzed in this paper are publicly available. The Danube discharge data are available in the supplementary material of~\citet{ADE2015}. The US flight delay data are publicly available from the U.S. Bureau of Transportation Statistics; the preprocessed version used here can be obtained through the \texttt{R} package \texttt{graphicalExtremes}~\citep{HES2024}.

The code used for the simulation studies and data analysis is available at
\url{https://github.com/HUSHHuShuang/HR_variogram_estimation_by_clipped_moments}.

\bibliography{reference}

\newpage

\appendix
\section{Lemmas and proofs of theoretical results}
\label{sec:Appendix}

In this section, we give the proofs of all the propositions, theorems and lemmas. We first introduce some notation that will be used in the following proofs.

Recall that for the random vector $\bm{X}=(X_i,i\in V)$ with continuous marginal distribution functions $F_i(x)$, $i\in V$, the random vector $\bm{U}=(U_i,i\in v)$ is defined by $U_i=1-F_i(X_i)$ for $V=\{1,\ldots,d\}$.
For $i\in V$ and $t=1,\ldots,n$, let
\[
\bm{U}_{t}=(U_{t1},\ldots,U_{td}), \   U_{ti}=1-F_i(X_{ti}).
\]
For non-empty $I\subset V$, denote
\begin{align*}
    T_{n,I}(\bm{x}_{I})
    &= \frac{1}{k} \sum_{t=1}^{n} \I\left\{U_{ti}\le k x_{i}/n, \  i\in I \right\}, \\
    R_{n,I}(\bm{x}_{I})
    &= \frac{n}{k} C_{I}\left(\frac{kx_{i}}{n}, \ i\in I\right).
\end{align*}
Define the tail process on $[0,\infty]^{|I|}\setminus \{(\infty,\ldots,\infty)\}$ as
\begin{align}
\label{eq:nu}
\nu_{n,I}(\bm{x}_I)=\sqrt{k} \left(T_{n,I}(\bm{x}_{I})-R_{n,I}(\bm{x}_{I})\right).
\end{align}

Denote the empirical distribution function of $U_{ti}$, $t=1,\ldots, n$ by
\[
\hat{F}_{U,i}(x)=\sum_{t=1}^n \I\{U_{ti}\le x\}/(n+1).
\]
For every $s\in \{1,\ldots,n\}$, we have
\begin{align*}
    \hat{F}_{U,i}(U_{si})
    &= \frac{1}{n+1}\sum_{t=1}^n \I\{1-F_i(X_{ti})\le 1-F_i(X_{si})\} \\
    &= 1-\frac{1}{n+1}\sum_{t=1}^n \I\{X_{ti}\le X_{si}\}
    =1-\hat{F}_{i}(X_{si})
\end{align*}
almost surely. Hence $\hat{U}_{ti}=\hat{F}_{U,i}(U_{ti})$.
Let
\begin{equation}
\label{eq:InverseDF}
\hat{F}_{U,i}^{\leftarrow}(x)=\inf \{y \in \mathbb{R} : \hat{F}_{U,i}(y) \ge x\}
\end{equation}
denote the generalized inverse function of $\hat{F}_{U,i}$. For $I\subset V$ and $\bm{x}_I\in [0,\infty]^{|I|}\setminus{(\infty,\ldots,\infty)}$, set
\begin{align}
\label{eq:EmpiricalTailCopula}
    \tilde{R}_{n,I}(\bm{x}_I)
    &=\frac{1}{k} \sum_{t=1}^{n} \I\left\{\hat{U}_{ti}\le k x_{i}/n,i\in I\right\}
    \end{align}
and
\begin{align}
\label{eq:UTDFEst}
    \hat{R}_{n,I}(\bm{x}_I)
    &=\frac{1}{k} \sum_{t=1}^{n} \I\left\{U_{ti}\le \hat{F}_{U,i}^{\leftarrow} (k x_{i}/n),i\in I\right\}.
\end{align}
For a subset $I \subset V$, suppose that $1/k \le x_i \le n/k$ for all $i \in I$.
Let
\begin{align}
\label{eq:Tailprocess}
\hat{\nu}_{n,I}(\bm{x}_I)=\sqrt{k} \left(\hat{R}_{n,I}(\bm{x}_I)-R_{I}(\bm{x}_I)\right)
\end{align}
with $\hat{R}_{n,I}(\bm{x}_I)$ and $R_{I}(\bm{x}_I)$ given by~\eqref{eq:UTDFEst} and~\eqref{eq:UTDF}.

\subsection{Lemmas}
We first present some lemmas that will be used in the proofs of the propositions and the theorems in this section.

\begin{lemma}[{\bf Existence of clipped moments for MGPD}]
\label{lemma:ExistTruVario}
Let $\bm{Y}=(Y_v,v\in V)$ be a standard MGP distributed random vector, and $c$ be a real number that $0\le c<\infty$. For any $i,m\in V$, we have $\E \left\{(Y_{i}^{(m)} + c)_{+}^{2}\right\}<\infty$.
\end{lemma}

\begin{proof}
Assume $\tilde{\bm{X}}=(\tilde{X}_i,i\in V)$ is a random vector with standard exponential margins and it belongs to the domain of attraction of $\bm Y$. Note that
\begin{align*}
\pr(\tilde{X}_i>y+\tilde{X}_m\mid \tilde{X}_m>u)
=\frac{\pr(\tilde{X}_i>y+\tilde{X}_m,\tilde{X}_m>u)}
{\pr(\tilde{X}_m>u)}
\leq \frac{\pr(\tilde{X}_i>y+u)}
{\pr(\tilde{X}_m>u)} \le \exp(-y)
\end{align*}
by the standard exponential margins. Hence from \eqref{eq:MPD} and by letting $u\to \infty$ on the left-hand side of the last inequality, we have
\[
\pr(Y_{i}^{(m)}>y)<\exp(-y),\  y> 0.
\]
Therefore,
\begin{align*}
    &\E \left[\left\{Y_{i}^{(m)} \vee (-c)\right\}^2\right] \\
    &\quad =\E \left[\left\{ Y_{i}^{(m)} \vee (-c)\right\}^2 \I\{ Y^{(m)}_{i}<0\}\right]+\E \left[\left\{Y_{i}^{(m)} \vee (-c)\right\}^2 \I\{Y^{(m)}_{i}\ge 0\}\right]\\
    &\quad \le c^2 +\int_{0}^{\infty} \pr\left\{(Y_{i}^{(m)})^2 \I\{Y^{(m)}_{i}\ge 0\}>u\right\} \d u\\
    &\quad = c^2 + 2\int_{0}^{\infty} u \pr(Y_{i}^{(m)}>u)  \d u \\
    &\quad \le c^2 + 2\int_{0}^{\infty} u \exp(-u)  \d u  = c^2 + 2<\infty.
\end{align*}
The desired result follows from the decomposition that
\begin{align*}
\E \left\{(Y_{i}^{(m)} + c)_{+}^{2}\right\}& = \E \left[\left\{ (Y_{i}^{(m)} \vee (-c)) + c \right\}^{2}\right]\\
& =\E \left\{(Y_{i}^{(m)} \vee(-c))^2\right\} + 2 c \E \left\{ Y_{i}^{(m)} \vee (-c)\right\} + c^2 < \infty.
\end{align*}
The proof is complete.
\end{proof}

To establish the asymptotic normality of the moment estimators, we need a convergence result for the weighted tail empirical process. A similar result has already been discussed in Proposition~3.1 of~\citet{EDL06}. The difference from their work is that we use a different weighting function $(x\wedge y)^\eta$, while they weighted the process by $(x\vee y)^\eta$. However, our result can be derived through their proof approach with minor modifications. Another related result is Lemma 1 from~\citet{CEDZ2015}. In that lemma, the weight function of the bivariate empirical tail process is $x^\eta$. In our case, the weight function $(x\wedge y)^{\eta}$ can be decomposed into two cases: $x>y$ and $x\le y$. By analyzing these cases separately, our setting can be reduced to the one considered in the cited lemma, and hence, similar conclusions can be derived.
However, the result in their work was stated without a formal proof. For completeness, we present the proof in the following lemma. We note that this lemma applies to general MGPDs and is not restricted to the \HR{} MGPD.

\begin{lemma}
\label{lemma:WTP}
Let $\bY$ be an arbitrary MGP distributed random vector, and let $\bm{X}_{t}=(X_{t1},\ldots,X_{td})$, $t=1,\ldots,n$, denote the independent copies of a random vector $\bm{X}$, which has continuous margins and such that the marginally transformed vector $\tilde{\bm X}$ in \eqref{eq:marginal_Transform} lies in the domain of attraction of $\bY$. If \cref{Asp:IntermediateSeq} holds, then for $T>0$ and $0\le \eta < 1/2$, the process
\begin{align*}
    &\Bigg(\frac{\nu_{n,ij}(x,y)}{(x\wedge y)^{\eta}}, (x,y)\in (0,T]^2,\;  \frac{\nu_{n,s}(u)}{u^{\eta}}, u\in (0,T], \; i,j\in V, \; i< j, \; s\in V,\bigg)
\end{align*}
converges jointly in distribution to
\begin{align*}
    \Bigg(\frac{W_{ij}(x,y)}{(x \wedge y)^{\eta}}, (x,y)\in (0,T]^2,\; \frac{W_{s}(u)}{u^{\eta}}, u\in (0,T], \; i,j\in V, \; i< j, \; s\in V\Bigg)
\end{align*}
as $n\to\infty$ in $(\ell^{\infty}((0, T]^2))^{d(d-1)/2} \times (\ell^{\infty}((0, T]))^{d}$.
\end{lemma}

\begin{proof}
For nonempty subset $I\subset V$ with $|I|\le 2$ and $\bm{x}_{I}=(x_i,i\in I)\in (0,T]^{|I|}$, define
\[
f_{I}(\bm{x}_I):=f_{I}(\bm{x}_I,\cdot)=\frac{\I_{(\bm{0},\bm{x}_I]}(\cdot)}{(\min_{i\in I} x_i)^{\eta}}
\]
as a function on $[0,\infty]^d\setminus\{(\infty,\ldots,\infty)\}$, where $(\bm{0},\bm{x}]$ denotes the Cartesian product $\prod_{i\in V}(0,y_{i}]$ with $y_i=x_i$ if $i\in I$ and $y_i=\infty$ if $i\in V\setminus{I}$. That is, for any $\bm{z}\in [0,\infty]^d\setminus\{(\infty,\ldots,\infty)\}$,
\[
f_{I}(\bm{x}_I)(\bm{z})=f_{I}(\bm{x}_I,\bm{z})=\frac{\I\{z_i\le x_i,i\in I\}}{(\min_{i\in I} x_i)^{\eta}}.
\]
Denote the class of $f_{I}(\bm{x}_I)$ by $\mathcal{F}$, i.e.,
\[
\mathcal{F}=\{f_{I}(\bm{x}_I), \; I\subset V, \; 1\le |I|\le 2, \; \bm{x}_{I}=(x_i,i\in I)\in (0,T]^{|I|}\}.
\]
For $I,J\subset V$ and $f_{I}(\bm{x}_I),f_{J}(\bm{y}_J)\in \mathcal{F}$, define a semi-metric $\rho$ on $\mathcal{F}$ by
\begin{align*}
    \rho (f_{I}(\bm{x}_I),f_{J}(\bm{y}_{J}))=\left[\E\left\{\left(\frac{W_{I}(\bm{x}_{I})}{(\min_{i\in I} x_i)^{\eta}}-\frac{W_{J}(\bm{y}_{J})}{(\min_{i\in J} y_i)^{\eta}} \right)^2\right\}\right]^{1/2}.
\end{align*}
Let
\[
Z_{n,t}=\frac{1}{\sqrt{k}} \delta_{\left(\frac{n}{k}U_{ti},\;i\in V\right)}
\]
be the scaled Dirac measure on $[0,\infty]^d\setminus\{(\infty,\ldots,\infty)\}$.
For any $f\in \mathcal{F}$, let
\[
Z_{n,t}(f)=\int f \d Z_{n,t}.
\]
Then we have
\[
Z_{n,t}(f_{I}(\bm{x}_{I}))=\frac{1}{\sqrt{k}} \frac{\I\left\{U_{ti}<kx_{i}/n, \; i\in I\right\}}{(\min_{i\in I}x_{i})^{\eta}}
\]
and
\[
\frac{\nu_{n,I}(\bm{x}_{I})}{(\min_{i\in I}x_{i})^{\eta}}=\sum_{t=1}^{n} \left\{Z_{n,t}(f_{I}(\bm{x}_I))-\E \left(Z_{n,t}(f_{I}(\bm{x}_I))\right)\right\}.
\]

Note that for each $t$, $Z_{n,t}$ can be viewed as a stochastic process indexed by functions $f_I(\bm{x}_I)\in \mathcal{F}$. For independent stochastic processes $Z_{n,t}$, $t=1,\ldots,n$, indexed by a totally bounded semi-metric space $(\mathcal{F},\rho)$, the weak convergence of $\sum_{t=1}^n (Z_{n,t}-\E Z_{n,t})$ is implied by the weak convergence of its finite-dimensional distributions and its asymptotic tightness. Note that the finite-dimensional convergence follows from the Cram\'er--Wold device together with the univariate Lindeberg--Feller central limit theorem, since convergence of all linear combinations implies convergence of multivariate distributions. Hence it suffices to establish asymptotic tightness. To this end, let $\{\mathcal{F}_{\varepsilon j}\}_{i=1}^{N_{\varepsilon}}$ be a partition of $\mathcal{F}$ such that  $\mathcal{F}=\bigcup_{j=1}^{N_{\varepsilon}} \mathcal{F}_{\varepsilon j}$ and
\begin{align}
\label{eq:TightFiniteMoment}
    \sum_{t=1}^{n} \E^{\ast} \left(\sup_{f,g\in \mathcal{F}_{\varepsilon j}} \big| Z_{n,t}(f)-Z_{n,t}(g)\big|^2\right) \le \varepsilon^2
\end{align}
for each $j=1,,\ldots,N_{\varepsilon}$, where $\E^{\ast}$ is the outer integral (for definition see \citet{VW96}). For $\varepsilon>0$, let $N_{[\;]}(\varepsilon,\mathcal{F},\rho)$ be the minimal number of $\varepsilon$-brackets (with respect to $\rho$) required to cover $\mathcal{F}$. By Theorem 2.11.9 in \citet{VW96}, it is sufficient to verify:
\begin{equation}
\label{eq:SumMoments}
\sum_{t=1}^{n}  {\E}^{\ast} \left[ \| Z_{n,t} \|_{\mathcal{F}} \I(\Vert Z_{n,t} \Vert_{\mathcal{F}}>\lambda) \right] \to 0
\end{equation}
as $n\to\infty$ for every $\lambda>0$, where $\Vert Z_{n,t}\Vert_{\mathcal{F}}=\sup_{f\in\mathcal{F}}|Z_{n,t}(f)|$, and
\begin{equation}
\label{eq:entropy}
\int_{0}^{\delta_n} \sqrt{\log N_{[\;]}(\varepsilon,\mathcal{F},\rho)} \d \varepsilon \to 0
\end{equation}
for $\delta_n\to 0$. Together with~\eqref{eq:TightFiniteMoment}, this implies the tightness of $\sum_{t=1}^n (Z_{n,t}-\E( Z_{n,t}))$.

To validate \eqref{eq:SumMoments}, we first show that $(\mathcal{F},\rho)$ is a bounded space. Given $I\subset V$, let $\mathcal{F}_{I}=\{f_{I}(\bm{x}_{I}), \bm{x}_I\in (0,T]^{|I|}\}$. Recall that $f_I$ are defined for nonempty subsets $I\subset V$ and $|I|\le 2$. Since there are only finitely many such subsets, the totally boundedness of $(\mathcal{F},\rho)$ is equivalent to the totally boundedness of all the subclass $\mathcal{F}_{I}$, $I\subset V$ with $|I|\le 2$. For brevity, we only show that $\mathcal{F}_{I}$ is totally bounded for $I=\{i,j\}$ where $|I|=2$. The remaining cases are similar.

For every $f_{ij}(x,y)\in \mathcal{F}_I$, we show that the mapping: $(x,y) \mapsto f_{ij}(x,y)$ is uniformly continuous from $(0,T]^2$ to $(\mathcal{F}_I,\rho)$. Since $(0,T]^2$ is totally bounded under the Euclidean metric, this will imply total boundedness of $(\mathcal{F}_I,\rho)$. For $x,y,u,v\in (0,T]$, there are four possible orderings of $(x,y)$ and $(u,v)$. Without loss of generality, let $x\ge y$, $u\ge v$, and $x\ge u$, $y\ge v$. For any $\delta>0$, assuming $|x-u|\le \delta$ and $|y-v|\le \delta$. Then we have
\begin{align*}
    \rho^{2}(f_{ij}(x,y),f_{ij}(u,v))=&\E \left\{\left(\frac{W_{ij}(x,y)}{(x\wedge y)^{\eta}}-\frac{W_{ij}(u,v)}{(u\wedge v)^{\eta}}\right)^2\right\}\\
    =&\E \left\{\left(\frac{W_{ij}(x,y)}{y^{\eta}}-\frac{W_{ij}(u,v)}{v^{\eta}}\right)^2\right\}\\
    =&\E \left\{\left(\frac{v^{\eta} W_{ij}(x,y)-y^{\eta} W_{ij}(u,v)}{(yv)^{\eta}}\right)^2\right\}\\
    =&\frac{v^{2\eta} R_{ij}(x,y)-2 v^{\eta} y^{\eta} R_{ij}(u,v) + y^{2\eta} R_{ij}(u,v)}{(yv)^{2\eta}}.
\end{align*}
Note that $R_{ij}(x,y)\le x\wedge y$. If $v\le \delta$, we have
\begin{align*}
    \rho^{2}(f_{ij}(x,y),f_{ij}(u,v))\le &\frac{R_{ij}(x,y)}{y^{2\eta}}+\frac{3R_{ij}(u,v)}{v^{2\eta}}\\
    \le & y^{1-2\eta} +3 v^{1-2\eta} \le  (2\delta)^{1-2\eta}+3\delta^{1-2\eta} \le 5 \delta^{1-2\eta}.
\end{align*}
Otherwise, for $v>\delta$, by noting that
\[
|R_{ij}(x,y)-R_{ij}(u,v)|\le |x-u|+|y-v| \le 2\delta,
\]
we obtain
\begin{align*}
    \rho^{2}(f_{ij}(x,y),f_{ij}(u,v))
    &= \frac{R_{ij}(u,v)(y^{\eta}-v^{\eta})^2}{(yv)^{2\eta}}+\frac{v^{2\eta}[R_{ij}(x,y)-R_{ij}(u,v)]}{(yv)^{2\eta}}\\
    &\le \frac{R_{ij}(u,v)(y^{\eta}-v^{\eta})^2}{(yv)^{2\eta}}+\frac{2\delta v ^{2\eta}}{(yv)^{2\eta}}\\
    &\le  v^{1-4\eta} (y^{\eta}-v^{\eta})^2 +2 \delta^{1-2\eta}\\
    &\le  \eta^2 v^{1-4\eta} v^{2\eta-2} (y-v)^2 +2 \delta^{1-2\eta}\\
     &\le v^{-1-2\eta} (y-v)^2+2 \delta^{1-2\eta} \le 3 \delta^{1-2\eta},
\end{align*}
where the fourth step follows from the mean value theorem. Hence, for every $\varepsilon>0$, there exists $\delta>0$ such that for all $|x-u|\le \delta$ and $|y-v|\le \delta$, we have $\rho^{2}(f_{ij}(x,y),f_{ij}(u,v))<\varepsilon$. This completes the proof for the case $x\ge y$, $u\ge v$, and $x\ge u$, $y\ge v$. The other three cases can be followed by symmetry. Consequently, the map $(x,y) \mapsto f_{ij}(x,y)$ is uniformly continuous from $(0,T]^2$ to $(\mathcal{F}_I,\rho)$.
Since there are finitely many $I \subset V$ with $|I|\le 2$, the whole class $(\mathcal{F},\rho)$ is totally bounded.

Now we show that~\eqref{eq:SumMoments} holds for every $\lambda>0$.
For $I\subset V$ and $\bm{x}_I\in (0,T]^{|I|}$, assume ${i_0}=\{i:x_i=\min_{i\in I} x_i\}$. Since
\begin{align*}
\sup_{f_{I}(\bm{x}_I)\in \mathcal{F}} |Z_{n,t}(f_{I}(\bm{x}_I))|
&= \sup_{f_{I}(\bm{x}_I)\in \mathcal{F}} \frac{\I\left\{U_{ti}<kx_i/n,\; i\in I\right\}}{\sqrt{k} (\min_{i\in I} x_{i})^{\eta}} \\
&\le  \frac{1}{\sqrt{k}} \frac{\I(U_{ti_0}\le \frac{k}{n}x_{i_0})}{x_{i_0}^{\eta}}
\le \frac{1}{\sqrt{k}} \frac{1}{\left(\frac{n}{k} U_{t i_0}\right)^{\eta}},
\end{align*}
we have for each $\lambda>0$ and $0\le \eta < 1/2$,
\begin{align*}
    &\sum_{t=1}^{n} \E^{\ast} \left\{\Vert Z_{n,t} \Vert_{\F} \I\left(\Vert Z_{n,t} \Vert_{\F}>\lambda\right)\right\}\\
&\quad \le \frac{n}{\sqrt{k}} \E \left[ \left(\frac{n}{k} U_{t i_0}\right)^{-\eta} \I\left\{\frac{1}{\sqrt{k}} \frac{1}{\left(\frac{n}{k} U_{ti_0}\right)^{\eta}}>\lambda\right\}\right]\\
&\quad =  \frac{n}{\sqrt{k}} \left(\frac{n}{k}\right)^{-\eta} \int_{0}^{(k/n)(\sqrt{k}\lambda)^{-1/\eta}} x^{-\eta} \d x \\
&\quad =  \frac{1}{1-\eta} \lambda^{1-1/\eta} k^{1-1/(2\eta)}\to 0
\end{align*}
as $n\to\infty$. This establishes the validity of~\eqref{eq:SumMoments}.

It remains to show that~\eqref{eq:entropy} is satisfied for every sequence $\delta_{n}\to 0$. Although the original goal is to establish the result on $[0,T]$, we follow~\citet{EDL06} and take $T=1$ for simplicity. The general case follows analogously. Let $\mathcal{F}^{j}=\{f_I(\bm{x}_I), \, I\subset{V}, \, |I|=j\}$. It suffices to verify~\eqref{eq:entropy} for $j=1,2$. We only treat the case $j=2$, the case $j=1$ being similar.
Fix $\varepsilon>0$ sufficiently small. Set $\alpha=\varepsilon^{3/(1-2\eta)}$ and $\theta=1-\varepsilon^3$. Define
\begin{align*}
\mathcal{F}(\alpha)=& \left\{f_{I}(x,y)\in \mathcal{F}^{2}: x\wedge y \le \alpha\right\},
\end{align*}
and for intergers $r,s\ge 0$,
\begin{align*}
\mathcal{F}(r,s)=&\left\{f_{I}(x,y)\in \mathcal{F}^{2}: \theta^{r+1}\le x \le \theta^{r}, \; \theta^{s+1}\le y \le \theta^{s}\right\}.
\end{align*}
Then
\begin{align*}
\mathcal{F}^2= \mathcal{F}(\alpha) \cup \left(\bigcup_{r,s=0}^{[\log \alpha/ \log \theta]} \mathcal{F}(r,s)\right).
\end{align*}

We first verify~\eqref{eq:entropy} for the class $\mathcal{F}(\alpha)$. By the definition of $Z_{n,t}$,
\begin{align*}
& \sum_{t=1}^{n} \E^{\ast} \left\{\sup_{f,g\in \mathcal{F}(\alpha)} \left(Z_{n,t}(f)-Z_{n,t}(g)\right)^{2}\right\}
= n \E \left\{\sup_{f,g\in \mathcal{F}(\alpha)} \left(Z_{n,1}(f)-Z_{n,1}(g)\right)^{2}\right\}
\\
& \quad \le  4n \E  \left\{\sup_{ f \in \mathcal{F} (\alpha)} Z_{n,1}^{2}(f)\right\}
= \frac{4n}{k} \E \left\{\sup_{ \substack{x,y>0
\\
x\wedge y \le \alpha}} \frac{\I\left(U_{1i}<\frac{kx}{n}, U_{1j}<\frac{ky}{n} \right)}{(x\wedge y)^{2\eta}} \right\}
\\
& \quad \le \frac{4n}{k}  \E \left\{\left(\frac{n}{k} U_{1i}\right)^{-2\eta} \I \left(\frac{n}{k} U_{1i}<\alpha \right)\right\}
\\
&= \quad \frac{4n}{k} \int_{0}^{\alpha k/n} \left(\frac{n}{k}x\right)^{-2\eta} \d x
= \frac{4 \alpha^{1-2\eta}}{1-2\eta}  \le \varepsilon^3.
\end{align*}
Thus the contribution of $\mathcal{F}(\alpha)$ is negligible.
Now we establish~\eqref{eq:entropy} for the classes $\{\mathcal{F}(r,s), 0\le r,s\le [\log \alpha/ \log \theta]\}$. For given $r,s$, without loss of generality, we assume $r=r \wedge s$. Then
\begin{align*}
& \sum_{t=1}^{n} \E^{\star} \left\{\sup_{f,g\in \mathcal{F}(r,s)} \left(Z_{n,t}(f)-Z_{n,t}(g)\right)^2\right\}\\
& \quad \le n \E \left\{ \left(\sup_{f\in\mathcal{F}(r,s)} Z_{n,1}(f)-\inf_{f\in\mathcal{F}(r,s)} Z_{n,1}(f)\right)^2 \right\}\\
& \quad \le  \frac{n}{k} \E\left\{\left(\frac{\I\left\{U_{1i}\le \frac{k}{n} \theta^{r}, U_{1j}\le \frac{k}{n} \theta^{s}  \right\}}{(\theta^{r+1}\wedge \theta^{s+1} )^{\eta}} - \frac{\I\left\{U_{1i}\le \frac{k}{n} \theta^{r+1}, U_{1j}\le \frac{k}{n} \theta^{s+1}\right\}}{(\theta^{r}\wedge \theta^{s} )^{\eta}} \right)\right\}^2\\
&\quad = \frac{n}{k} \E\Bigg[\left\{\I\left\{U_{1i}\le \frac{k}{n} \theta^{r}, U_{1j}\le \frac{k}{n} \theta^{s}\right\} \left(\theta^{-\eta(s+1)}-\theta^{-\eta s}\right)\right.\\
&\qquad \qquad  \left.+\left\{\I\left\{U_{1i}\le \frac{k}{n} \theta^{r}, U_{1j}\le \frac{k}{n} \theta^{s} \right\}-\I\left\{U_{1i}\le \frac{k}{n} \theta^{r+1}, U_{1j}\le \frac{k}{n} \theta^{s+1} \right\}\right\} \theta^{-\eta s}\right\}^2\Bigg]\\
& \quad \le \frac{2n}{k} \left[C_{ij}\left(\frac{k}{n}\theta^{r}, \frac{k}{n} \theta^{s}\right) \frac{\left(\theta^{-\eta}-1\right)^2}{\theta^{2\eta s}}  \right. \\
& \qquad \qquad + \left.\left\{C_{ij}\left(\frac{k}{n} \theta^{r},\frac{k}{n} \theta^{s} \right)-C_{ij}\left(\frac{k}{n} \theta^{r+1},\frac{k}{n} \theta^{s+1}\right)\right\}\theta^{-2\eta s}\right]\\
& \quad \le   2 \theta^{(1-2\eta)s} \left\{(\theta^{-\eta}-1)^2+2(1-\theta)\right\}
\le  2 \left\{(\theta^{-1/2}-1)^2+2(1-\theta)\right\} \le \varepsilon^6 +4 \varepsilon^3 \le 5 \varepsilon^3,
\end{align*}
where $(\theta^{-\eta}-1)^2\le \varepsilon^6$ holds for sufficiently small $\varepsilon$. Thus, the number of elements of the partition of $\mathcal{F}^{2}$ is bounded by $C \varepsilon^{-6} \log(1/\varepsilon)^2$ for some constant $C>0$. Consequently, \eqref{eq:entropy} holds as $n\to \infty$. This completes the proof of the lemma.
\end{proof}

Recall that $W(\bm{x})$ is a mean-zero Gaussian process defined in~\cref{sec:Results} with continuous trajectories.
Since the weighted empirical tail process is asymptotically tight, by Theorems 1.5.7 and 1.5.8 of \citet{VW96}, the processes take values in a separable subspace of  $(\ell^{\infty}((0, T]^2))^{d(d-1)/2} \times (\ell^{\infty}((0, T]))^{d}$.
Hence, by the Skorokhod's representation theorem (cf. Theorem 1.10.4 in \citet{VW96}), there exists a probability space with processes $\nu^{\ast}_{n,ij}(x,y)$, $\nu^{\ast}_{n,s}(u)$
and $W^{\ast}_{ij}(x,y)$, $W^{\ast}_{s}(u)$
for all $i,j,s\in V$ and $i< j$ such that for $x,y,u\in (0,T]$,
\begin{align*}
    \left(\nu^{\ast}_{n,ij}(x,y),\nu^{\ast}_{n,s}(u), \; i,j,s\in V, i< j\right) &\overset{d} = \left(\nu_{n,ij}(x,y),\nu_{n,s}(u),\; i,j,s\in V, \; i<j\right),\\
    \left( W_{ij}^{\ast}(x,y), W_{s}^{\ast}(u),\; i,j,s\in V, i<j\right) & \overset{d} =(W_{ij}(x,y), W_{s}(u),\; i,j,s\in V,\; i<j),
\end{align*}
and the convergence in \cref{lemma:WTP} holds almost surely.
Then, for $0\le \eta<1/2$, we have
\begin{align}
\label{eq:uniformconverge}
\nonumber
    \sup_{0<x,y\le T} \frac{\big|\nu^{\ast}_{n,ij}(x,y)- W_{ij}^{\ast}(x,y)\big|}{(x\wedge y)^{\eta}} =o_{p}(1),\\
    \sup_{0<u\le T} \frac{\big|\nu^{\ast}_{n,s}(u)- W_{s}^{\ast}(u)\big|}{u^{\eta}} =o_{p}(1).
\end{align}
We shall work within this probability space while keeping the notation unchanged. Next, we establish an inequality analogous to Lemma 3.2 in \citet{EDL06}.

\begin{lemma}
\label{lemma:inequalityadd}
Let $0\le \eta <1/2$. For any $i,j\in V$ and sufficiently small $\varepsilon \in (0,1)$, we have
\[
\pr\left(\sup_{ \substack{0<x\le \varepsilon}} \frac{|W_{ij}(x,1)|}{x^{\eta}} \ge \lambda\right)\le 4 \sum_{l=0}^{\infty} \exp\left(-\frac{\lambda^2}{2^{1+2\eta}} \frac{2^{l(1-2\eta)}}{\varepsilon^{1-2\eta}}\right).
\]
\end{lemma}

\begin{proof}
Recall that for any $i,j\in V$, $\{W_{ij}(x,1)$, $x\in [0,\infty]\}$ defined in~\eqref{eq:covW} is a centered Gaussian process. We firstly show that, for $0\le a<b$,
\begin{align}
\label{eq:GaussianExtremes}
\pr\left\{\sup_{x\in [a,b] } |W_{ij}(x,1)|\ge \lambda \right\} \le 2 \pr\left\{ |W_{ij}(b,1)|\ge \lambda \right\}.
\end{align}
To prove the result, let $\mathcal{F}_x=\sigma\{W_{ij}(s,1):0\leq s\leq x\}$.
For $0\leq s\leq x\leq y$, we have
\[
\begin{aligned}
&\cov\{W_{ij}(y,1)-W_{ij}(x,1),W_{ij}(s,1)\} \\
&=\cov \{W_{ij}(y,1),W_{ij}(s,1)\}-\cov\{W_{ij}(x,1),W_{ij}(s,1)\}\\
&=R(y\wedge s,1)-R(x\wedge s,1)\\
&=
R(s,1)-R(s,1)=0.
\end{aligned}
\]
Since the process is Gaussian, the increment $W_{ij}(y,1)-W_{ij}(x,1)$ is independent of $\mathcal{F}_x$. Consequently,
\[
\E\{W_{ij}(y,1)\mid \mathcal{F}_x\}
=
W_{ij}(x,1),
\]
which shows that $\{W_{ij}(x,1):x\geq0\}$ is a continuous martingale. Moreover, by the continuity of the covariance function \(R(\cdot,1)\),
\(W_{ij}(\cdot,1)\) admits a continuous modification. From the Dambis--Dubins--Schwarz representation theorem, there exists
a standard Brownian motion $B$ such that
\[
W_{ij}(x,1)=B_{R(x,1)},\qquad x\geq0 .
\]
Consequently,
\[
\sup_{x\in[0,b]}|W_{ij}(x,1)|
=
\sup_{t\in[0,R(b,1)]}|B_t|.
\]
Since $R(x,1)$ is an increasing function on $x\in [a,b]$, applying Lemma 1.2 of \citet{OP73} yields
\[
\pr\left\{\sup_{x\in[a,b]}|W_{ij}(x,1)|\geq\lambda\right\}
\leq
\pr\left\{\sup_{x\in[0,b]}|W_{ij}(x,1)|\geq\lambda\right\} \leq
2\pr\{|W_{ij}(b,1)|\geq\lambda\}.
\]

Next, we show the main result. For $l=1,2,\ldots$, define
\[
\mathcal{A}_{l}=\left\{x:\frac{\varepsilon}{2^{l+1}}\le x \le \frac{\varepsilon}{2^l}\right\},
\]
then using~\eqref{eq:GaussianExtremes}, we have
\begin{align}
\label{eq:IeqW}
\nonumber
\pr\left\{\sup_{0<x\le \varepsilon} \frac{|W_{ij}(x,1)|}{x^{\eta}} \ge \lambda\right\}
&= \pr\left\{\sup_{l\in\{0,1,\ldots\}} \sup_{x\in \mathcal{A}_l } \frac{|W_{ij}(x,1)|}{x^{\eta}} \ge \lambda\right\}\\
\nonumber
&\le \sum_{l=0}^{\infty} \pr\left\{\sup_{x\in \mathcal{A}_l} \frac{\vert W_{ij}(x,1) \vert}{x^{\eta}}\ge \lambda\right\}\\
\nonumber
&\le \sum_{l=0}^{\infty} \pr\left\{\sup_{x\in \mathcal{A}_l } |W_{ij}(x,1)|\ge \lambda \left(\frac{\varepsilon}{2^{l+1}}\right)^{\eta}\right\}\\
\nonumber
&\le 2 \sum_{l=0}^{\infty} \pr\left\{|W_{ij}\left(\frac{\varepsilon}{2^l},1\right)|\ge \lambda \left(\frac{\varepsilon}{2^{l+1}}\right)^{\eta}\right\}\\
&\le 4 \sum_{l=0}^{\infty} \exp\left(-\frac{\lambda^2}{2^{1+2\eta}} \frac{2^{l(1-2\eta)}}{\varepsilon^{1-2\eta}}\right),
\end{align}
by Mill's inequality and the assumption that $\varepsilon<1$.
Notice that $W_{ij}(\frac{\varepsilon}{2^l},1)$ follows a normal distribution with zero mean and variance $R(\frac{\varepsilon}{2^l},1)$. The last inequality in \eqref{eq:IeqW} can be justified as follows,
\begin{align*}
\pr\left\{|W_{ij}\left(\frac{\varepsilon}{2^l},1\right)| \ge \lambda \left(\frac{\varepsilon}{2^{l+1}}\right)^{\eta}\right\}
&= \pr\left\{ \frac{|W_{ij}\left(\frac{\varepsilon}{2^l},1\right)|}{\sqrt{R(\frac{\varepsilon}{2^l},1)}} \ge \frac{ \lambda \left(\frac{\varepsilon}{2^{l+1}}\right)^{\eta}}{ \sqrt{R(\frac{\varepsilon}{2^l},1)}} \right\}\\
&\le  \pr\left\{ \frac{|W_{ij}\left(\frac{\varepsilon}{2^l},1\right)|}{\sqrt{R(\frac{\varepsilon}{2^l},1)}} \ge \frac{ \lambda \left(\frac{\varepsilon}{2^{l+1}}\right)^{\eta}}{ \sqrt{\frac{\varepsilon}{2^l}}} \right\}\\
&\le \frac{2}{\sqrt{2 \pi}} \frac{2^{\eta}}{\lambda} \left(\frac{\varepsilon}{2^l}\right)^{1/2-\eta} \exp\left(-\frac{\lambda^2}{2^{1+2\eta}} \frac{2^{l(1-2\eta)}}{\varepsilon^{1-2\eta}}\right)
\end{align*}
by Mill's inequality. For a fixed $\lambda>0$, taking $\varepsilon$ sufficiently small such that the constant $\frac{1}{\sqrt{2 \pi}} \frac{2^{\eta}}{\lambda} \left(\frac{\varepsilon}{2^l}\right)^{1/2-\eta}$ is less than one, this proves the last inequality in \eqref{eq:IeqW}.
\end{proof}

To establish the asymptotic normality of the proposed moment estimators, we require the following lemma, which provides a uniform convergence result for $\hat{\nu}_{n,I}(\bm{x}_I)$ defined in~\eqref{eq:Tailprocess}. A closely related conclusion is the weak convergence of the weighted tail copula process associated with the stable tail dependence function, as established in~\citet[Theorem 2.2]{EDL06}. Although we draw on some ideas from that theorem, our result differs by establishing weak convergence of the integral of the weighted tail empirical process with different weight functions, which require additional technical conditions in our lemma.

\begin{lemma}
\label{lemma:UniformConverge}
Suppose that \cref{Asp:SecondOrderCdt} holds and that the tail copula $R(\bm{x})$ of $\bm{X}$ satisfies, for $i,j\in V$ with $i\neq j$:
\begin{enumerate}
    \item[(\romannumeral1)] The partial derivative functions $\dot{R}_{ij}^i(x,y)$ and $\dot{R}_{ij}^{j}(x,y)$ exist and are continuous on $(0,\infty)^2$.
    \item[(\romannumeral2)] For $0<\eta<1/2$ and  $x\in(0,\exp(c)]$, the partial derivative satisfies
    \[
    \sup_{x\in [1/k,\exp(c)]} x^{-\eta} \Bigg| \dot{R}_{ij}^{j}(S_{ni}(x),\theta_{nj})- \dot{R}_{ij}^{j}(x,1) \Bigg|=o_{\pr}(1)
    \]
with $\theta_{nj}\in (S_{nj}(1)\wedge 1, \ S_{nj}(1)\vee 1)$ almost surely as $n\to\infty$, where $S_{n,I}(\bm{x}_{I})=(S_{ni}(x_{i}), i \in I)$ with
\[
S_{ni}(x):=\frac{n}{k} \hat{F}_{U,i}^{\leftarrow} \left(\frac{k x}{n}\right),
\]
and $\hat{F}_{U,i}^{\leftarrow}$ given by~\eqref{eq:InverseDF}.
\end{enumerate}
Then we have
\[
\int_{1/k}^{\exp(c)} \frac{\hat{\nu}_{n,ij}(x,1) \ell (-\log x)^{\ell-1}}{x} \d x \overset{\pr} \to \int_{0}^{\exp(c)} \frac{B_{ij}(x,1) \ell (-\log x)^{\ell-1} }{x} \d x
\]
as $n\to\infty$ for $\ell=1,2$, where $B_{ij}(x,1)$ is given in~\eqref{eq:ProcessBij}.
\end{lemma}

\begin{proof}
Set $c=-\log a$ with $a\in (0,1]$. Recall that
\[
\hat{\nu}_{n,I}(\bm{x}_I)=\sqrt{k} \left\{\hat{R}_{n,I}(\bm{x}_I)-R_{I}(\bm{x}_I)\right\}, \  I\subset V,
\]
with $\bm{x}_{I} \in [0,\infty]^{|I|}\setminus\{(\infty,\ldots,\infty)\}$.
By the definition of $\nu_{n,I}$ in~\eqref{eq:nu} and the equality that $\hat{R}_{n,I}(\bm{x}_I)=T_{n,I}(S_{n,I}(\bm{x}_I))$, we have
\begin{align*}
\hat{\nu}_{n,I}(\bm{x}_I)
&=\nu_{n,I}(S_{n,I}(\bm{x}_I))
+ \sqrt{k} \left\{R_{n,I}(S_{n,I}(\bm{x}_I))-R_{I}(S_{n,I}(\bm{x}_I))\right\}
\\
&\quad +\sqrt{k} \left\{R_{I}(S_{n,I}(\bm{x}_I))-R_{I}(\bm{x}_I)\right\}.
\end{align*}
Hence, for $i,j\in V$ and $i\neq j$,
\begin{align}
\label{eq:Dcp}
\nonumber
&\int_{1/k}^{1/a} \frac{\left\{\hat{\nu}_{n,ij}(x,1)-B_{ij}(x,1)\right\} \ell (-\log x)^{\ell-1} }{x} \d x \\
\nonumber
&\quad =\int_{1/k}^{1/a}  \frac{\left\{\nu_{n,ij}(S_{ni}(x),S_{nj}(1))-W_{ij}(x,1)\right\} \ell (-\log x)^{\ell-1} }{x} \d x \\
\nonumber
&\qquad + \int_{1/k}^{1/a}  \frac{\sqrt{k}\left\{R_{n,ij}(S_{ni}(x),S_{nj}(1))-R_{ij}(S_{ni}(x),S_{nj}(1))\right\} \ell (-\log x)^{\ell-1} }{x} \d x \\
\nonumber
&\qquad +  \int_{1/k}^{1/a} \frac{1}{x} \Bigg       [\sqrt{k}\left\{R_{ij}(S_{ni}(x),S_{nj}(1))-R_{ij}(x,1)\right\}\\
&\qquad \qquad \qquad  \quad +\left\{\dot{R}^{i}_{ij}(x,1)W_{i}(x)+\dot{R}^{j}_{ij}(x,1)W_{j}(1)\right\} \Bigg] \ell (-\log x)^{\ell-1}  \d x .
\end{align}
We will show, in turn, that each term on the right-hand side of~\eqref{eq:Dcp} converges to zero.

For the first part in~\eqref{eq:Dcp}, since for any $i\in V$, $S_{ni}(x)$ is a nondecreasing function of $x$ and $S_{ni}(1/a)\overset{\pr} \to 1/a$ as $n\to\infty$ (cf. Eq.~(3.10) in~\citet{EDL06}), there exists a constant $T_0$ such that, with high probability, $0<S_{ni}(x)\le T_0$ for all $x\in (0,1/a]$. Hence for $0\le \eta <1/2$,
\begin{align*}
    &\sup_{ 1/k \le x\le 1/a} \frac{\big|\nu_{n,ij}(S_{ni}(x), S_{nj}(1))-W_{ij}(x,1) \big|}{x^{\eta}}
    \\
    &\quad \le  \sup_{ 1/k \le x\le 1/a} \frac{\big|\nu_{n,ij}(S_{ni}(x), S_{nj}(1))-W_{ij}(S_{ni}(x),S_{nj}(1)) \big|}{x^{\eta}}
    \\
    &\qquad + \sup_{ 1/k \le x\le 1/a} \frac{\big|W_{ij}(S_{ni}(x),S_{nj}(1))-W_{ij}(x,1) \big|}{x^{\eta}}
    \\
    &\quad \le  \sup_{ 1/k \le x\le 1/a} \frac{\big|\nu_{n,ij}(S_{ni}(x), S_{nj}(1))-W_{ij}(S_{ni}(x),S_{nj}(1)) \big|}{\left(S_{ni}(x)\right)^{\eta}}
    \times \sup_{ 1/k \le x\le 1/a} \left(\frac{S_{ni}(x)}{x}\right)^{\eta}
    \\
    &\qquad + \sup_{ 1/k \le x\le 1/a} \frac{\big|W_{ij}(S_{ni}(x),S_{nj}(1))-W_{ij}(x,1) \big|}{x^{\eta}}
    \\
    & \quad \le  \sup_{ u\in (0,T_0]} \frac{\big|\nu_{n,ij}(u, S_{nj}(1))-W_{ij}(u,S_{nj}(1)) \big|}{u^{\eta}} \cdot \sup_{s\in (0,k/(na)]} \left(\frac{\hat{F}_{U,i}^{\leftarrow} (s)}{s}\right)^{\eta}
    \\
    &\qquad + \sup_{ 1/k \le x\le 1/a} \frac{\big|W_{ij}(S_{ni}(x),S_{nj}(1))-W_{ij}(x,1) \big|}{x^{\eta}}
    \\
    &\quad =: D_{n11}\cdot D_{n12} +D_{n13},
\end{align*}
where the last inequality holds with high probability. Since $S_{nj}(1)\overset{\pr} \to 1$ as $n\to\infty$, $T_0$ can be set sufficiently large such that with high probability, $S_{nj}(1)\le T_0$, and then
\begin{align*}
\sup_{ u\in (0,T_0]} \frac{\big|\nu_{n,ij}(u,  S_{nj}(1))-W_{ij}(u,S_{nj}(1) \big|}{u^{\eta}}
 \le  \sup_{ u\in (0,T_0]} \frac{\big|\nu_{n,ij}(u, S_{nj}(1))-W_{ij}(u,S_{nj}(1)) \big|}{\left(u \wedge S_{nj}(1)\right)^{\eta}},
\end{align*}
implying $D_{n11}\overset{\pr}\to 0$ as $n\to\infty$ by \cref{lemma:WTP}, \eqref{eq:uniformconverge}. With the fact that
\begin{equation}
\label{eq:ScaledQuantile}
\sup_{s\ge 1/n} \frac{\hat{F}_{U,i}^{\leftarrow} (s)}{s} =O_{\pr}(1),
\end{equation}
see Eq.(3.10) of~\citet{EDL06} or~\citet[Page 416]{SW1986}, we have $D_{n11} D_{n12}\overset{\pr}\to 0$ as $n\to\infty$.
Furthermore, for $1/k < \varepsilon <
1$, we know that
\begin{align*}
    D_{n13} &\le  \sup_{ \substack{x\in (0,1/a]
    \\ x\ge \varepsilon}} \frac{\big|W_{ij}(S_{ni}(x),S_{nj}(1))-W_{ij}(x,1) \big|}{x^{\eta}}
    \\
    & \quad + \sup_{ \substack{x\in (0,1/a]\\ 1/k\le x < \varepsilon}} \frac{\big|W_{ij}(S_{ni}(x),S_{nj}(1))-W_{ij}(x,1) \big|}{x^{\eta}}
    \\
    & \le  \sup_{ \substack{x\in (0,1/a]\\ x\ge \varepsilon}} \frac{\big|W_{ij}(S_{ni}(x),S_{nj}(1))-W_{ij}(x,1) \big|}{\varepsilon^{\eta}} + \sup_{ \substack{x\in (0,1/a]\\ 1/k\le x\le \varepsilon}} \frac{\big|W_{ij}(x,1)\big |}{x^{\eta}}
    \\
    & \quad + \sup_{ \substack{x\in (0,1/a]\\ 1/k\le x\le \varepsilon}} \frac{\big|W_{ij}(S_{ni}(x),S_{nj}(1)) \big |}{S_{ni}(x)^{\eta}} \sup_{s\in (0,k/(na)] } \left(\frac{\hat{F}_{U,i}^{\leftarrow} (s)}{s}\right)^{\eta}
    \\
    &=: D_{n14}+D_{n15}+D_{n16}.
\end{align*}
Since by Smirnov's lemma, see, e.g., \citet[Lemma~2.2.3]{S49},
\begin{align}
\label{eq:uniformCovQuantile}
\sup_{0<s<k/(na)} \frac{n}{k} \big| \hat{F}_{U,i}^{\leftarrow} (s)-s \big| \overset{a.s} \to 0
\end{align}
as $n\to\infty$, it follows from the uniform continuity of $W_{ij}$ that $D_{n14}\to 0$ almost surely for any $\varepsilon>0$. By~\cref{lemma:inequalityadd} and~\eqref{eq:ScaledQuantile}, for any $\delta>0$, we also have $\pr(D_{n15}>\delta)<\delta$ and $\pr(D_{n16}>\delta)<\delta$ for sufficiently large $n$ and small $\varepsilon$. Thus
\begin{align}
\label{eq:Part1}
\nonumber
& \Bigg| \int_{1/k}^{1/a} \frac{\left\{\nu_{n,ij}(S_{ni}(x),S_{nj}(1))-W_{ij}(x,1)\right\} \ell (-\log x)^{\ell-1}}{x} \d x   \Bigg|
\\
& \quad \le \sup_{ 1/k \le x\le 1/a } \frac{\big|\nu_{n,ij}(S_{ni}(x), S_{nj}(1))-W_{ij}(x,1) \big|}{x^{\eta}}  \times\Bigg| \int_{1/k}^{1/a}  x^{{\eta}-1 } \cdot \ell (-\log x)^{\ell-1} \d x   \Bigg|
 \overset{\pr}\to 0
\end{align}
as $n\to\infty$ for $\ell=1,2$.

Now we consider the second part in~\eqref{eq:Dcp}. \Cref{Asp:SecondOrderCdt} implies that
\begin{align*}
   \sup_{\bm{x}\in (0,1]^{|I|}} \big|\frac{1}{q} C_{I}(q\bm{x})-R_{I}(\bm{x}) \big|\le K_{R} q^{\xi},
\end{align*}
for details see Eq.~(S.18) in \citet{EV20}. Let $r=x\wedge y$, then we have, for any $i,j\in V$ and $x,y>0$, that
\begin{align}
\label{eq:ConvergeRate}
\nonumber
    &\frac{1}{q} C_{ij}(q x,q y)=\frac{r}{q r} C_{ij}\left(q r \frac{x}{r}, q r \frac{y}{r}\right)\\
    &\quad = r\left\{ R_{ij}\left(x/r,y/r\right) +O((q r)^{\xi}) \right\} = R_{ij}\left(x,y\right)+ r^{1+\xi} O(q^{\xi})
\end{align}
as $q\to 0$. Moreover, noting that by~\eqref{eq:ConvergeRate}, \cref{Asp:SecondOrderCdt}, the relation $k=o\left(n^{\xi/(\xi+\frac{1}{2})}\right)$ and $1+\xi-\eta >0$, we have
\begin{align*}
    &\sup_{u\in (0,T_0]} \frac{\big|\sqrt{k} \left\{R_{n,ij}(u,S_{nj}(1))- R_{ij}(u,S_{nj}(1))\right\} \big|}{\left[u \wedge S_{nj}(1)\right]^{\eta}}\\
    &\quad = \sup_{u\in (0,T_0]}  \left\{u \wedge S_{nj}(1)\right\}^{1+\xi-\eta} O_{\pr}\left(k^{\xi+1/2} /n^{\xi}\right) =o_{\pr}(1).
\end{align*}
With the fact that $D_{n12}=\sup_{s\in (0,k/(nc)] } \left(\hat{F}_{U,i}^{\leftarrow} (s)/s \right)^{\eta}\overset{\pr}\to 0$, we obtain that, with arbitrarily high probability,
\begin{align*}
   & \sup_{ 1/k \le x\le 1/a } \Bigg|\frac{\sqrt{k} \left\{R_{n,ij}(S_{ni}(x),S_{nj}(1))-R_{ij}(S_{ni}(x),S_{nj}(1))\right\}}{x^{\eta}} \Bigg| \\
   & \quad  \le  \sup_{ 1/k \le x\le 1/a } \frac{\big|\sqrt{k} \left\{R_{n,ij}\left(S_{ni}(x),S_{nj}(1)\right)- R_{ij}\left(S_{ni}(x),S_{nj}(1)\right)\right\} \big|}{S_{ni}(x)^{\eta}} \cdot \sup_{ 1/k \le x\le 1/a } \left(\frac{S_{ni}(x) }{x}\right)^{\eta}\\
    & \quad \le \sup_{u\in (0,T_0]} \frac{\big|\sqrt{k} \left\{R_{n,ij}(u,S_{nj}(1))- R_{ij}(u,S_{nj}(1))\right\} \big|}{\left\{u \wedge S_{nj}(1)\right\}^{\eta}}
    \cdot \sup_{s\in (0,k/(na)] } \left(\frac{\hat{F}_{U,i}^{\leftarrow} (s)}{s}\right)^{\eta}\\
    & \quad \overset{\pr} \to 0
\end{align*}
as $n\to\infty$. It further follows that for $\ell=1,2$,
\begin{align}
\label{eq:Part2}
\nonumber
&\Bigg| \int_{1/k}^{1/a} \frac{\sqrt{k}\left\{R_{n,ij}(S_{ni}(x),S_{nj}(1))-R_{ij}(S_{ni}(x),S_{nj}(1))\right\} \ell (-\log x)^{\ell-1} }{x} \d x  \Bigg|\\
\nonumber
&\quad \le  \sup_{ \substack{x\in (0,1/a]\\ x\ge 1/k}} \Bigg|\frac{\sqrt{k}\left\{R_{n,ij}(S_{ni}(x),S_{nj}(1))-R_{ij}(S_{ni}(x),S_{nj}(1))\right\}}{x^{\eta}} \Bigg| \\
\nonumber
&\qquad \cdot \Bigg| \int_{1/k}^{1/a} x^{\eta-1}
   \cdot \ell (-\log x)^{\ell-1} \d x \Bigg| \\
& \quad \overset{\pr} \to 0, \qquad n \to \infty.
\end{align}

For the third part on the right-hand side of~\eqref{eq:Dcp}, let $(0,T_1)\times (0,T_2)\supset (0,1/a]^2$ be the set such that the partial derivatives of $R_{ij}(x,y)$ exist and are continuous by assumption. Since $S_{ni}(x)$ converges $x$ uniformly on $(0,1/a]$ and $R_{ij}(x,y)$ is continuous on $[0,T_1]\times [0,T_2]$, for large $n$ such that $(S_{ni}(x),S_{nj}(1))\in [0,T_1]\times [0,T_2]$, applying the mean value theorem gives
\begin{align*}
    &R_{ij}\left(S_{ni}(x),S_{nj}(1)\right)-R_{ij}(x,1)\\
    &\quad =\left\{R_{ij}\left(S_{ni}(x),S_{nj}(1)\right)-R_{ij}\left(S_{ni}(x),1\right)\right\} +\left\{R_{ij}(S_{ni}(x),1)-R_{ij}(x,1)\right\}\\
    &\quad =\dot{R}_{ij}^{i}(\theta_{ni},1) \left(S_{ni}(x)-x\right) + \dot{R}_{ij}^{j}(S_{ni}(x),\theta_{nj}) \left(S_{nj}(1)-1\right),
\end{align*}
where $\theta_{ni}\in (S_{ni}(x)\wedge x, \ S_{ni}(x)\vee x)$ and $\theta_{nj}\in (S_{nj}(1)\wedge 1, \ S_{nj}(1)\vee 1)$ almost surely for large $n$. Thus,
\begin{align}
\label{eq:Part3}
\nonumber
& \int_{1/k}^{1/a} \frac{1}{x} \Bigg[\sqrt{k}\left\{R_{ij}(S_{ni}(x),S_{nj}(1))-R_{ij}(x,1)\right\}\\
\nonumber
&\qquad \qquad  +\left\{\dot{R}^{i}_{ij}(x,1)W_{i}(x)+\dot{R}^{j}_{ij}(x,1)W_{j}(1)\right\} \Bigg] \ell (-\log x)^{\ell-1}  \d x \\
\nonumber
& \quad \le  \Bigg|  \int_{1/k}^{1/a} \frac{ \left[\dot{R}_{ij}^{j}(S_{ni}(x),\theta_{nj}) \sqrt{k} \left\{S_{nj}(1)-1\right\}+\dot{R}_{ij}^{j}(x,1)W_{j}(1) \right] \ell (-\log x)^{\ell-1}  }{x } \d x \Bigg|\\
\nonumber
    & \qquad +  \Bigg| \int_{1/k}^{1/a}  \frac{ \left[\dot{R}_{ij}^{i}(\theta_{ni},1) \sqrt{k} \left\{S_{ni}(x)-x\right\}+\dot{R}_{ij}^{i}(x,1)W_{i}(x) \right ] \ell (-\log x)^{\ell-1}  }{x } \d x  \Bigg| \\
&\quad=: E_{n1} + E_{n2}.
\end{align}

We begin by analyzing the term $E_{n1}$. Observe that
\begin{align}
\label{eq:En1}
\nonumber
|E_{n1}|&\le \Bigg| \int_{1/k}^{1/a} \frac{\dot{R}_{ij}^{j}(S_{ni}(x),\theta_{nj})  \left[\sqrt{k} \{S_{nj}(1)-1\}+W_{j}(1) \right] \ell (-\log x)^{\ell-1} }{x } \d x  \Bigg|
\\
\nonumber
&\quad + \Bigg| \int_{1/k}^{1/a}  \frac{\left\{\dot{R}_{ij}^{j}(S_{ni}(x),\theta_{nj})- \dot{R}_{ij}^{j}(x,1) \right\} W_{j}(1) \ell (-\log x)^{\ell-1}}{x } \d x \Bigg|
\\
\nonumber
&\le \left\{\sqrt{k} \{S_{nj}(1)-1\}+W_{j}(1) \right\} \cdot \Bigg| \int_{1/k}^{1/a} \frac{\dot{R}_{ij}^{j}(S_{ni}(x),\theta_{nj})  \ell (-\log x)^{\ell-1} }{x } \d x  \Bigg|
\\
&\quad + \sup_{x\in [1/k,1/a]} x^{-\eta} \Bigg| \dot{R}_{ij}^{j}(S_{ni}(x),\theta_{nj})- \dot{R}_{ij}^{j}(x,1) \Bigg| W_{j}(1) \cdot \Bigg| \int_{1/k}^{1/a} x^{\eta-1} \cdot \ell (-\log x)^{\ell-1} \d x  \Bigg|,
\end{align}
where the second term on the right hand side of the above inequality goes to zero in probability by condition $(\romannumeral2)$ of the assumption. Setting $\eta=0$ in~\eqref{eq:uniformconverge}, one have
\[
\sqrt{k}\left[S_{nj}(1)-1\right]+W_{j}(1)\overset{\pr}\to 0.
\]
Moreover, by the fact that
\begin{equation}
\label{eq:PartialDevIeq}
0\le \dot{R}_{ij}^{j}(x,1) \le R_{ij}(x,1) \le x,
\end{equation}
we have
\begin{align*}
& \Bigg| \int_{1/k}^{1/a} \frac{\dot{R}_{ij}^{j}(S_{ni}(x),\theta_{nj}) \ell (-\log x)^{\ell-1} }{x } \d x  \Bigg|\\
&\le \int_{0}^{1/a} \frac{|\dot{R}_{ij}^{j}(S_{ni}(x),\theta_{nj})|  \ell (-\log x)^{\ell-1}}{x } \d x
\\
&\le \int_{0}^{1/a} \frac{|S_{ni}(x)|}{x } \ell (-\log x)^{\ell-1} \d x\\
& \le \sup_{1/k\le x \le 1/a} \frac{|S_{ni}(x)|}{x } \times \int_{0}^{1/a} \ell (-\log x)^{\ell-1} \d x < \infty
\end{align*}
by~\eqref{eq:uniformCovQuantile}. Hence the first term in the right-hand of~\eqref{eq:En1} converges to zero as $n\to\infty$.
Note that the tail copula $R(\bm{x})$ is nondecreasing in each coordinate, its partial derivatives (whenever they exist) are nonnegative. Hence, the inequality~\eqref{eq:PartialDevIeq} can be validated by differentiating $R(\lambda \bm{x})=\lambda R(\bm{x})$ with respect to $\lambda$ and evaluating at $\lambda=1$, which yields
\begin{equation}
\label{eq:PartialEuation}
\sum_{j=1}^d x_j \dot{R}^{j}(\bm{x})=R(\bm{x}),
\end{equation}
implying $x_j \dot{R}^{j}(\bm{x}) \le R(\bm{x}) \le \min_{i\in V} x_i$. In particular, this leads to the inequality~\eqref{eq:PartialDevIeq}. For more details, see Lamma~5 and Lemma~A.3 in~\citet{ES2021}.

For the term $E_{n2}$, note that
\begin{align*}
    \lefteqn{|E_{n2}|} \\
    &\le \Bigg| \int_{1/k}^{1/a} \frac{\dot{R}_{ij}^{i}(\theta_{ni},1)  \left[\sqrt{k} \{S_{ni}(x)-x\} + W_{i}(x) \right] \ell (-\log x)^{\ell-1} }{x } \d x  \Bigg|
\\
&\quad + \Bigg| \int_{1/k}^{1/a}  \frac{\left\{\dot{R}_{ij}^{i}(\theta_{ni},1)- \dot{R}_{ij}^{i}(x,1) \right\} W_{i}(x) \ell (-\log x)^{\ell-1} }{x } \d x \Bigg|
\\
&\le \sup_{1/k\le x\le 1/a} \Bigg|\dot{R}_{ij}^{i}(\theta_{ni},1) \frac{ \sqrt{k} \{S_{ni}(x)-x\}-W_{i}(x) }{x^{\eta} } \Bigg| \cdot \Bigg| \int_{1/k}^{1/a} x^{\eta-1} \cdot \ell (-\log x)^{\ell-1}\d x  \Bigg|
\\
&\quad + \sup_{x\in [1/k,1/a]}  \Bigg| \dot{R}_{ij}^{i}(\theta_{ni},1)- \dot{R}_{ij}^{i}(x,1) \Bigg| \cdot \sup_{x\in [1/k,1/a]} \frac{|W_{i}(x)|}{x^{\eta}}   \cdot \Bigg| \int_{1/k}^{1/a} x^{\eta-1} \cdot \ell (-\log x)^{\ell-1} \d x  \Bigg|.
\end{align*}
The first part on the right-hand side goes to zero as $n\to \infty$ from~\eqref{eq:uniformconverge} and the inequality $0 \le \dot{R}_{ij}^{i}(\theta_{ni},1)\le 1$, where the latter is implied by $\theta_{ni} \dot{R}_{ij}^{i}(\theta_{ni},1) \le \dot{R}_{ij}^{i}(\theta_{ni},1) \le  \theta_{ni}$ from~\eqref{eq:PartialEuation}. By assumption $\dot{R}_{ij}^i(x,1)$ is a continuous function, and hence uniformly continuous on the interval $x \in [1/k,1/a]$. Recall that $\theta_{ni}$ lies between $x$ and $S_{ni}(x)$. From~\eqref{eq:uniformCovQuantile} we know that $S_{ni}(x)$ converges uniformly to $x$. As a consequence, $\theta_{ni}$ also converges to its limiting value. Hence we have
\[
\sup_{x\in [1/k,1/a]}\Bigg|\dot{R}_{ij}^{i}(\theta_{ni},1)-\dot{R}_{ij}^{i}(x,1) \Bigg| \to 0
\]
as $n\to\infty$. Recall that $W_i(x)$ is a normal distributed random variable with zero mean and variance $x$. For any $M>0$,
\begin{align}
\label{eq:BoundedMarginalProcess}
\nonumber
&\pr\left(\bigg| \sup_{x\in(0,1/a]}\frac{W_i(x)}{x^{\eta}}\bigg|>M\right)
=\pr\left(|W_i(1)| \bigg| \sup_{x\in(0,1/a]}x^{1/2-\eta}\bigg|>M\right)
\\
&\quad =2 \pr\left[W_{i}(1)>M a^{\eta-1/2}\right] \le \frac{2\E(W_{i}(1)^2)}{M^2 a^{2\eta-1}} \to 0
\end{align}
as $M\to \infty$ by Chebyshev inequality. Hence  $E_{n2}\overset{\pr}\to 0$ as $n\to \infty$. Consequently, by~\eqref{eq:Dcp}, \eqref{eq:Part1}, \eqref{eq:Part2} and \eqref{eq:Part3}, we obtain
\[
\int_{1/k}^{1/a} \frac{\left\{\hat{\nu}_{n,ij}(x,1)-B_{ij}(x,1) \right\} \ell (-\log x)^{\ell-1} }{x } \d x \overset{\pr}\to 0
\]
as $n\to\infty$.
Recall the definition of $B_{ij}(x,1)$ in~\eqref{eq:ProcessBij}.
By~\eqref{eq:PartialDevIeq}, we have
\begin{align*}
\int_{0}^{1/k} \frac{B_{ij}(x,1) \ell (-\log x)^{\ell-1} }{x} \d x \to 0,
\end{align*}
in mean square as $n\to\infty$, implying the desired result.
\end{proof}

Recall the definition of $S_{n,I}(\bm{x}_I)$ in~\eqref{eq:InverseDF}. The following lemma verifies condition~$(\romannumeral2)$ of~\cref{lemma:UniformConverge} for the \HR{} MGPD.

\begin{lemma}
\label{lemma:ConRatePartialDer}
Let $R(\bm{x})$ be the tail copula function of a \HR{} MGPD with variogram matrix $\bm{\Gamma}=(\gamma_{ij})_{i,j\in V}$ satisfying \cref{Asp:PositiveVariogram}. For each $j
\in V$, assume $\{\theta_{nj}\}_{n\ge 1}$ is a sequence lying between $1$ and $S_{nj}(1)$ almost surely for large $n$. Then, for $i,j\in V$ and $ i\neq j$, we have
\[
\sup_{x \in [1/k,\exp(c)]} x^{-\eta} \Bigg| \dot{R}_{ij}^{j}(S_{ni}(x),\theta_{nj})-\dot{R}_{ij}^{j}(x,1)\Bigg|\to 0
\]
in probability as $n\to\infty$.
\end{lemma}

\begin{proof}
For $i,j\in V$, the bivariate tail copula associated with the $(i,j)$ component of the H\"{u}sler--Reiss distribution is a function of $\gamma_{ij}$ only. For notational simplicity, we write
\[
R_{ij}(x,y)=R(x,y;\gamma_{ij}), \ i,j \in V
\]
with
\begin{equation}
\label{eq:bivariateTailCopulaHR}
    R(x,y;\gamma)
    = x + y
    - x \, \Phi\!\left(
        \frac{\sqrt{\gamma}}{2}
        + \frac{\log{x}-\log{y}}{\sqrt{\gamma}}
    \right)
    - y \, \Phi\!\left(
        \frac{\sqrt{\gamma}}{2}
        + \frac{\log{y}-\log{x}}{\sqrt{\gamma}}
    \right)
\end{equation}
for $(x,y)\in [0,\infty)^2$.
Taking derivative with respect to $x$ and $y$ respectively, we get
\begin{equation}
\label{eq:firstPartialDerivative}
    \dot{R}_{ij}^{i}(x,y)
    = \dot{R}_{ij}^{j}(y,x)
    = 1
    - \Phi\!\left(
        \frac{\sqrt{\gamma_{ij}}}{2}
        + \frac{\log{x}-\log{y}}{\sqrt{\gamma_{ij}}}
    \right).
\end{equation}
Therefore,
\begin{align*}
&\sup_{x \in [1/k,1/a]} x^{-\eta} \Bigg| \dot{R}_{ij}^{j}(S_{ni}(x),S_{nj}(1))-\dot{R}_{ij}^{j}(x,1) \Bigg|
\\
&\qquad \le \sup_{x \in [1/k,1/a]} x^{-\eta} \Bigg|  \Phi\left(\sqrt{\gamma_{ij}}-\frac{\log{S_{ni}(x)}-\log{S_{nj}(1)}}{2\sqrt{\gamma_{ij}}}\right) - \Phi\left(\sqrt{\gamma_{ij}}-\frac{\log{x}}{2\sqrt{\gamma_{ij}}}\right)\Bigg|
\\
&\qquad =: h_n(x).
\end{align*}

Recall that $c=-\log a$. By~\eqref{eq:uniformCovQuantile}, we have
\[
\frac{S_{ni}(x)}{S_{nj}(1)}\to x
\]
almost surely as $n\to\infty$. Applying Taylor expansion to the standard normal distribution function at $\sqrt{\gamma_{ij}}-\log{x}/(2\sqrt{\gamma_{ij}})$ gives
\begin{align*}
h_n(x)&=\sup_{x\in[1/k,1/a]} x^{-\eta} \left|-\frac{1}{2\sqrt{\gamma_{ij}}} \frac{1}{x} \phi\left(\sqrt{\gamma_{ij}}-\frac{\log{x}}{2\sqrt{\gamma_{ij}}}\right) \left(\frac{S_{ni}(x)}{S_{nj}(1)}-x\right)\right| (1+o(1))
\\
&\le \frac{1}{\sqrt{\gamma_{ij}}} \sup_{x\in[1/k,1/a]}    \frac{1}{x} \phi\left(\sqrt{\gamma_{ij}}-\frac{\log{x}}{2\sqrt{\gamma_{ij}}}\right) \cdot \sup_{x\in [1/k,1/a]} x^{-\eta} \Bigg| \frac{S_{ni}(x)}{S_{nj}(1)}-x\Bigg| (1+o(1))
\end{align*}
for sufficiently large $n$. Since
\[
\lim_{n\to\infty} k\phi\left(\sqrt{\gamma_{ij}}+\frac{\log{k}}{2\sqrt{\gamma_{ij}}}\right)=0.
\]
and it is continuous on $[1/k,1/a]$, we have
\[ \sup_{x\in[1/k,1/a]} \frac{1}{x} \phi\left(\sqrt{\gamma_{ij}}-\frac{\log{x}}{2\sqrt{\gamma_{ij}}}\right) < \infty
\]
Moreover, note that
\begin{align*}
    \lefteqn{\sup_{x\in [1/k,1/a]} x^{-\eta} \Bigg| \frac{S_{ni}(x)}{S_{nj}(1)}-x\Bigg|} \\
    &=\frac{1}{|S_{nj}(1)|} \Biggl[\frac{1}{\sqrt{k}} \sup_{x\in [1/k,1/a]} \left\{\Bigg| \frac{\sqrt{k} (S_{ni}(x)-x)-W_i(x)}{x^{\eta}} \Bigg|+  \frac{|W_i(x)|}{x^{\eta}}\right\}  \\
    &\hspace{30mm} \mbox{} + \sup_{x\in [1/k,1/a]} x^{1-\eta} |S_{nj}(1)-1| \Biggr]
    \\
    &\le \frac{1}{|S_{nj}(1)|} \Biggl[\frac{1}{\sqrt{k}} \left\{ \sup_{x\in [1/k,1/a]} \Bigg| \frac{\sqrt{k} (S_{ni}(x)-x)-W_i(x)}{x^{\eta}} \Bigg| + \sup_{x\in [1/k,1/a]} \frac{|W_i(x)|}{x^{\eta}}\right\} \\
    &\hspace{30mm} \mbox{} +  a^{\eta-1} |S_{nj}(1)-1| \Biggr]  \\
    &\to 0
\end{align*}
in probability by~\eqref{eq:uniformconverge}, \eqref{eq:BoundedMarginalProcess} and the fact that $S_{nj}(1)\to 1$ almost surely, as $n\to \infty$. This implies $h_n(x)\to 0$ in probability and the proof is complete.
\end{proof}

\subsection{Proofs of propositions and theorems}
\label{sec:A.2}

In this section, we give proofs of the lemma, propositions and theorems in~\cref{sec:Results}.

\begin{proof}[{\bf Proof of~\cref{lemma:monotonicity}}]
We first derive the expressions of the moment functions. Recall that for any $i,j\in V$ and $i\neq j$,
\[
\left(Y_i^{(j)}, \, Y_j^{(j)}\right) \overset{d}= \left( Y_j^{(i)}, \, Y_i^{(i)}\right)
\overset{d}
=
\left(E+Z_i^{(j)}, \,  E\right),
\]
where $Z_i^{(j)}$ is a normally distributed random variable with mean $-\gamma_{ij}/2$ and variance $\gamma_{ij}$, independent of the standard exponential random variable $E$.
By the independence of $E$ and $Z_i^{(j)}$, a direct calculation gives
\begin{align}
\label{eq:FisrtPos}
\nonumber
\E[(E+x)_{+}]&=\int_{0}^{\infty} (y+x)_{+} \exp(-y) \d y = \int_{(-x)\vee 0}^{\infty} (y+x) \exp(-y) \d y \\
&= \exp(x) \I\{x<0\} + (1+x) \I\{x\ge 0\},
\end{align}
and then
\begin{align*}
&e^{(1)}(\gamma_{ij},c)
= \E \left\{ \rbr{E + Z_i^{(j)} +c}_+ \right\}  \\
    &\quad =
    \int_{-\infty}^\infty \E \left\{ (E + x)_+ \right\} \, \d \pr[Z_i^{(j)} +c \le x]  \\
    &\quad =
    \int_{-\infty}^0 \e^x \, \d \pr\left\{Z_i^{(j)} + c  \le x\right\} + \int_0^\infty (1+x) \, \d \pr\left\{Z_i^{(j)}  + c \le x\right\} \\
    &\quad = \exp(c) \Phi\left(- \frac{ c+ \gamma_{ij}/2}{\sqrt{\gamma_{ij}}}\right)
    +  \sqrt{\gamma_{ij}} \phi\left(\frac{c - \gamma_{ij}/2}{\sqrt{\gamma_{ij}}}\right)
     + \left(1-\frac{\gamma_{ij}}{2} + c\right) \Phi\left(\frac{c - \gamma_{ij}/2}{\sqrt{\gamma_{ij}}}\right).
\end{align*}

Similarly to~\eqref{eq:FisrtPos}, we have
\begin{align*}
\E\left[\left\{(E+x)_{+}\right\}^2\right]&=\int_{0}^{\infty} \left\{(y+x)_{+}\right\}^2 \exp(-y) \d y = \int_{(-x)\vee 0}^{\infty} (y+x)^2 \exp(-y) \d y \\
&= 2\exp(x) \I(x<0) + (2+2x+x^2) \I(x\ge 0).
\end{align*}
It follows that
\begin{align*}
&e^{(2)}(\gamma_{ij},c)=\E \left[ \left\{\rbr{E + Z_i^{(j)} +c }_+\right\}^2 \right] \\
&\quad =
    \int_{-\infty}^\infty \E\left[\{(E+x)_{+}\}^2\right] \, \d \pr \left\{Z_i^{(j)} + c \le x\right\}  \\
    &\quad =
   2 \int_{-\infty}^0 \exp(x) \, \d \pr\left\{Z_i^{(j)} + c \le x\right\} + \int_0^\infty (2+2x+x^2) \, \d \pr\left\{Z_i^{(j)} + c \le x\right\}\\
    &\quad =2 e_{ij}^{(1)}   +  \int_{0}^{\infty} x^2 \, \d \pr\left\{Z_i^{(j)} + c \le x\right\}\\
    &\quad = 2 e_{ij}^{(1)}  -\sqrt{\gamma_{ij}} \left(\frac{\gamma_{ij}}{2} - c\right) \phi\left(\frac{ c - \gamma_{ij}/2}{\sqrt{\gamma_{ij}}}\right)    + \left\{\frac{1}{4} \gamma_{ij}^2 +(1 - c) \gamma_{ij} + c^2\right\} \Phi\left(\frac{ c - \gamma_{ij}/2}{\sqrt{\gamma_{ij}}}\right).
    \end{align*}
Plugging in the expression of $e^{(1)}(\gamma_{ij},c)$ yields the formula of $e^{(2)}(\gamma_{ij},c)$.

We now show the monotonicity of the moment functions.
For given $c$, the functions $e^{(\ell)}(\gamma)=e^{(\ell)}(\gamma, c)$ with $\ell=1,2$ that map $[0,\infty)$ to $\mathbb{R}$,
are injective and continuously differentiable on $(0,\infty)$, and right-continuous at $0$. The derivatives of $e^{(\ell)}(\gamma)$, $\ell=1,2$, satisfy
\[
e^{(1),'}(\gamma)=-\frac{1}{2} \Phi\left(\frac{ c -\frac{\gamma}{2}}{\sqrt{\gamma}}\right) <0
\]
and
\begin{equation}
\label{eq:Derivative2}
e^{(2),'}(\gamma)=\left(\frac{\gamma}{2} - c\right)  \Phi\left(\frac{c -\frac{\gamma}{2}}{\sqrt{\gamma}}\right)-\sqrt{\gamma} \phi\left(\frac{ c - \frac{\gamma}{2}}{\sqrt{\gamma}}\right)  <0
\end{equation}
for $\gamma\in (0,\infty)$.
Here, the inequality in~\eqref{eq:Derivative2} follows from Mill's inequality that $(1-\Phi(x))/\phi(x) < 1/x $ for $x>0$. Specifically, if $\gamma/2 - c \le 0$, the inequality~\eqref{eq:Derivative2} already holds since both terms in the left expression of the inequality are negative for $\gamma>0$; otherwise, if $\gamma/2 -  c > 0$, we have by Mill's inequality that
\[
\frac{\left(\frac{\gamma}{2} - c\right)  \Phi\left(\frac{c - \frac{\gamma}{2}}{\sqrt{\gamma}}\right)}{\sqrt{\gamma} \phi\left(\frac{c - \frac{\gamma}{2}}{\sqrt{\gamma}}\right)} = \frac{\frac{\gamma}{2} - c}{\sqrt{\gamma}} \cdot \frac{1-\Phi\left(\frac{\frac{\gamma}{2} - c}{\sqrt{\gamma}}\right)}{\phi\left(\frac{  \frac{\gamma}{2} -c }{\sqrt{\gamma}}\right)} <1,
\]
which yields~\eqref{eq:Derivative2}. Moreover, by right-continuity at $\gamma=0$, the monotonicity extends to $[0,\infty)$. Consequently, $e^{(\ell)}: [0,\infty) \to (0,\infty)$ ($\ell=1,2$) is strictly decreasing. The proof is complete.
\end{proof}

\begin{proof}[Proof of~\cref{pro:weakconsistency}]
Recall that $Y_i^{(j)}$ is $Y_i \mid Y_j>0$ and that $R(\bm{x})$ is the tail copula defined in~\eqref{eq:UTDF}. By~\eqref{eq:MPD}, \eqref{eq:marginal_Transform}, \eqref{eq:UTDF} and the assumption on $\bm X$ and $\bm Y$, we know that
\begin{equation}
\label{eq:TailDistributionSame}
\pr\left(Y_{i}^{(j)}\ge x, Y_{j}^{(j)}\ge y\right)=R_{ij}(\exp(-x),\exp(-y)), \qquad (x,y)\in[-\infty,\infty]\times [0,\infty].
\end{equation}
With the notation $c=-\log a$ for $a\in (0,1]$, we have
\begin{equation}
\label{eq:RepMomentC}
\E\left[ \left\{Y_{i}^{(j)} \vee (-c) \right\}^{\ell} \right]=\int_{0}^{1/a} \frac{R_{ij}(x,1) \ell (-\log x)^{\ell-1}}{x} \d x +(\log a)^{\ell}:=\tilde{e}^{(\ell)}(\gamma,c);
\end{equation}
for the proof of this expression, see~\cref{sec:proofexpansion}.

Note that, with the transformation $u=-\log q$ ($0< q \le 1$), the excesses on the
exponential scale can be equivalently represented as
\begin{align*}
-\log(1-F_i(X_i))-u \cdt -\log(1-F_j(X_j))>u
= \log\left\{\frac{q}{1-F_i(X_i)}\right\} \cdt 1-F_j(X_j)<q.
\end{align*}
Therefore, we formulate the pre-asymptotic representation of $\tilde{e}^{(\ell)}(\gamma,c)$ on the uniform scale directly. For $q\in (0,1]$, let
\[
\bm{U}^{(m)}(q):=\bm{U} \mid U_m \le q
\]
be the random vector with distribution on
\[
\mathcal{D}^m(q):=[0,1]^{m-1} \times[0, q] \times[0,1]^{d-m}
\]
given by
\begin{align}
    \label{eq:DFUM}
\nonumber
\pr\left(\bm{U}^{(m)}(q) \le \bm{x}\right)&=q^{-1} \pr\left(U_1 \le x_1, \ldots, U_m \le (x_m \wedge q),\ldots, U_d \le x_d\right)\\
&=q^{-1} C(x_1, \ldots, (x_m \wedge q),\ldots, x_d), \qquad \bm{x}\in [0,1]^{d}.
\end{align}
Define the clipped pre-asymptotic form moment $e_{q,ij}^{(\ell)}$ for $i,j\in V$ as
\begin{equation*}
e_{q,ij}^{(\ell)} :=\E\left[\left\{\log \left(q/U_i^{(j)}(q) \right) \vee (-c) \right\}^\ell \right]=\E\left[\left\{\log \left((q/U_i^{(j)}(q)) \vee a\right)\right\}^\ell \right].
\end{equation*}
Then we have
\begin{equation}
\label{eq:RepMomentPre}
e_{q,ij}^{(\ell)}=\int_{0}^{1/a} \frac{C_{ij}(qt,q) \ell (-\log t)^{\ell-1}}{qt} \d t +(\log a)^{\ell},
\end{equation}
see~\cref{sec:proofexpansion} for the proof.
Since $q^{-1}C_{ij}(qx,q) \le x$, it follows from the representation of $\E\left[ \left\{Y_{i}^{(j)} \vee (-c) \right\}^{\ell} \right]$ and $e_{q,ij}^{(\ell)}$, the convergence in~\eqref{eq:UTDF} and the dominated convergence theorem that, for $\ell=1,2$,
\begin{align}
\label{eq:precovg}
    e_{q,ij}^{(\ell)} \to \E\left[ \left\{Y_{i}^{(j)} \vee (-c) \right\}^{\ell} \right], \  \text{as} \  q\to 0.
\end{align}

Now we consider the empirical moments. Based on the definition of $\hat{Y}_{ti}$ in \eqref{eq:emp_Y}, we define
\begin{align}
    \notag
    \tilde{e}_{n,ij}^{(\ell)}\left(k,c\right)
    &:= \frac{1}{k} \sum_{t=1}^{n} \left\{ \hat{Y}_{ti} \vee (-c) \right\}^{\ell} \I\left(\hat{F}_{j}(X_{tj}) \ge 1-\frac{k}{n} \right)\\
\label{eq:TruncEmpMoment}
    &= \frac{1}{k} \sum_{t=1}^{n} \left\{\log\left(\frac{k}{n(1-\hat{F}_i(X_{ti}))}\vee a\right)\right\}^{\ell} \I\left(\hat{F}_{j}(X_{tj}) \ge 1-\frac{k}{n} \right).
\end{align}
Using the definition of $\hat{U}_{ti}$, it can be written as
\begin{equation*}
    \tilde{e}_{n,ij}^{(\ell)}(k,-\log a)= \frac{1}{k} \sum_{t=1}^{n} \left\{- \log \left(\frac{n}{k}\hat{U}_{tj} \wedge a^{-1}\right) \right\} ^{\ell} \I\left( \hat{U}_{ti} \le \frac{k}{n} \right).
\end{equation*}
With the definition of $\tilde{R}(\bm{x})$ in~\eqref{eq:EmpiricalTailCopula}, we obtain the expansion of $\tilde{e}_{n,ij}^{(\ell)}$ as
\begin{equation}
\label{eq:RepMomentHat}
\tilde{e}_{n,ij}^{(\ell)}=\int_{1/k}^{1/a} \frac{\tilde{R}_{n,ij}(t,1) \ell (-\log t)^{\ell-1}}{t} \d t  + (\log a)^{\ell},
\end{equation}
see~\cref{sec:proofexpansion} for the proof.

Note that
\begin{equation}
\label{eq:EmpTDF}
 \sup_{\bm{x}\in [0,n/k]^{|I|}} \bigg| \hat{R}_{n,I}(\bm{x}) - \tilde{R}_{n,I}(\bm{x})\bigg|=O(1/k)
\end{equation}
almost surely as $n\to\infty$. For any constant $T \ge 0$, $I\subset V$ and $|I|\le 3$, by~(S.47) in \citet{EV20} and the homogeneity property of tail copulas, we know that
\begin{equation}
\label{eq:UniformCovg}
\sup_{\bm{x}\in [0,T]^{|I|}} \bigg|\hat{R}_{n,I}(\bm{x})-\frac{n}{k}C_{I}(k\bm{x}/n)\bigg|=O_p(k^{-1/2}).
\end{equation}
Therefore, by setting $q=k/n$ and combining the representations of $\tilde{e}_{q,ij}^{(\ell)}$ and $\tilde{e}_{n,ij}^{(\ell)}$ with~\eqref{eq:EmpTDF}--\eqref{eq:UniformCovg}, we have
\begin{align*}
    \lefteqn{\bigg|\tilde{e}_{n,ij}^{(\ell)}-\tilde{e}_{\frac{k}{n},ij}^{(\ell)}\bigg|}\\
   &=  \biggl|\int_{1/k}^{1/a} \frac{\{\hat{R}_{n,ij}(x,1)-\frac{n}{k} C_{ij}\left(\frac{kx}{n},\frac{k}{n}\right) \} \ell (-\log x)^{\ell-1}  }{x} \d x \\
   & \hspace{30mm} \mbox{} -\int_{0}^{1/k} \frac{\frac{n}{k} C_{ij}\left(\frac{kx}{n},\frac{k}{n}\right) \ell (-\log x)^{\ell-1} }{x} \d x + O_{p}(k^{-1} (\log k)^{\ell})\biggr|\\
   &\le \sup_{x,y\in [0,1/a]^2} \biggl| \hat{R}_{n,ij}(x,y) -\frac{n}{k} C_{ij}\left(\frac{kx}{n},\frac{ky}{n}\right) \biggr| \cdot\left(\log k-\log a\right)^{\ell}  + O_{p}\left(k^{-1}(\log k)^{\ell}\right)\\
    &= O_{p}\left(k^{-1} (\log k)^{\ell}\right)=o_p(1)
    \end{align*}
as $n\to\infty$. Setting $q=k/n$ in~\eqref{eq:precovg} and combining with the above equation yields that
\begin{equation}
\label{eq:tildeeCvg}
\tilde{e}_{n,ij}^{(\ell)} \overset{\pr}\to \E\left[ \left\{Y_{i}^{(j)} \vee (-c) \right\}^{\ell} \right], \, \mbox{as} \  n\to \infty.
\end{equation}

To establish the main result, note that the clipped moment function $\hat{e}_{n,ij}^{(\ell)}$ in~\eqref{eq:Empetilde} can be expressed as continuous functions of the empirical moment functions studied above. Specifically, by the definition of $\tilde{e}_{n,ij}^{(\ell)}$ in~\eqref{eq:TruncEmpMoment} and $\hat{e}_{n,ij}^{(\ell)}$, for $i,j\in V$ ($i\neq j$) and $\ell=1,2$, we have
\begin{align}
\label{eq:TildeEmpM1}
\nonumber
\hat{e}_{n,ij}^{(1)}&=\frac{1}{k} \sum_{t=1}^{n} \left\{\left( \log \frac{k}{n(1-\hat{F}_i(X_{ti}))}-\log  a\right)_{+} \right\} \I\left(\hat{F}_{j}(X_{tj}) \ge 1-\frac{k}{n} \right)\\
\nonumber
&=\frac{1}{k} \sum_{t=1}^{n} \left\{\log\left(  \frac{k}{n(1-\hat{F}_i(X_{ti}))} \vee a\right) \right\} \I\left(\hat{F}_{j}(X_{tj}) \ge 1-\frac{k}{n} \right) - \log a \\
&= \tilde{e}_{n,ij}^{(1)} - \log a
\end{align}
and
\begin{align}
\label{eq:TildeEmpM2}
\nonumber
\hat{e}_{n,ij}^{(2)}&=\frac{1}{k} \sum_{t=1}^{n} \left\{\left( \log \frac{k}{n(1-\hat{F}_i(X_{ti}))}-\log  a\right)_{+} \right\}^{2} \I\left(\hat{F}_{j}(X_{tj}) \ge 1-\frac{k}{n} \right)\\
\nonumber
&=\frac{1}{k} \sum_{t=1}^{n} \left\{\log\left(  \frac{k}{n(1-\hat{F}_i(X_{ti}))} \vee a\right) - \log a \right\}^{2} \I\left(\hat{F}_{j}(X_{tj}) \ge 1-\frac{k}{n} \right) \\
&= \tilde{e}_{n,ij}^{(2)} - 2 (\log a) \tilde{e}_{n,ij}^{(1)} + (\log a)^2.
\end{align}
By the continuous mapping theorem, \eqref{eq:tildeeCvg} and the representations above, the convergence in \cref{pro:weakconsistency} holds.
\end{proof}

\begin{proof}[Proof of~\cref{thm:WeakConstMoment}]

From \cref{pro:weakconsistency}
we know that, as $n\to \infty$,
\[
\hat{e}_{n,ij}^{(\ell)} \overset{\pr} \to  e^{(\ell)}(\gamma_{ij}).
\]
By the symmetry of the variogram matrix, we have $\gamma_{ij}=\gamma_{ji}$. Therefore, $e^{(\ell)}(\gamma_{ij})=e^{(\ell)}(\gamma_{ji})$. Consequently,
\begin{equation}
\label{eq:AveragedMomentCov}
\frac{1}{2} \left(\hat{e}_{n,ij}^{(\ell)}+\hat{e}_{n,ji}^{(\ell)} \right) \overset {\pr} \to e^{(\ell)}(\gamma_{ij})
\end{equation}
as $n \to \infty$ for $\ell=1,2$.

Furthermore, by \cref{lemma:monotonicity}, the function $e^{(\ell)}(\gamma)$ is strictly decreasing on $(0,\infty)$, with boundary limits
\[
\lim_{\gamma\to 0} e^{(1)}(\gamma,c) =1+c, \  \  \lim_{\gamma\to 0} e^{(2)}(\gamma,c) =(1+c)^2+1
\]
and
\[
\lim_{\gamma\to \infty} e^{(\ell)}(\gamma,c) = 0.
\]
Hence, $e^{(\ell)}(\gamma)=e^{(\ell)}(\cdot,c): (0,\infty) \to \mathcal{R}_c^{(\ell
)}$ ($\ell=1,2$) is a bijection with range $\mathcal{R}_c^{(1)}=(0,1+c)$ or $\mathcal{R}_c^{(2
)}=(0,(1+c)^2+1)$. Therefore, its inverse function $e^{(\ell),\leftarrow}$ is well-defined and continuous on  $\mathcal{R}_c^{(\ell)}$. Consequently, the moment estimator $\hat{\gamma}_{n,ij}^{\mathrm{M},(\ell)}$ $(\ell=1,2)$ in~\eqref{eq:M1} is well-defined whenever
\[\frac{1}{2}\left(\hat e_{n,ij}^{(\ell)}+\hat e_{n,ji}^{(\ell)}\right) \in \mathcal{R}_c^{(\ell)},
\]
which holds with probability tending to one by~\eqref{eq:AveragedMomentCov}. This establishes the existence of $\hat{\gamma}_{n,ij}^{\mathrm{M},(\ell)}$ with probability tending to one.

Consider the case where $\gamma_{ij}=0$. Recall that
we interpret $e^{(\ell)}(0,c)$ by continuous extension
\[
e^{(1)}(0,c)=1+c, \quad e^{(2)}(0,c)=(1+c)^2+1.
\]
Since $e^{(\ell)}(\gamma,c)$ is continuous and strictly monotone on $(0,\infty)$, its inverse admits a right-continuous extension at the upper boundary points of $\mathcal{R}_c^{(\ell)}$, in the sense that
\[
e^{(1),\leftarrow}(y) \to 0 \quad \text{as } y \uparrow 1+c
\]
and
\[
e^{(2),\leftarrow}(y) \to 0 \quad \text{as } y \uparrow (1+c)^2+1.
\]
Therefore, the inverse mapping can be continuously extended to $e^{(\ell)}(0,c)$ via right-continuity by setting
\[
e^{(\ell),\leftarrow}(e^{(\ell)}(0,c))=0.
\]
Hence, the moment estimators remain well-defined in this case.

Finally, the weak consistency of $\hat{\gamma}_{n,ij}^{\mathrm{M},(\ell)}$ follows from the continuous mapping theorem applied to extended inverse mapping
$e^{(\ell),\leftarrow}:\tilde{\mathcal{R}}_c^{(\ell)} \to [0,\infty)$ and \eqref{eq:AveragedMomentCov}, where $\tilde{\mathcal{R}}_c^{(1)}=(0,1+c]$ and $\tilde{\mathcal{R}}_c^{(2)}=(0,(1+c)^2+1]$.
\end{proof}

\begin{proof}[Proof of \cref{pro:WeakCovMoment}]
Recall that
\[
\hat{\nu}_{n,ij}(x,1)=\sqrt{k} \left\{\hat{R}_{n,ij}(x,1)-R_{ij}(x,1) \right\}.
\]
For $\ell=1,2$, using the expansion of $\tilde{e}_{n,ij}^{(\ell)}=\tilde{e}_{n,ij}^{(\ell)}(k,c)$ in~\eqref{eq:RepMomentHat} and that of $\tilde{e}^{(\ell)}(\gamma)=\tilde{e}^{(\ell)}(\gamma,c)$ in~\eqref{eq:RepMomentC}, together with~\eqref{eq:EmpTDF} we have
\begin{align*}
&\sqrt{k}\left(\tilde{e}_{n,ij}^{(\ell)}-\tilde{e}^{(\ell)}(\gamma_{ij})\right)\\
&=\sqrt{k} \left\{\int_{1/k}^{1/a} \frac{\tilde{R}_{n,ij}(x,1) \ell (-\log x)^{\ell-1}}{x} \d x - \int_{0}^{1/a} \frac{R_{ij}(x,1) \ell (-\log x)^{\ell-1}}{x} \d x \right\}
\\
&=\sqrt{k} \Bigg[\int_{1/k}^{1/a} \frac{\left\{\tilde{R}_{n,ij}(x,1)-R_{ij}(x,1)\right\} \ell (-\log x)^{\ell-1}}{x} \d x \\
&\qquad \quad  - \int_{0}^{1/k} \frac{R_{ij}(x,1) \ell (-\log x)^{\ell-1} }{x} \d x \Bigg]
\\
&=\int_{1/k}^{1/a} \frac{\hat{\nu}_{n,ij}(x,1) \ell (-\log x)^{\ell-1} }{x} \d x - \sqrt{k} \int_{0}^{1/k} \frac{R_{ij}(x,1) \ell (-\log x)^{\ell-1}}{x} \d x \\
& \quad +O\left(k^{-1/2} (\log{k})^{\ell} \right)
\end{align*}
almost surely. Note that
\[
\sqrt{k} \int_{0}^{1/k} \frac{R_{ij}(x,1) \ell (-\log x)^{\ell-1} }{x} \d x
\le \sqrt{k} \int_{0}^{1/k}  \ell (-\log x)^{\ell-1}  \d x \le \frac{2}{\sqrt{k}} (\log k +1)
\]
for $\ell=1,2$ by the inequality $R_{ij}(x,1)\le x$. Under the assumption that $\gamma_{ij}>0$, the partial derivatives of the bivariate tail copula function of the \HR{} distribution are continuous. Combining \cref{lemma:UniformConverge},  \cref{lemma:ConRatePartialDer} with the Slutsky's theorem (see Theorem~2.7 in~\citet{VW96}), we obtain
\begin{align}
\label{eq:JointWeakCov}
\nonumber
&\sqrt{k} \left(\tilde{e}_{n,ij}^{(1)}-\tilde{e}^{(1)}(\gamma_{ij}), \,\tilde{e}_{n,ij}^{(2)}-\tilde{e}^{(2)}(\gamma_{ij}), \, \, i, j\in V, \,  i\neq j \right)  \\
&\overset{d}\to \left(\int_{0}^{1/a} \frac{B_{ij}(x,1)}{x} \d x, \,  -2 \int_{0}^{1/a} \frac{B_{ij}(x,1) \log x }{x} \d x , \, i, j\in V, \, i\neq j\right).
\end{align}

To obtain the asymptotic distribution of the original moment functions, we next relate $\hat{e}_{n,ij}^{(1)}(k,c)$ to $\tilde{e}_{n,ij}^{(1)}(k,c)$. By \eqref{eq:TildeEmpM1} and \eqref{eq:TildeEmpM2}, we have
\begin{align*}
\sqrt{k} \left(\hat{e}_{n,ij}^{(1)}(k,c)-e^{(1)}(\gamma_{ij})\right) = \sqrt{k} \left(\tilde{e}_{n,ij}^{(1)}(k,c)-\tilde{e}^{(1)}(\gamma_{ij}) \right)
\end{align*}
and
\begin{multline*}
\sqrt{k} \left(\hat{e}_{n,ij}^{(2)}(k,c)-e^{(2)}(\gamma_{ij})\right) \\
= \sqrt{k}  \left(\tilde{e}_{n,ij}^{(2)}(k,c)-\tilde{e}^{(2)}(\gamma_{ij})\right)   -2 \log a \cdot \sqrt{k} \left( \tilde{e}_{n,ij}^{(1)}(k,c)-\tilde{e}^{(1)}(\gamma_{ij}) \right).
\end{multline*}
Combining these identities with \eqref{eq:JointWeakCov} yields the desired weak convergence result.
\end{proof}

\begin{proof}[Proof of~\cref{thm:AsyNormalityMoment}]
By~\cref{pro:WeakCovMoment} and the relation that $c=-\log a$, we have
\begin{align}
\label{eq:AsyNormalMoment}
\nonumber
&   \left( \frac{1}{2} \left\{\sqrt{k} \left( \hat{e}^{(\ell)}_{n,ij}(k,c) - e^{(\ell)}(\gamma_{ij}) \right) + \sqrt{k} \left( \hat{e}^{(\ell)}_{n,ji}(k,c) - e^{(\ell)}(\gamma_{ji}) \right) \right\} , \, i, j\in V, \, i\neq j \right) \\
&\quad \overset{d}\to \left(\frac{1}{2} \int_{0}^{\exp(c)} \frac{\left\{ B_{ij} (x,1)+B_{ji}(x,1) \right\} \ell (-\log x +c )^{\ell-1} }{x} \d x  , \, i, j\in V, \, i\neq j \right)
\end{align}
for $\ell=1,2$ as $n \to \infty$.

Recall that  $e^{(\ell),\leftarrow}$ denotes the inverse function of $e^{(\ell)}(\gamma)$ for $\ell=1,2$. From the definition of $\hat{\gamma}_{n,ij}^{\mathrm{M},(\ell)}$ in~\eqref{eq:M1} and the fact that $\gamma_{ij}=\gamma_{ji}$, we can write
\begin{align*}
\sqrt{k}\left(\hat{\gamma}_{n,ij}^{\mathrm{M},(\ell)}-\gamma_{ij}\right)
&=e^{(\ell),\leftarrow} \left(\frac{\hat{e}_{n,ij}^{(\ell)}+\hat{e}_{n,ji}^{(\ell)}}{2} \right) - e^{(\ell),\leftarrow} \left(\frac{e^{(\ell)}(\gamma_{ij})+e^{(\ell)}(\gamma_{ji})}{2} \right)
\end{align*}
for $i,j\in V$ and $i\neq j$. Here, the inverse at $(\hat{e}_{n,ij}^{(\ell)}+\hat{e}_{n,ji}^{(\ell)})/2$ is well-defined whenever $(\hat{e}_{n,ij}^{(\ell)}+\hat{e}_{n,ji}^{(\ell)})/2 \in \mathcal{R}_c^{(\ell)}$ (see the proof of~\cref{thm:WeakConstMoment}).
Note that $e^{(\ell)}(\gamma)$ is differentiable on $(0,\infty)$, with derivative $e^{(\ell),'}(x)$ given in~\cref{thm:AsyNormalityMoment}. Hence, by the delta method and~\eqref{eq:AsyNormalMoment}, we have for $\ell=1,2$ that
\begin{align*}
&\left(\sqrt{k}\left(\hat{\gamma}_{n,ij}^{\mathrm{M},(\ell)}-\gamma_{ij}\right); \; i,j\in V, \, i\neq j \right) \\
&\quad \overset{d}\to \left(\frac{1}{2 e^{(\ell),'}(\gamma_{ij})}  \int_{0}^{\exp(c)} \frac{(B_{ij}(x,1) + B_{ji}(x,1))\cdot \ell \cdot (-\log x +c )^{\ell-1} }{x} \d x ; \; i,j\in V, \, i\neq j \right).
\end{align*}
The proof is complete.
\end{proof}

\subsection{Proofs of the representations of moments}
\label{sec:proofexpansion}

\subsubsection{Proof of~\cref{eq:RepMomentC}}

\begin{proof}
For $x\in \mathbb{R}$ and $\ell=1,2$, we have
\begin{equation}
\label{eq:PowerIntegral}
    x^{\ell}=\int_{0}^{\infty} \I\{t\le x\} \ell t^{\ell-1} \d t+ (-1)^{\ell} \int_{0}^{\infty} \I\{ t<-x\} \ell t^{\ell-1}\d t.
\end{equation}
Hence,
\begin{align}
\label{eq:LogIntegralmoment}
\nonumber
    \left\{Y_{i}^{(j)} \vee (-c) \right\}^{\ell}
    &=  \int_{0}^{\infty} \I\{t\le Y_{i}^{(j)} \vee (-c)\} \cdot \ell t^{\ell-1}\d t \\
    \nonumber
    &\quad + (-1)^{\ell} \int_{0}^{\infty} \I\{t< -(Y_{i}^{(j)} \vee (-c))\} \cdot \ell t^{\ell-1}  \d t \\
    \nonumber
    & = \int_{0}^{\infty} \I\{t\le Y_{i}^{(j)} \vee (-c)\}  \cdot \ell t^{\ell-1}  \d t
    + (-1)^\ell \int_{0}^{c} \I\{-c \le  Y_{i}^{(j)}< -t\} \cdot \ell t^{\ell-1} \d t\\
    \nonumber
    &\quad + (-1)^\ell \int_{0}^{c} \I\{Y_{i}^{(j)}< -c\} \cdot \ell t^{\ell-1} \d t \\
    \nonumber
    & =  \int_{0}^{\infty} \I\left\{t\le Y_{i}^{(j)} \right\} \cdot \ell t^{\ell-1} \d t + (-1)^\ell \int_{0}^{c} \I\{-c\le  Y_{i}^{(j)}< -t\} \cdot  \ell t^{\ell-1} \d t \\
      &\quad + (-1)^{\ell} c^{\ell} \I\left\{Y_{i}^{(j)}< -c\right\}.
\end{align}
By \cref{lemma:ExistTruVario}, Fubini's theorem and~\eqref{eq:TailDistributionSame}, we have
\begin{align*}
    \E\left[ \left\{ Y_{i}^{(j)} \vee (-c) \right\}^{\ell} \right]
    &= (-c)^{\ell} \pr\left(Y_{i}^{(j)}< -c\right) +
    (-1)^\ell   \int_{0}^{c} \pr(-c\le Y_{i}^{(j)}<-t) \cdot \ell t^{\ell-1} \d t\\
    &\quad  +  \int_{0}^{\infty} \pr\left(t\le Y_{i}^{(j)} \right) \cdot \ell t^{\ell-1}  \d t\\
    &= (-c)^{\ell} \pr\left(Y_{i}^{(j)}< -c\right) +
    (-1)^\ell   \int_{0}^{c} \pr(-c\le Y_{i}^{(j)}<-t) \cdot \ell t^{\ell-1} \d t\\
    &\quad  +  \int_{0}^{\infty} \pr\left(t\le Y_{i}^{(j)} \right) \cdot \ell t^{\ell-1}  \d t\\
    &= (-c)^{\ell} \left\{1-R_{ij}(\exp(c),1)\right\} \\
    &\quad + (-1)^\ell  \int_{0}^{c} \left\{R_{ij}(\exp(c),1)-R_{ij}(\exp(t),1)\right\} \cdot \ell t^{\ell-1}\d t  \\
    &\quad +  \int_{0}^{\infty} R_{ij}(\exp(-t),1) \cdot \ell  t^{\ell-1} \d t,
\end{align*}
where Fubini's theorem is used in the first step together with the fact in \cref{lemma:ExistTruVario} and the second equality follows from~\eqref{eq:TailDistributionSame}. Since we set $c=-\log a$. By the changes of variables $t=\log u$ and $t=-\log u$ in the first and second integrals, respectively, the expectation can be rewritten as
\begin{align*}
    \E\left[ \left\{ Y_{i}^{(j)} \vee (-c) \right\}^{\ell} \right]
    &= (-\log a)^{\ell} \left\{1-R_{ij}(1/a,1)\right\} \\
    &\quad + (-1)^\ell  \int_{1}^{1/a} \frac{\left\{R_{ij}(1/a,1)-R_{ij}(u,1)\right\} \cdot \ell (\log u)^{\ell-1} }{u} \d u  \\
    &\quad +  \int_{0}^{1} \frac{ R_{ij}(u,1) \cdot \ell  (-\log u)^{\ell-1}}{u} \d u\\
    &= \int_{0}^{1/a} \frac{R_{ij}(u,1) \cdot \ell (-\log u)^{\ell-1}}{u} \d u  + (\log a)^{\ell}.
\end{align*}

\end{proof}

\subsubsection{Proof of~\cref{eq:RepMomentPre}}

\begin{proof}
Similarly to~\eqref{eq:PowerIntegral}, for $x>0$ and $\ell=1,2$, we have
\begin{equation}
\label{eq:LogIntegral}
    (\log x)^{\ell}=\int_{1}^{\infty} \frac{\I\{t\le x\} \ell (\log t)^{\ell-1}}{t} \d t - \int_{0}^{1} \frac{\I\{t> x\} \ell (\log t)^{\ell-1}}{t} \d t.
\end{equation}
Therefore,
\begin{align*}
    &\left\{\log \left(q/U_i^{(j)}(q) \vee a\right) \right\}^{\ell}\\
    &\quad =  \int_{1}^{\infty} \frac{\I\{t\le [q/U_i^{(j)}(q) \vee a] \} \ell (\log t)^{\ell-1}}{t} \d t - \int_{0}^{1} \frac{\I\{t> [q/U_i^{(j)}(q) \vee a]\} \ell (\log t)^{\ell-1}}{t} \d t \\
    &\quad = \int_{1}^{\infty} \frac{\I\{t\le q/U_i^{(j)}(q) \} \ell (\log t)^{\ell-1} }{t} \d t -\int_{a}^{1} \frac{\I\{a\le  q/U_i^{(j)}(q)< t\} \ell (\log t)^{\ell-1} }{t} \d t\\
    &\qquad- \int_{a}^{1} \frac{\I\left\{q/U_i^{(j)}(q)< a\right\} \ell (\log t)^{\ell-1} }{t} \d t \\
    &\quad =  \int_{1}^{\infty} \frac{\I\left\{t\le q/U_i^{(j)}(q) \right\} \ell (\log t)^{\ell-1} }{t} \d t - \int_{a}^{1} \frac{\I\{a\le  q/U_i^{(j)}(q)< t\} \ell (\log t)^{\ell-1} }{t} \d t \\
    &\qquad  + (\log a)^{\ell} \I\left\{q/U_i^{(j)}(q)< a\right\}.
\end{align*}
Therefore, by Fubini's theorem and~\eqref{eq:DFUM}, we have
\begin{align*}
    &\E\left[\left\{\log \left(q/U_i^{(j)}(q) \vee a\right) \right\}^{\ell} \right]\\
    &\quad= (\log a)^{\ell} \pr\left(q/U_i^{(j)}(q)< a\right) -  \int_{a}^{1} \frac{\pr(a\le q/U_i^{(j)}(q)<t) \ell (\log t)^{\ell-1} }{t} \d t  \\
    &\qquad +   \int_{1}^{\infty} \frac{\pr\left(t\le q/U_i^{(j)}(q) \right) \ell (\log t)^{\ell-1} }{t} \d t\\
    &\quad= (\log a)^{\ell} \left\{1-q^{-1}C_{ij}(q a^{-1},q)\right\} -  \int_{1}^{1/a} \frac{[C_{ij}(q a^{-1},q)-C_{ij}(qt,q)] \ell (-\log t)^{\ell-1} }{qt} \d t  \\
    &\qquad +  \int_{0}^{1} \frac{C_{ij}(qt,q) \ell (-\log t)^{\ell-1}}{qt} \d t\\
    &\quad = \int_{0}^{1/a} \frac{C_{ij}(qt,q) \ell (-\log t)^{\ell-1}}{qt} \d t +(\log a)^{\ell},
\end{align*}
where Fubini's theorem is used in the first step and the second equality follows from~\eqref{eq:DFUM}.
\end{proof}

\subsubsection{Proof of~\cref{eq:RepMomentHat}}
\begin{proof}
Recall that $0<a\le 1$. By~\eqref{eq:LogIntegral} (with the equal sign in the indicator function moved from the first integrand to the second), we have for $i,j\in V$, $i\neq j$ and $\ell=1,2$ that
\begin{align*}
    &\left\{-\log \left(\frac{n}{k}\hat{U}_{ti} \wedge a^{-1}\right) \right\}^{\ell}\\
    &\quad =  -\int_{1}^{\infty} \frac{\I\{t< \left(\frac{n}{k}\hat{U}_{ti} \wedge a^{-1}\right)\} \ell (-\log t)^{\ell-1} }{t} \d t + \int_{0}^{1} \frac{\I\{t\ge \left(\frac{n}{k}\hat{U}_{ti} \wedge a^{-1}\right)\} \ell (-\log t)^{\ell-1} }{t} \d t \\
    &\quad = -\int_{1}^{1/a} \frac{\I\{a^{-1}<  \frac{n}{k}\hat{U}_{ti} \} \ell (-\log t)^{\ell-1} }{t} \d t- \int_{1}^{1/a} \frac{\I\{t< \frac{n}{k}\hat{U}_{ti} \le a^{-1} \} \ell (-\log t)^{\ell-1} }{t} \d t \\
   &\qquad + \int_{1/k}^{1} \frac{\I\left\{t\ge \frac{n}{k}\hat{U}_{ti} \right\} \ell (-\log t)^{\ell-1}}{t} \d t \\
    &\quad =  (\log a)^{\ell} \I\left\{\hat{U}_{ti} > \frac{k}{n} a^{-1}\right\}  - \int_{1}^{1/a} \frac{\I\left\{\frac{k}{n} t< \hat{U}_{ti}\le  \frac{k}{n} a^{-1} \right\} \ell (-\log t)^{\ell-1} }{t} \d t\\
    &\qquad + \int_{1/k}^{1} \frac{\I\{  \hat{U}_{ti} \le \frac{k}{n} t\} \ell (-\log t)^{\ell-1} }{t} \d t.
\end{align*}
Hence, using the equations
\[
\I\left\{\frac{k}{n} t< \hat{U}_{ti}\le  \frac{k}{n} a^{-1},\hat{U}_{tj} \le k/n \right\}
= \I\left\{\hat{U}_{ti} \le \frac{k}{n} a^{-1},\hat{U}_{tj} \le k/n\right\}
- \I\left\{\hat{U}_{ti} \le \frac{k}{n} t,\hat{U}_{tj} \le k/n\right\}
\]
and
\[
\I\left\{\hat{U}_{ti} > \frac{k}{n} a^{-1},\hat{U}_{tj} \le k/n\right\} = \I\left\{\hat{U}_{tj} \le k/n\right\} - \I\left\{\hat{U}_{ti} \le \frac{k}{n} a^{-1},\hat{U}_{tj} \le k/n\right\},
\]
we get
\begin{align*}
 \hat{e}_{n,ij}^{(\ell)} &=\frac{1}{k}\sum_{t=1}^{n}\left\{-\log \left(\frac{n}{k}\hat{U}_{ti} \wedge a^{-1}\right)\right\}^{\ell} \I\left\{\hat{U}_{tj} \le k/n\right\}\\
 &=  \frac{1}{k}\sum_{t=1}^{n} \left\{(\log a)^{\ell} \I\left\{\hat{U}_{ti} > \frac{k}{n} a^{-1},\hat{U}_{tj} \le k/n\right\} \right. \\
  &\quad - \int_{1}^{1/a} \frac{\I\left\{\frac{k}{n} t< \hat{U}_{ti}\le  \frac{k}{n} a^{-1},\hat{U}_{tj} \le k/n \right\} \ell (-\log t)^{\ell-1} }{t} \d t \\
 &\quad  + \left.\int_{1/k}^{1} \frac{\I\{  \hat{U}_{ti} \le \frac{k}{n} t,\hat{U}_{tj} \le k/n\} \ell (-\log t)^{\ell-1}}{t} \d t \right\}\\
 &=\int_{1/k}^{1/a} \frac{\tilde{R}_{n,ij}(t,1) \ell (-\log t)^{\ell-1}}{t} \d t  + (\log a)^{\ell}.
\end{align*}
The proof is complete.
\end{proof}

\subsection{Explicit expression of the asymptotic variance}
\label{sec:AsyVar}
In this section, we derive the explicit expression of the asymptotic variance of the limiting distribution in \cref{thm:AsyNormalityMoment}.

For a \HR{} MGP distributed random vector $\bm{Y}$ with variogram matrix $\bm{\Gamma}=(\gamma_{ij})_{i,j \in V}$, the bivariate tail copula is given in \eqref{eq:bivariateTailCopulaHR}, while its partial derivatives with respect to $x$ and $y$ are  specified by \eqref{eq:firstPartialDerivative}. We redefine the notations as
\begin{align*}
&\dot{R}_1(x,y;\gamma_{ij})=\dot{R}_{ij}^{i}(x,y), & \dot{R}_{2}(x,y;\gamma_{ij})=\dot{R}_{ij}^{j}(x,y).
\end{align*}

For each pair $(i,j)\in V^2$ with $i \ne j$, the limiting random variable in \cref{thm:AsyNormalityMoment} is a centered Gaussian random variable with variance $\sigma_\ell^2(\gamma_{ij},c)$, which has expression
\begin{equation*}
    \sigma_\ell^2(\gamma,c)
    = \int_0^{\exp(c)} \int_0^{\exp(c)} V(x,y;\gamma) \cdot f_\ell(x;\gamma,c) \cdot f_\ell(y;\gamma,c) \, \mathrm{d}x \, \mathrm{d}y,
\end{equation*}
with
\begin{equation*}
f_\ell(x;\gamma,c) = \frac{\ell \cdot (-\log x+c)^{\ell-1}}{e^{(\ell),'}(\gamma) \cdot x}, \qquad x>0,
\end{equation*}
and
\[
 V(x,y;\gamma_{ij})=\E\left( \frac{B_{ij}(x,1)+B_{ji}(x,1)}{2} \cdot \frac{B_{ij}(y,1)+B_{ji}(y,1)}{2} \right).
\]
Note that by the definition of the process $B_{ij}(x,1)$, we have
\[
B_{ij}(x,1) \overset{d}= B_{ji}(x,1)
\]
for each pair of $(i,j)\in V\times V$ and $i\neq j$. By \cref{eq:covW}, \cref{eq:ProcessBij} and the relation that $\dot{R}_{1}(x,y;\gamma)=\dot{R}_{2}(y,x;\gamma)$, a direct calculation gives that
\begin{align*}
V(x,y;\gamma)
&= \frac{1}{2} \Bigg\{\left(1-\dot{R}_{1}(x,1;\gamma)-\dot{R}_{1}(y,1;\gamma)\right) R(x\wedge y, 1;\gamma) \\
&\qquad - \left( 1- \dot{R}_{1}(y,1;\gamma)\right) \dot{R}_{1}(1,x;\gamma) R(y,1;\gamma) \\
& \qquad- \left( 1- \dot{R}_{1}(x,1;\gamma)\right) \dot{R}_{1}(1,y;\gamma) R(x,1;\gamma)
 + (x \wedge y)  \dot{R}_{1}(x,1;\gamma) \dot{R}_{1}(y,1;\gamma)
 \\
&\qquad +  \dot{R}_{1}(1,x;\gamma) \dot{R}_{1}(1,y;\gamma) \Bigg\} \\
&\qquad + \frac{1}{2} \Bigg\{ R(x\wedge1,y\wedge 1; \gamma) - (R_1(x,1;\gamma)+R_1(1,y;\gamma)) R(x\wedge 1, y;\gamma) \\
&\qquad - (R_1(1,x;\gamma)+R_1(y,1;\gamma) ) R(x,y\wedge 1;\gamma) + R_1(x,1;\gamma) R_1(y,1;\gamma) R(x,y;\gamma) \\
&\qquad +R_1(1,x;\gamma) R_1(y,1;\gamma) (y\wedge 1)
+ R_1(x,1;\gamma) R_1(1,y;\gamma) (x\wedge 1)\\
&\qquad + R_1(1,x;\gamma) R_1(1,y;\gamma) R(1,1;\gamma) \Bigg\}.
\end{align*}

\section{Additional simulation results}
\label{sec:gamma2}

The randomly generated $10$-dimensional variogram matrix $\bm{\Gamma}_1$ used in the simulation study in \cref{sec:simstudy} is
\begin{align*}
\bm{\Gamma}_1=
\begin{pmatrix}
0.00 & 0.89 & 0.82 & 0.61 & 0.98 & 1.05 & 0.25 & 0.69 & 1.08 & 0.75 \\
  0.89 & 0.00 & 1.91 & 1.40 & 2.08 & 2.16 & 1.07 & 1.50 & 2.05 & 1.64 \\
  0.82 & 1.91 & 0.00 & 1.51 & 2.07 & 2.12 & 1.07 & 1.71 & 2.00 & 1.58 \\
  0.61 & 1.40 & 1.51 & 0.00 & 1.65 & 1.73 & 0.85 & 1.32 & 1.72 & 1.37 \\
  0.98 & 2.08 & 2.07 & 1.65 & 0.00 & 2.22 & 1.23 & 1.84 & 2.13 & 1.74 \\
  1.05 & 2.16 & 2.12 & 1.73 & 2.22 & 0.00 & 1.31 & 1.90 & 2.20 & 1.81 \\
  0.25 & 1.07 & 1.07 & 0.85 & 1.23 & 1.31 & 0.00 & 0.93 & 1.33 & 1.00 \\
  0.69 & 1.50 & 1.71 & 1.32 & 1.84 & 1.90 & 0.93 & 0.00 & 1.83 & 1.45 \\
  1.08 & 2.05 & 2.00 & 1.72 & 2.13 & 2.20 & 1.33 & 1.83 & 0.00 & 1.83 \\
  0.75 & 1.64 & 1.58 & 1.37 & 1.74 & 1.81 & 1.00 & 1.45 & 1.83 & 0.00 \\
\end{pmatrix}.
\end{align*}

In \crefrange{fig:Cmp10DBias}{fig:Cmp10DMSE}, we report the mean bias, variance, and MSE of the first $10$ entries of $\bm{\Gamma}_1$ based on the empirical variogram estimator $\hat{\gamma}_{n,ij}^{\mathrm{EMP}}(k)$ and the moment estimators $\hat{\gamma}_{n,ij}^{\mathrm{M},(\ell)}(k,-\log a)$ with $\ell=1,2$ and $a=0.25$. The samples are generated from the 10-dimensional \HR{} distribution with variogram matrix $\bm{\Gamma}_1$. These results are provided here for completeness and to facilitate a more detailed inspection of the component-wise estimation performance.

In addition, the related analysis for samples generated from a $5$-dimensional \HR{} distribution with parameter matrix $\bm{\Gamma}_2$ is presented in \crefrange{fig:Cmp5DBias}{fig:Cmp5DMSE}.

\begin{figure}
\centering
\includegraphics[width=0.95\textwidth]{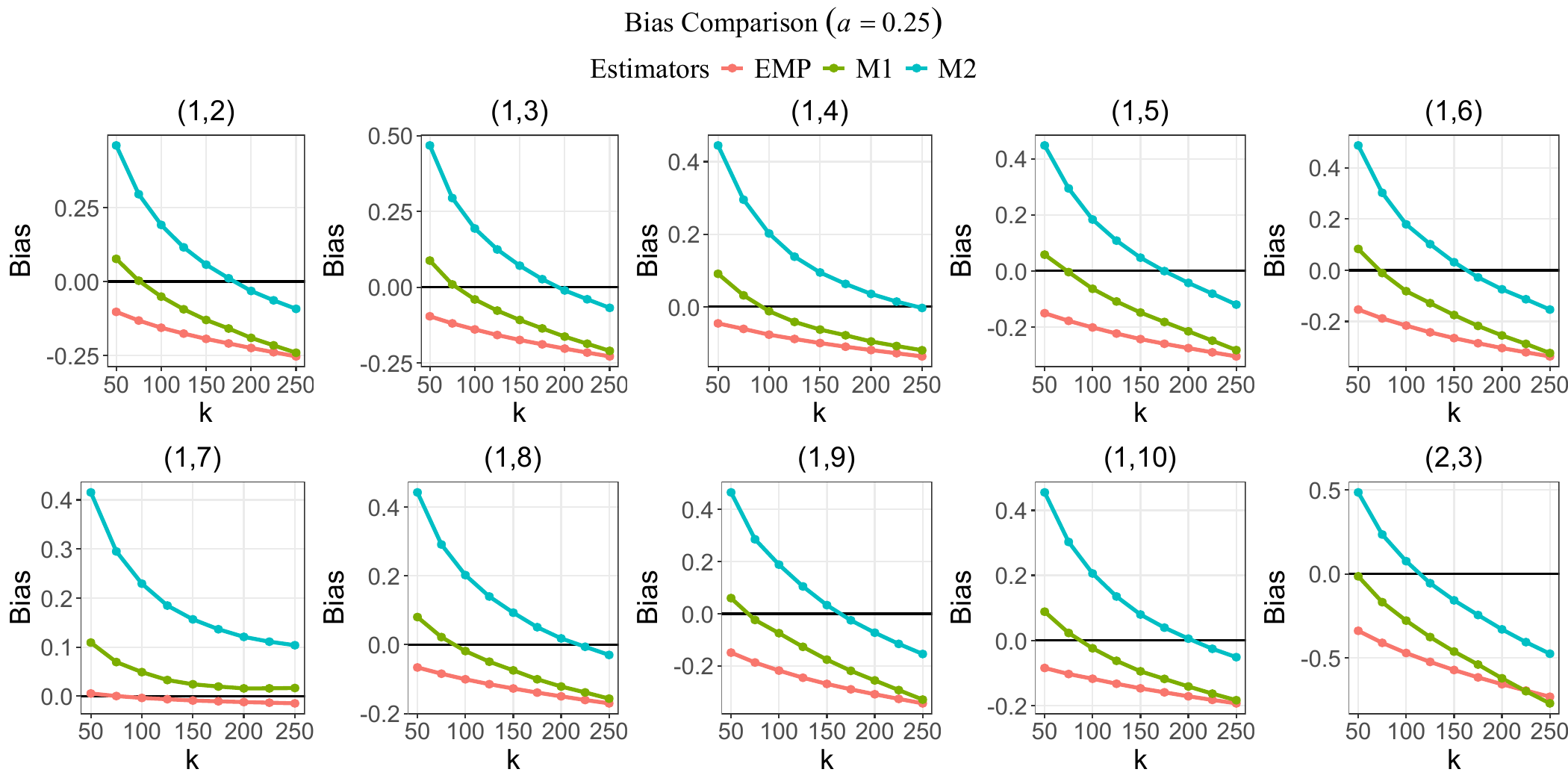}
\caption{The mean estimation bias for the first ten elements $\gamma_{ij}$ of $\bm{\Gamma}_1$, based on the empirical variogram estimator $\hat{\gamma}_{n,ij}^{\EMP}$ (EMP), first-order moment estimator $\hat{\gamma}_{n,ij}^{\mathrm{M},(1)}(k,-\log a)$ (M1) and second-order moment estimator $\hat{\gamma}_{n,ij}^{\mathrm{M},(2)}(k,-\log a)$ (M2) with $a=0.25$ in 300 replications. The random samples with size $n=1000$ are drawn from the $10-$dimensional \HR{} distribution with variogram matrix $\bm{\Gamma}_1$.}
\label{fig:Cmp10DBias}
\end{figure}

\begin{figure}
\centering
\includegraphics[width=0.95\textwidth]{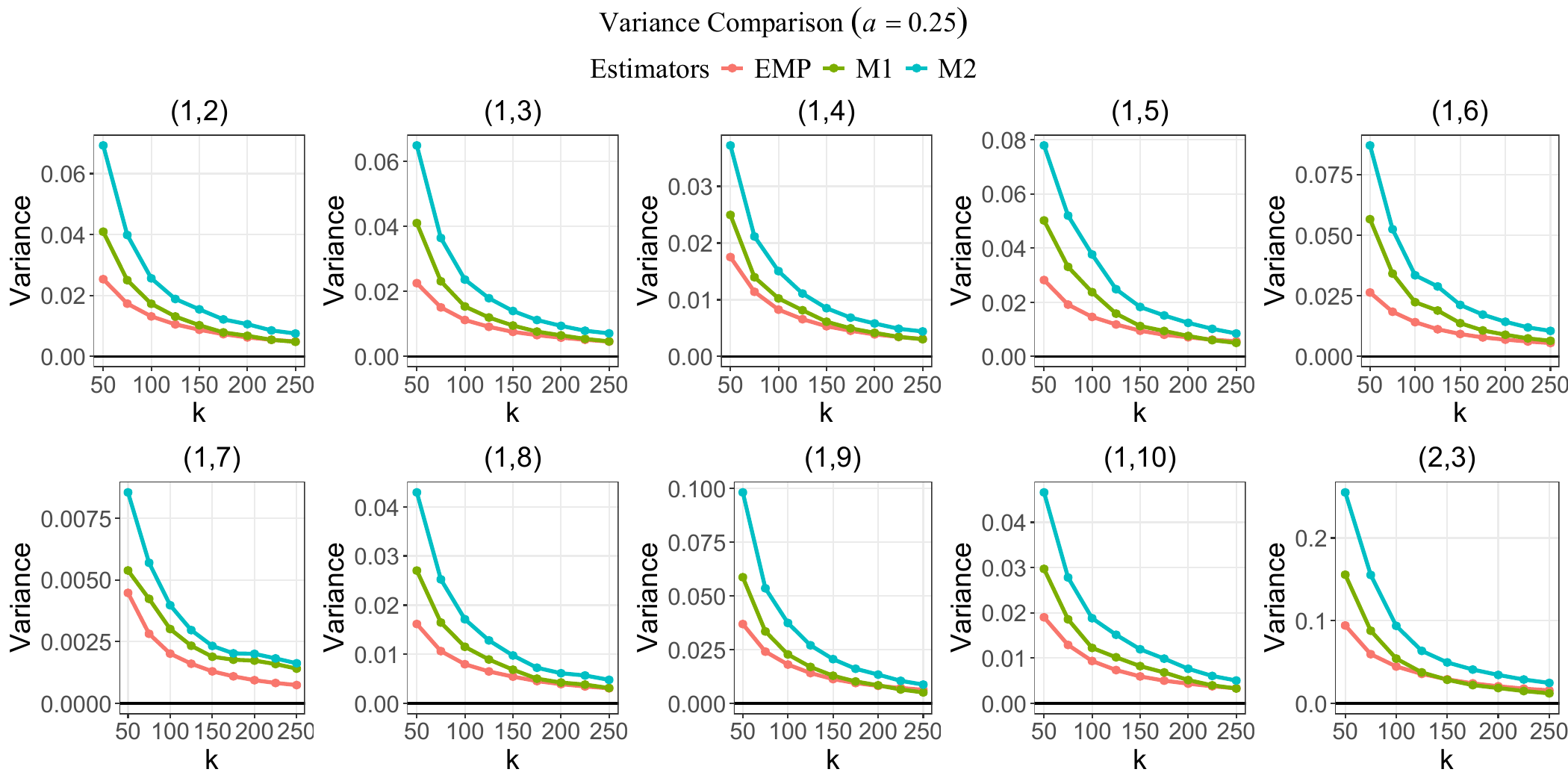}
\caption{The mean estimation variance for the first ten elements $\gamma_{ij}$ of $\bm{\Gamma}_1$, based on the empirical variogram estimator $\hat{\gamma}_{n,ij}^{\EMP}$ (EMP), first-order moment estimator $\hat{\gamma}_{n,ij}^{\mathrm{M},(1)}(k,-\log a)$ (M1) and second-order moment estimator $\hat{\gamma}_{n,ij}^{\mathrm{M},(2)}(k,-\log a)$ (M2) with $a=0.25$ in 300 replications. The random samples with size $n=1000$ are drawn from the $10-$dimensional \HR{} distribution with variogram matrix $\bm{\Gamma}_1$.}
\label{fig:Cmp10DVariance}
\end{figure}

\begin{figure}
\centering
\includegraphics[width=0.95\textwidth]{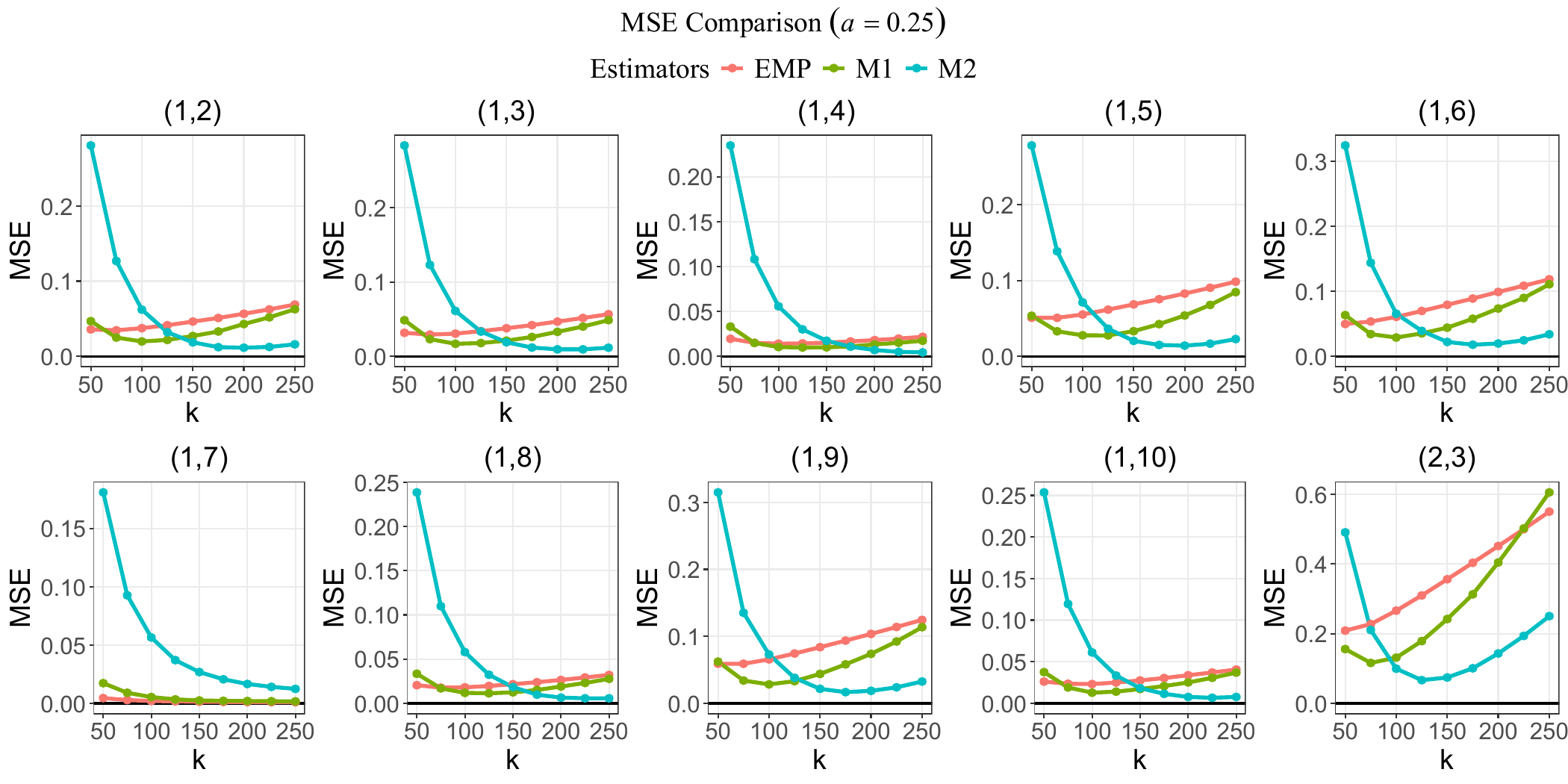}
\caption{The MSE for the first ten elements $\gamma_{ij}$ of $\bm{\Gamma}_1$, based on the empirical variogram estimator $\hat{\gamma}_{n,ij}^{\EMP}$ (EMP), first-order moment estimator $\hat{\gamma}_{n,ij}^{\mathrm{M},(1)}(k,-\log a)$ (M1) and second-order moment estimator $\hat{\gamma}_{n,ij}^{\mathrm{M},(2)}(k,-\log a)$ (M2) with $a=0.25$ in 300 replications. The random samples with size $n=1000$ are drawn from the $10-$dimensional \HR{} distribution with variogram matrix $\bm{\Gamma}_1$.}
\label{fig:Cmp10DMSE}
\end{figure}

\begin{figure}
\centering
\includegraphics[width=0.95\textwidth]{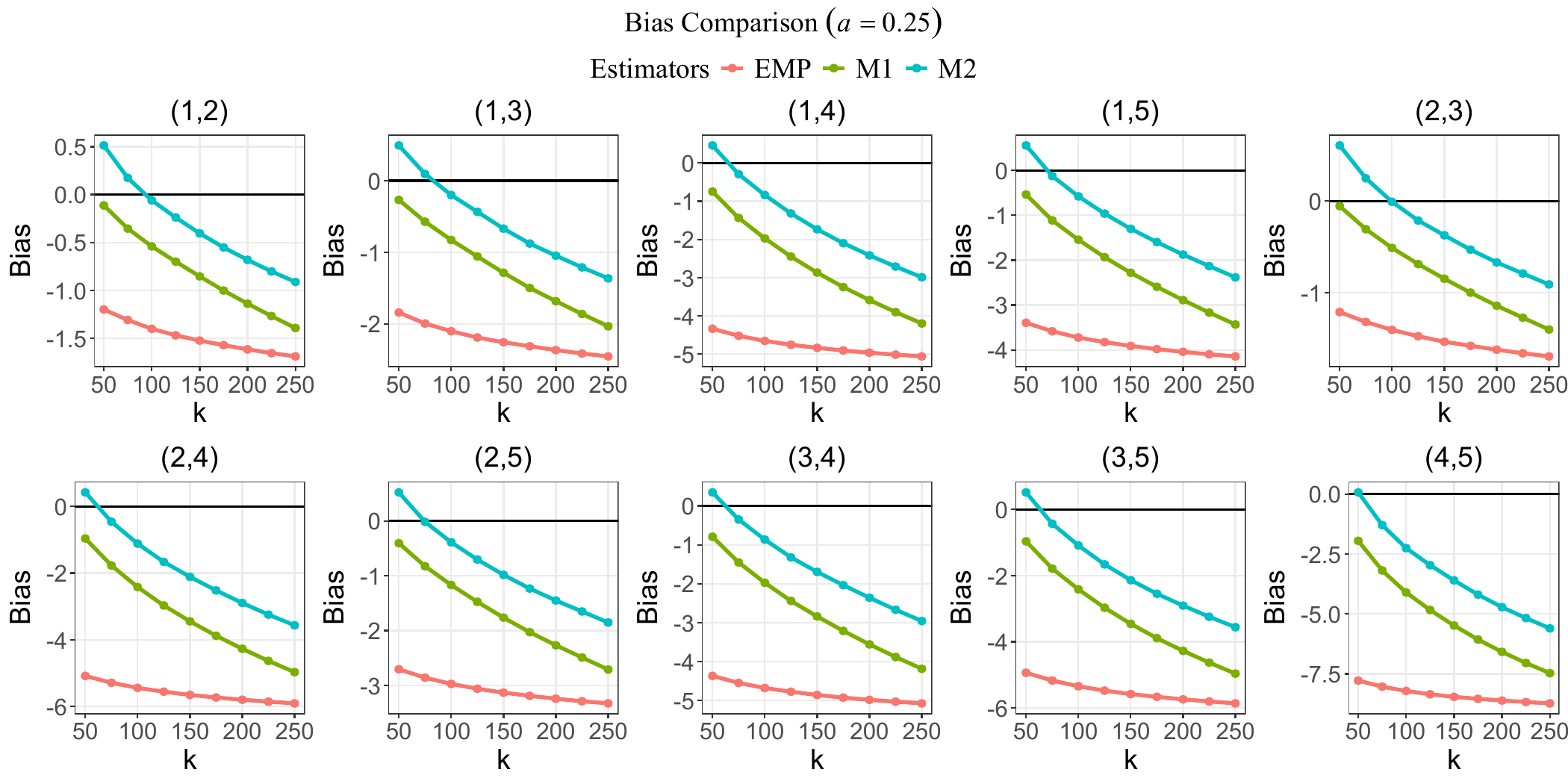}
\caption{The mean bias for $\gamma_{ij}$ of $\bm \Gamma_2$ with $i,j\in V$ and $i\neq j$, based on the empirical variogram estimator $\hat{\gamma}_{n,ij}^{\EMP}$ (EMP), first-order moment estimator $\hat{\gamma}_{n,ij}^{\mathrm{M},(1)}(k,-\log a)$ (M1) and second-order moment estimator $\hat{\gamma}_{n,ij}^{\mathrm{M},(2)}(k,-\log a)$ (M2) with $a=0.25$ in 300 replications. The random samples with size $n=1000$ are drawn from the $5-$dimensional \HR{} distribution with variogram matrix $\bm{\Gamma}_2$.}
\label{fig:Cmp5DBias}
\end{figure}

\begin{figure}
\centering
\includegraphics[width=0.95\textwidth]{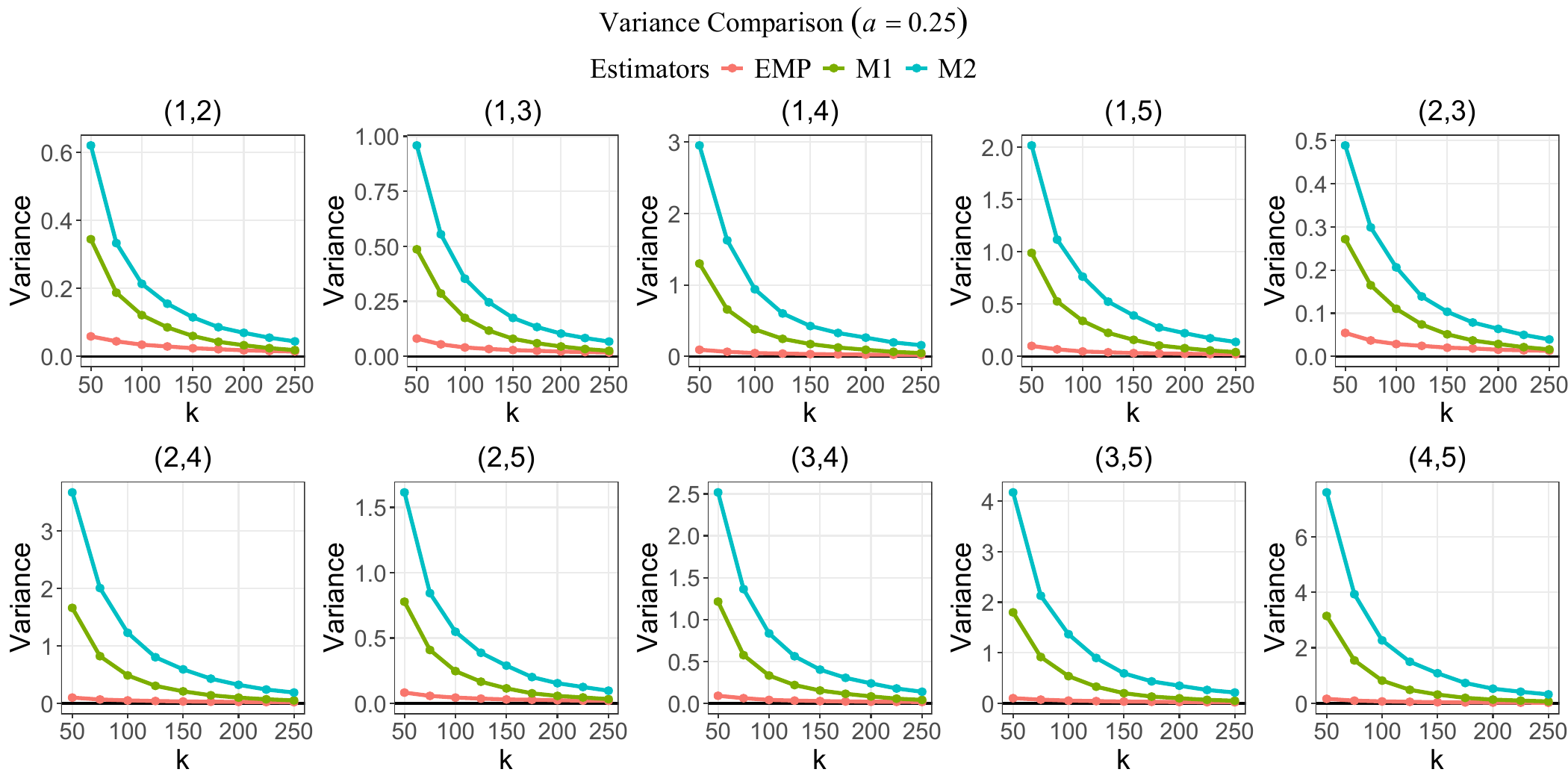}
\caption{The mean estimation variance for $\gamma_{ij}$ of $\bm \Gamma_2$ with $i,j\in V$ and $i\neq j$, based on the empirical variogram estimator $\hat{\gamma}_{n,ij}^{\EMP}$ (EMP), first-order moment estimator $\hat{\gamma}_{n,ij}^{\mathrm{M},(1)}(k,-\log a)$ (M1) and second-order moment estimator $\hat{\gamma}_{n,ij}^{\mathrm{M},(2)}(k,-\log a)$ (M2) with $a=0.25$ in 300 replications. The random samples with size $n=1000$ are drawn from the $5-$dimensional \HR{} distribution with variogram matrix $\bm{\Gamma}_2$.}
\label{fig:Cmp5DVariance}
\end{figure}

\begin{figure}
\centering
\includegraphics[width=0.95\textwidth]{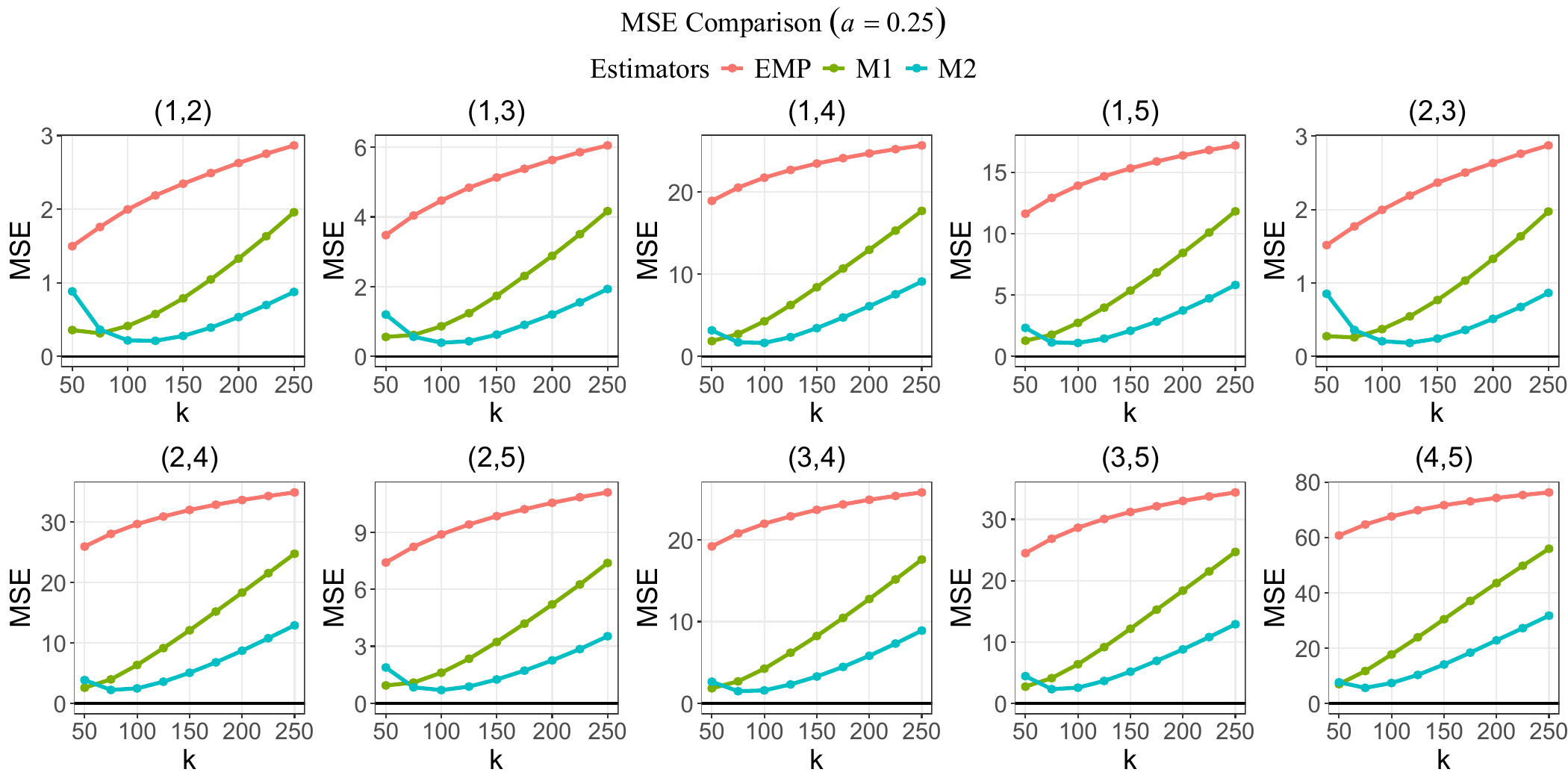}
\caption{The MSE for $\gamma_{ij}$ of $\bm \Gamma_2$ with $i,j\in V$ and $i\neq j$, based on the empirical variogram estimator $\hat{\gamma}_{n,ij}^{\EMP}$ (EMP), first-order moment estimator $\hat{\gamma}_{n,ij}^{\mathrm{M},(1)}(k,-\log a)$ (M1) and second-order moment estimator $\hat{\gamma}_{n,ij}^{\mathrm{M},(2)}(k,-\log a)$ (M2) with $a=0.25$ in 300 replications. The random samples with size $n=1000$ are drawn from the $5-$dimensional \HR{} distribution with variogram matrix $\bm{\Gamma}_2$.}
\label{fig:Cmp5DMSE}
\end{figure}

\end{document}